\documentclass[11pt,reqno]{amsart}%
\usepackage{graphicx,adjustbox}
\usepackage{amsmath}
\usepackage{amsfonts}
\usepackage{amssymb}
\usepackage{mathtools}
\usepackage{enumerate}
\usepackage{color}
\usepackage{cite}
\usepackage{url}
\usepackage{enumitem}
\usepackage[margin=1in]{geometry}%
\usepackage{tikz-cd}
\usepackage{tikz}
\usetikzlibrary{calc,positioning}
\usepackage{xcolor}
\usetikzlibrary{shapes,arrows}
\usetikzlibrary{decorations.pathmorphing, decorations.text}

\usepackage{hyperref}
\usepackage[normalem]{ulem}
\usepackage{graphicx}
\usepackage{mleftright}
\usepackage{caption}
\usepackage{stmaryrd}
\usepackage{algorithm}
\usepackage{algpseudocode}
\usepackage{tabu}
\makeatletter
\newcommand{\algrule}[1][0.4pt]{\par\vskip.5\baselineskip\hrule height #1\par\vskip.5\baselineskip}
\makeatother

\providecommand{\U}[1]{\protect\rule{.1in}{.1in}}

\newtheorem{theorem}{Theorem}[section]
\newtheorem{proposition}[theorem]{Proposition}
\newtheorem{lemma}[theorem]{Lemma}
\newtheorem{corollary}[theorem]{Corollary}
\newtheorem{definition}[theorem]{Definition}

\theoremstyle{remark}

\let\O\undefined

\DeclareMathOperator{\O}{O}

\DeclareMathOperator{\diag}{diag}
\DeclareMathOperator{\rank}{rank}

\DeclareMathOperator{\sgn}{sgn}

\DeclareMathOperator{\Aut}{Aut}

\DeclareMathOperator{\im}{im}

\DeclareMathOperator{\V}{V}
\DeclareMathOperator{\B}{B}
\DeclareMathOperator{\OB}{OB}

\def\A{\mathcal A}
\def\SS{\operatorname{S}}

\def\Tt{\operatorname{T}}

\def\diag{\operatorname{diag}}
\def\Diag{\operatorname{Diag}}
\def\O{\operatorname{O}}

\def\x{\mathbf x}
\def\y{\mathbf y}
\def\z{\mathbf z}

\def\ee{\mathbf e}

\def\uu{\mathbf u}
\def\vv{\mathbf v}
\def\w{\mathbf w}

\def\A{\mathcal A}

\def\T{\mathcal T}

\newcommand{\tp}{{\scriptscriptstyle\mathsf{T}}}

\usepackage{geometry}
\tikzset{
  symbol/.style={
    draw=none,
    every to/.append style={
      edge node={node [sloped, allow upside down, auto=false]{$#1$}}}
  }
}

\tikzstyle{startstop} = [rectangle, rounded corners, minimum width = 2cm, minimum height=1cm,text centered, draw = black, fill = red!40]
\tikzstyle{process} = [rectangle, minimum width=1cm, minimum height=1cm, text centered, draw=black, fill = yellow!50]
\tikzstyle{processb} = [rectangle, minimum width=1cm, minimum height=1cm, text centered, draw=black, fill = blue!50]
\tikzstyle{arrow} = [->,>=stealth]
\tikzstyle{dotsnd}=[rectangle, minimum width=1cm, minimum height=1cm, text centered]

\tikzstyle{processd} = [dashed, minimum width=8cm, minimum height=1.5cm, text centered, draw=black]
\tikzstyle{processdd} = [dashed, minimum width=7cm, minimum height=1.5cm, text centered, draw=black]

\begin{document}
\title[Low rank orthogonal tensor approximation]{When geometry meets optimization theory: partially orthogonal tensors}


\author{Ke Ye}
\address{KLMM, Academy of Mathematics and Systems Science, Chinese Academy of Sciences, Beijing 100190, China.}
\email{keyk@amss.ac.cn}

\author{Shenglong Hu}
\address{Department of Mathematics, School of Science, Hangzhou Dianzi University, Hangzhou 310018, China.}
\email{shenglonghu@hdu.edu.cn}

\begin{abstract}
Due to the multi-linearity of tensors, most algorithms for tensor optimization problems are designed based on the block coordinate descent method. Such algorithms are widely employed by practitioners for their implementability and effectiveness. However, these algorithms usually suffer from the lack of theoretical guarantee of global convergence and analysis of convergence rate. In this paper, we propose a block coordinate descent type algorithm for the low rank partially orthogonal tensor approximation problem and analyse its convergence behaviour. To achieve this, we carefully investigate the variety of low rank partially orthogonal tensors and its geometric properties related to the parameter space, which enable us to locate KKT points of the concerned optimization problem. With the aid of these geometric properties, we prove without any assumption that: (1) Our algorithm converges globally to a KKT point; (2) For any given tensor, the algorithm exhibits an overall sublinear convergence with an explicit rate which is sharper than the usual $O(1/k)$ for first order methods in nonconvex optimization; {(3)} For a generic tensor, our algorithm converges $R$-linearly.
\end{abstract}

\maketitle

\tableofcontents
\newpage
\section{Introduction}\label{sec:intro}
Tensors are ubiquitous objects in a large variety of applications. A basic mathematical model in these applications is approximating a given tensor by another of lower rank \cite{L-21,C-14,L-13}. Different from matrices, there exist various notions of ranks for tensors, including canonical polyadic (CP) rank \cite{H-27}, Tucker rank \cite{T-66}, hierachical rank \cite{GH11}, Vandermonde rank \cite{TS01}, symmetric rank \cite{C91}, tensor train rank \cite{O11} and tensor network rank \cite{YL18}. Among these different ranks, the CP rank \cite{L-12,C-14,L-21,L-13,H-27} is probably the most extensively discussed one because of its flexibility and generality. In this paper, we adopt the convention that rank refers to CP rank, unless otherwise stated.

The low rank approximation problem is notoriously known for its ill-posedness \cite{DL-08} and NP-hardness \cite{HL-13}. However, tensors involved in real-world applications always have additional structures, such as orthogonality \cite{ZG-01}, nonnegativity \cite{Lim-Com:non,Qi-Com-Lim:non,QCL-16}, Hankel structure \cite{NY-19}, etc. It is often the case that when we restrict the low rank approximation problem to these structural tensors, the problem instantly becomes well-posed and even tractable \cite{Lim-Com:non,HY-19,GC-19,Y-19}. In this paper, we focus on a widely existing special structure: orthogonality, which is first considered in \cite{F-92} and \cite{ZG-01}. On the one hand, orthogonal structure plays an important role in the theoretical study of tensors. For instance, the celebrated Eckart-Young-Mirsky theorem can be generalized to such tensors \cite{DDV-00,CS-09,LS-98}. The set of tensors with orthogonality constraints is also considered in the context of algebraic geometry \cite{R-16,BDHR17}.
On the other hand, tensors with orthogonality structures serve as natural mathematical models in numerous applied fields such as independent component analysis \cite{C-94}, DS-CDMA systems \cite{SDD-10}, image processing \cite{SL-01}, latent variable identification \cite{AGHKT-14}, joint singular value decomposition \cite{PPP-01}. Essentially, the  aforementioned applications are all concerned with the following optimization problem: For a given tensor $\mathcal{A}\in\mathbb{R}^{n_1}\otimes\dots\otimes\mathbb{R}^{n_k}$, find a best rank $r$ partially orthogonal approximation of $\mathcal{A}$. The problem can be formulated as:
\begin{equation}\label{eq:best-app}
\begin{array}{rl}
\min &\|\mathcal{A}-\sum_{j=1}^r\lambda_j\mathbf{x}^{(1)}_j\otimes\dots\otimes\mathbf{x}^{(k)}_j\|^2\\
\text{s.t.}&\|\mathbf{x}^{(i)}_j\|=1\ \text{for all }j=1,\dots,r,\ i=1,\dots,k.
\end{array}
\end{equation}
Here $1\le s \le k$ is a fixed positive integer and for each $1\le i \le s$, the $i$-th factor matrix $A^{(i)}:=[\mathbf{x}^{(i)}_1,\dots,\mathbf{x}^{(i)}_r]$ is orthonormal.
It is notable that in the literature, variants of \eqref{eq:best-app} are extensively studied as well. Nonetheless, we only focus on the analysis of \eqref{eq:best-app} in this paper and we refer readers who are interested in its variants to \cite{MV-08,IAV-13,MHG-15,LUC-18} and references therein for more details.

We observe that when $k = 2$ in \eqref{eq:best-app}, the problem is simply the low rank approximation problem for matrices, which can be solved by singular value decompositions of matrices. Hence it is tempting to expect that problem \eqref{eq:best-app} can be solved in polynomial time for $k \ge 3$ as well. Unfortunately, determining a solution to \eqref{eq:best-app} is NP-hard in general, according to \cite{HL-13}. Because of the NP-hardness of problem \eqref{eq:best-app}, solving it numerically is probably the inevitable choice for practitioners. To that end, block coordinate descent (BCD) is the most commonly used method, due to the inborn nature of tensors. As the objective function in \eqref{eq:best-app} is a squared distance, the BCD method is also called \textit{alternating least square (ALS)} method in the context of tensor decompositions. Various ALS type algorithms are proposed to solve problem \eqref{eq:best-app}. See, for example, \cite{CS-09,SDCJD-12,GV-13,CLD-14,WTY-15} and references therein.

When $r=1$, problem \eqref{eq:best-app} reduces to the best rank one approximation problem, which is studied in depth over the past two decades \cite{Kof-Reg:bes,ZG-01,DDV-01}. It is proved in \cite{L-05} that any limiting point of an iterative sequence generated by the ALS method is a singular vector tuple of the tensor. The global convergence of an ALS type algorithm is established in \cite{WC-14} for a generic tensor and later the genericity assumption is removed by \cite{U-15}. The convergence rate is discussed in \cite{EHK-15} which shows that although the algorithm converges globally, the convergence rate can be superlinear, linear or sublinear for different tensors. Very recently, the R-linear convergence is established for a generic tensor in \cite{HL-18}.

For $r \ge 2$, only some partial results regarding the convergence behaviour of ALS type algorithms for problem \eqref{eq:best-app} are known. The global convergence in case $s=1$ is addressed in \cite{WTY-15} for a generic tensor. For case $s=k$, the global convergence with a sublinear (resp. R-linear) convergence rate for arbitrary (resp. generic) tensor is established in \cite{HY-19}, without any assumption. The situation where $1 < s < k$ is more subtle. The global convergence is proved in \cite{GC-19} under a full rank assumption for the whole iteration sequence, which is removed later in \cite{Y-19} by utilizing a proximal technique. Moreover, as a consequence of \cite{LZ-19}, the global convergence with linear convergence rate can be obtained under an assumption on the limiting point of the iteration sequence. The convergence analysis of ALS type algorithms for problem \eqref{eq:best-app} is still far from accomplished.

This paper is devoted to completely analyse an ALS type algorithm for problem \eqref{eq:best-app}, for any values of $r$ and $s$. More specifically, we propose an algorithm called \emph{the iAPD-ALS algorithm} (Algorithm~\ref{algo}) and we prove without any assumption that
\begin{enumerate}[label=(\roman*)]
\item Any sequence generated by the iAPD-ALS algorithm converges globally to a KKT point of problem \eqref{eq:best-app} (Theorem~\ref{thm:global}).
\item The convergence rate is sublinear for any tensor $\mathcal{A}$ (Theorem~\ref{thm:sublinear}).
\item The convergence rate is $R$-linear for a generic tensor $\mathcal{A}$ (Theorem~\ref{thm:generic}).
\end{enumerate}
Surprisingly, except for the commonly used techniques in optimization theory, proofs of the above convergence results are heavily relied on algebraic and geometric properties of the feasible set of \eqref{eq:best-app}, among which the most important one is the location of KKT points of \eqref{eq:best-app} for a generic $\mathcal{A}$ (Propositions~\ref{prop:KKT location s=3}, \ref{prop:KKT location s=2} and \ref{prop:KKT location s=1}).

The rest of the paper is organized as follows. Preliminaries for linear and multilinear algebra, differential geometry, algebraic geometry and optimization theory are collected in Section~\ref{sec:preliminary}. As a preparation for the convergence analysis of the iAPD-ALS algorithm, we investigate in Section~\ref{sec:podt} the feasible set of problem~\eqref{eq:best-app}, which {consists} of partially orthogonal tensors. Section~\ref{sec:kkt} is concerned with properties of KKT points of problem \eqref{eq:best-app}. With the aid of results obtained in previous sections, the convergence analysis is eventually carried out in Section~\ref{sec:convergence}. To conclude this paper, some final remarks are given in Section~\ref{sec:conclusion}.

\section{Preliminaries}\label{sec:preliminary}
In this section, we provide essential preliminaries for this paper, including rudiments of multilinear algebra, differential geometry, algebraic geometry and optimization theory. For ease of reference, these basic notions and facts are divided into {six} subsections accordingly.

\subsection{Basics of tensors}\label{subsec:notation}
Given positive integers $n_1,\dots,n_k$, a real tensor $\mathcal{A}$ of dimension $n_1\times\dots\times n_k$ is an array of $\prod_{j =1}^k n_j$ real numbers indexed by $k$ indices. Namely, an element in $\mathcal{A}$ is a real number $a_{i_1,\dots, i_k}$ where $1\le i_j \le n_j$ for $1\le j \le k$. The space of real tensors of dimension $n_1\times\dots\times n_k$ is denoted {by} $\mathbb R^{n_1}\otimes\dots\otimes\mathbb R^{n_k}$ and this notation is simplified to be $(\mathbb{R}^n)^{\otimes k}$ if $n_1 = \cdots = n_k = n$.

For reader's convenience, we summarize below notations related to tensors we frequently use in this paper.
\begin{enumerate}[label=(\roman*)]
\item vector spaces are denoted by blackboard bold letters $\mathbb{U},\mathbb{V},\mathbb{W}$, etc. In particular, $\mathbb{R}$ (resp. $\mathbb{C}$) denotes the field of real (resp. complex) numbers. The only exception is $\mathbb{S}^n$, which denotes the $n$-dimensional sphere;
\item vectors are denoted by lower case bold face letters $\mathbf{a},\mathbf{b},\mathbf{c}$, etc.;
\item matrices are denoted by normal italic letters $A,B,C$, etc.;
\item tensors are denoted by calligraphic letters $\mathcal{A},\mathcal{B},\mathcal{C}$, etc.
\end{enumerate}
\subsubsection{Operations defined by contractions}\label{sec:contraction}
Given two tensors $\mathcal{A} \in \mathbb{R}^{n_1} \otimes \cdots \otimes \mathbb{R}^{n_k}$ and $\mathcal{B} \in \mathbb{R}^{n_{i_1}} \otimes \cdots \otimes \mathbb{R}^{{n_{i_s}}}$ where $1\le i_1< \cdots < i_s \le k$ and $\{j_1,\dots, j_{k-s}\} \coloneqq \{1,\dots, k\} \setminus \{i_1,\dots, i_s\}$ with $j_1 < \cdots < j_{k-s}$, we denote by $\langle \mathcal{A}, \mathcal{B} \rangle$ the tensor in $\mathbb{R}^{n_{j_1}} \otimes \cdots \otimes \mathbb{R}^{n_{j_{k-s}}}$ obtained by contracting $\mathcal{A}$ with $\mathcal{B}$. In this paper, the indices involved in the contraction are not explicitly spelled out and understood from the context for the sake of notational simplicity.  For instance, if $i_p = p + k - s, p =1,\dots, s$, then  
\begin{equation}\label{eqn:contraction}
(\langle \mathcal{A}, \mathcal{B} \rangle)_{q_1,\dots, q_{k-s}} = \sum_{\substack{1\le  {r}_l \le n_{l+k-s} \\ 1\le l \le s }} a_{q_1,\dots,q_{k-s},{r}_1,\dots, {r}_s} b_{{r}_1,\dots, {r}_s}.
\end{equation}
In particular, if $s = k$, then we obtain the Hilbert-Schmidt inner product of tensors $\mathcal A, \mathcal B\in\mathbb R^{n_1}\otimes\dots\otimes\mathbb R^{n_k}$:
\[
\langle\mathcal A,\mathcal B\rangle =\sum_{\substack{1\le  i_l \le n_{l} \\ 1\le l \le k }}a_{i_1,\dots, i_k}b_{i_1,\dots, i_k},
\]
of which the induced norm
\[
\lVert\mathcal A\rVert \coloneqq \sqrt{\langle\mathcal A,\mathcal A\rangle}
\]
is called the Hilbert-Schmidt norm \cite{L-13}. Moreover, if $s = k=2$, then an element  $A$ in $\mathbb{R}^{n_1} \otimes \mathbb{R}^{n_2}$ is simply an $n_1 \times n_2$ matrix, whose Hilbert-Schmidt norm reduces to the Frobenius norm $\lVert A \rVert_F$.

Given $k$ matrices $B^{(i)}\in\mathbb R^{m_i\times n_i}$ for $i\in\{1,\dots,k\}$, the \textit{matrix-tensor product} of $(B^{(1)},\dots,B^{(k)})$ and $\mathcal{A}$ is defined to be (cf.\ \cite{L-21})
\begin{equation}\label{eq:matrix-tensor}
(B^{(1)},\dots,B^{(k)})\cdot\mathcal A \coloneqq \langle B^{(1)}\otimes \cdots \otimes B^{(k)}, \mathcal{A} \rangle \in \mathbb{R}^{m_1} \otimes \cdots \otimes \mathbb{R}^{m_k}.
\end{equation}
More precisely, we have
\begin{equation}\label{eq:matrix-tensor1}
\left(  (B^{(1)},\dots,B^{(k)})\cdot\mathcal A\right)_{i_1,\dots, i_k}=\sum_{\substack{1\le j_l \le n_l \\  1\le l \le k  }}b^{(1)}_{i_1j_1}\dots b^{(k)}_{i_kj_k}a_{j_1,\dots, j_k}
\end{equation}
for all $1\le i_t \le m_t$ and $1\le t \le k$.

\subsubsection{Some special maps}\label{subsubsec:special maps}
We first introduce a map sending a vector to a diagonal tensor. For positive integers $k$ and $r$, we define a map $
\operatorname{diag}_k : \mathbb{R}^r\to  (\mathbb{R}^r)^{\otimes k}$ by
\begin{equation}\label{eqn:diag}
\left( \operatorname{diag}_k (\lambda_1,\dots,\lambda_r)\right)_{i_1,\dots, i_k} =
\begin{cases}
\lambda_j\quad &\text{if}~i_1 =\cdots = i_k = j\in\{1,\dots,r\}, \\
0\quad &\text{otherwise}.
\end{cases}
\end{equation}
In particular, $\operatorname{diag}_2(\lambda_1,\dots,\lambda_r)$ is the $r\times r$ diagonal matrix whose diagonal entries are $\lambda_1,\dots,\lambda_r$. It is clear that the map $\operatorname{Diag}_k: (\mathbb{R}^r)^{\otimes k} \to\mathbb R^r$ given by
\begin{equation}\label{eqn:Diag}
(\operatorname{Diag}_k(\mathcal{A}))_i = a_{i,\dots, i},\quad i =1,\dots, r
\end{equation}
is a left inverse of $\operatorname{diag}_k$, i.e., $\operatorname{Diag}_k \circ\operatorname{diag}_k = I_r$, the $r\times r$ identity matrix.

Next we record a class of maps which can be used to produce rank one tensors. We define a map $\tau : \mathbb R^{n_1}\times\dots\times\mathbb R^{n_k}\rightarrow\mathbb R^{n_1}\otimes\dots\otimes\mathbb R^{n_k}$ by
\begin{equation}\label{eq:segre}
\tau(\mathbf u_1,\dots, \mathbf{u}_k) \coloneqq \mathbf u_1\otimes\dots\otimes\mathbf u_k.
\end{equation}
For each $1 \le i \le k$, we also define a map $\tau_i : \mathbb R^{n_1}\times\dots\times\mathbb R^{n_k}\rightarrow\mathbb R^{n_1}\otimes\dots\otimes\mathbb R^{n_{i-1}}\otimes\mathbb R^{n_{i+1}}\otimes\dots\otimes\mathbb R^{n_k}$ by
\begin{equation}\label{eq:partial segre}
\tau_i(\mathbf{u}_1,\dots, \mathbf{u}_k) \coloneqq \tau(\mathbf{u}_1,\dots, \mathbf{u}_{i-1},\mathbf{u}_{i+1},\dots,\mathbf{u}_k) = \mathbf u_1\otimes\dots\otimes\mathbf u_{i-1}\otimes\mathbf u_{i+1}\otimes\dots\otimes\mathbf u_k.
\end{equation}

Given a tensor {$\mathcal A\in\mathbb R^{n_1}\otimes\dots\otimes\mathbb R^{n_k}$}, we define $\mathcal{A}\tau: \mathbb{R}^{n_1} \times \cdots \times \mathbb{R}^{n_k} \to \mathbb{R}$ by
\begin{equation}\label{eq:segre+contraction}
\mathcal{A}\tau (\mathbf{u}_1,\dots, \mathbf{u}_k) \coloneqq \langle \mathcal{A}, \tau (\mathbf{u}_1,\dots, \mathbf{u}_k) \rangle = \langle\mathcal A,\mathbf u_1\otimes\dots\otimes\mathbf u_k\rangle,
\end{equation}
and similarly for each $1\le i \le k$, we define {$\mathcal{A}\tau_i: \mathbb{R}^{n_1} \times \cdots \times \mathbb{R}^{n_k} \to \mathbb{R}^{n_i}$ by}
\begin{equation}\label{eq:partial segre+contraction}
\mathcal{A}\tau_i (\mathbf{u}_1,\dots, \mathbf{u}_k) \coloneqq \langle \mathcal{A}, \tau_i (\mathbf{u}_1,\dots, \mathbf{u}_k) \rangle = \langle\mathcal{A}, \mathbf u_1\otimes\dots\otimes\mathbf u_{i-1}\otimes\mathbf u_{i+1}\otimes\dots\otimes\mathbf u_k \rangle.
\end{equation}
We remark that $\mathcal A\tau$ and $\mathcal A\tau_i$ are related by the relation
\begin{equation}\label{eq:segre and partial segre}
\langle \mathcal A\tau_i (\mathbf{u}_1,\dots, \mathbf{u}_k),\mathbf u_i\rangle=\mathcal A\tau( \mathbf{u}_1,\dots, \mathbf{u}_k).
\end{equation}
Moreover, it holds true that 
\begin{equation}\label{eq:unit}
|\mathcal A\tau(\mathbf{u}_1,\dots, \mathbf{u}_k)|\leq \|\mathcal A\|
\end{equation}
for any unit vectors $\mathbf{u}_i\in\mathbb R^{n_i}$ for all $i=1,\dots,k$.
\subsubsection{Rank and decomposition}\label{subsubsec:rank}
For a tensor $\mathcal{A} \in \mathbb{R}^{n_1}\otimes \cdots \otimes \mathbb{R}^{n_k}$, the \emph{rank} of $\mathcal{A}$, denoted by $\rank (\mathcal{A})$, is the smallest nonnegative integer $r$ such that
\begin{equation}\label{eqn:tensor decomp}
\mathcal{A} = \sum_{j =1}^r \lambda_j \mathbf{u}_j^{(1)} \otimes \cdots \otimes \mathbf{u}_j^{(k)}
\end{equation}
for some $\lambda_j \in \mathbb{R}$ and $\mathbf{u}^{(i)}_j \in \mathbb{R}^{n_i}$ with $\lVert \mathbf{u}^{(i)}_j \rVert = 1$, where $i=1,\dots, k$ and $j=1,\dots,r$. A \textit{rank decomposition} of $\mathcal{A}$ is a decomposition of the form \eqref{eqn:tensor decomp} with $r = \rank(\mathcal{A})$. We say that a tensor $\mathcal{A}$ is \textit{identifiable} if its rank decomposition is \emph{essentially unique} in the following sense: if $\mathcal{A}$ can be written as
\[
\mathcal{A} = \sum_{j =1}^{{\rank(\mathcal A)}} \lambda_j \mathbf{u}_j^{(1)} \otimes \cdots \otimes \mathbf{u}_j^{(k)} = \sum_{j =1}^{{\rank(\mathcal A)}} \mu_j \mathbf{v}_j^{(1)} \otimes \cdots \otimes \mathbf{v}_j^{(k)},
\]
then there exist some permutation $\sigma$ on the set $\{1,\dots, {\rank(\mathcal A)}\}$ and nonzero real numbers $\alpha_{j}^{(i)}$ such that
\[
\mathbf v_{\sigma(j)}^{(i)} = \alpha_{j}^{(i)}  \mathbf u_{j}^{(i)},\quad \left( \prod_{s =1}^k  \alpha_j^{(s)} \right)  \mu_{\sigma(j)} = \lambda_j
\]
for each $i=1,\dots,k$ and $j=1,\dots, {\rank(\mathcal A)}$.

\subsection{Basics of real algebraic geometry} A subset $X$ of $\mathbb{R}^n$ is called a \emph{semi-algebraic set} if there exists a collection of  polynomials $\{f_{ij}\in \mathbb{R}[x_1,\dots, x_n]\colon 1\le j \le r_i,\ 1 \le i \le s\}$ such that
\[
X = \bigcup_{i=1}^s \bigcap_{j=1}^{r_i} \{\x\in \mathbb{R}^n: f_{ij} \ast_{ij} 0 \},
\]
where $\ast_{ij}$ is either $<$ or $=$. In particular, if $\ast_{ij}$ is $=$ for all $1\le j \le r_i, 1 \le i \le s$, then $X$ is called an \emph{algebraic variety} or \emph{algebraic set}. In this case, $X$ is said to be \emph{irreducible} if the decomposition $X = V_1 \cup V_2$ for some algebraic varieties $V_1,V_2$ implies that either $V_1 = X$ or $V_2 = X$.

Let $X\subseteq \mathbb{R}^n, Y\subseteq \mathbb{R}^m$ be two semi-algebraic subsets. A map $f:X\to Y$ is called a \emph{semi-algebraic map} if its graph is semi-algebraic in $\mathbb{R}^{n + m}$. If $f$ is also a homeomorphism, then $f$ is called a \emph{semi-algebraic homeomorphism} and we say that $X$ is \emph{semi-algebraically homeomorphic} to $Y$. The following decomposition theorem is the most fundamental result concerning the structure of a semi-algebraic set.
\begin{theorem}\cite[Theorem~2.3.6]{BCR-98}\label{thm:decomp}
Every semi-algebraic subset of $\mathbb{R}^n$ can be decomposed as the disjoint union of finitely many semi-algebraic sets, each of which is semi-algebraically homeomorphic to an open cube $(0,1)^d \subseteq \mathbb{R}^d$ for some nonnegative integer $d$.
\end{theorem}

\subsubsection{Dimension of a semi-algebraic set}
For a semi-algebraic subset $X$ of $\mathbb{R}^n$, we denote by $I(X)$ the ideal consisting of polynomials vanishing on $X$. The \emph{dimension} of $X$, denoted by $\dim (X)$, is defined to be the Krull dimension of the quotient ring $\mathbb{R}[x_1,\dots, x_n]/I(X)$. The subset $\overline{X}\subseteq \mathbb{R}^n$ consisting of common zeros of polynomials in $I(X)$ is called the \emph{Zariski closure} of $X$. To distinguish, we denote the Euclidean closure of $X$ by $\overline{X}^E$. Notice that in general we have
\[
X \subseteq \overline{X}^E \subseteq \overline{X},
\]
where equalities hold if and only if $X$ is an algebraic variety. As a comparison, dimensions of the three sets are always equal:
\begin{proposition}\cite[Proposition~2.8.2]{BCR-98}\label{prop:closure dimension}
Let $X\subseteq \mathbb{R}^n$ be a semi-algebraic set. Then we have
\[
\dim (X) = \dim (\overline{X}) = \dim (\overline{X}^E).
\]
\end{proposition}

The next result implies that the dimension of a semi-algebraic set is non-increasing under a semi-algebraic map.
\begin{theorem}\cite[Proposition~2.2.7 \& Theorem~2.8.8]{BCR-98}\label{thm:image dimension}
Let $X$ be a semi-algebraic subset in $\mathbb{R}^n$ and let $f:X\to \mathbb{R}^m$ be a semi-algebraic map. Then $f(X)$ is semi-algebraic and $\dim (X) \ge \dim (f(X))$. In particular, if $X$ is an algebraic variety and $f:\mathbb{R}^n \to \mathbb{R}^m$ is a polynomial map, then $\dim (X) \ge \dim (f(X))$.
\end{theorem}
We remark that if $X$ is an algebraic variety and $f$ is a polynomial map, Theorem~\ref{thm:image dimension} only ensures that $f(X)$ is also a semi-algebraic set. In general, $f(X)$ may not be an algebraic variety. For instance, the image of $f(x) = x^2, x\in \mathbb{R}$ is the half line $\mathbb{R}_+$. As a corollary of Theorem~\ref{thm:image dimension}, we obtain the following geometric characterization of the dimension of a semi-algebraic set.
\begin{proposition}\cite[Corollary~2.8.9]{BCR-98}\label{prop:dim}
Let $X = \bigcup_{i=1}^s X_i$ be a finite union of semi-algebraic sets, where each $X_i$ is semi-algebraically homeomorphic to an open cube $(0,1)^{d_i} \subseteq \mathbb{R}^{d_i}$, then $\dim (X) = \max\{d_1,\dots, d_s\}$.
\end{proposition}
By definition, It is obvious that $\dim(X)$ does not depend on the decomposition $X = \bigcup_{i=1}^s X_i$. Proposition~\ref{prop:dim} indicates that the algebraic definition of $\dim (X)$ also coincides with the intuitive geometric view on the dimension of a set. The following is a direct consequence of Proposition~\ref{prop:dim}.
\begin{corollary}\label{cor:0-dim}
Let $X\subseteq \mathbb{R}^n$ be a nonempty semi-algebraic subset. The dimension of $X$ is zero if and only if $X$ consists of finitely many points.
\end{corollary}

\subsubsection{Smooth locus of an algebraic variety} Let $X$ be an irreducible algebraic variety in $\mathbb{R}^n$ and let $I(X)$ be the ideal  of $X$, generated by $f_1,\dots, f_s\in \mathbb{R}[x_1,\dots, x_n]$. For each point $\x\in X$, we define the \emph{Zariski tangent space} of $X$ at $\x$ to be
\[
T_X (\x) = \ker(J_{\x}(f_1,\dots, f_s)),
\]
where $J_{\x}(f_1,\dots, f_s)$ is the Jacobian matrix
\[
J_{\x}(f_1,\dots, f_s) := \begin{bmatrix}
\frac{\partial f_1}{\partial x_1}(\x) &  \cdots & \frac{\partial f_1}{\partial x_n}(\x) \\
\vdots & \ddots & \vdots  \\
\frac{\partial f_s}{\partial x_1}(\x) & \cdots & \frac{\partial f_s}{\partial x_n}(\x) \\
\end{bmatrix} \in \mathbb{R}^{s \times n}
\]
and $\ker(J_{\x}(f_1,\dots, f_s))\subseteq \mathbb{R}^n$ is the right null space of $J_{\x}(f_1,\dots, f_s)$. We remark that in general $\dim(T_X (\x)) \ge \dim (X)$. A point $\x \in X$ is said to be \emph{nonsingular} if $\dim(T_X (\x)) = \dim (X)$. The set of all nonsingular points of $X$, denoted by $X_{\operatorname{sm}}$, is called the \emph{smooth locus} of $X$.
\begin{proposition}\cite[Propositions~3.3.11 \& 3.3.14]{BCR-98}\label{prop:smooth locus}
Let $X$ be a $d$-dimensional irreducible algebraic variety in $\mathbb{R}^n$. We have the following:
\begin{enumerate}[label=(\roman*)]
\item \label{prop:smooth locus:item1} For each nonsingular point $\x$ of $X$, there exists an open semi-algebraic neighbourhood of $\x$ in $X$ which is a $d$-dimensional smooth submanifold\footnote{Indeed, such a neighbourhood can be chosen to be a Nash submanifold in $\mathbb{R}^n$, but we will not need this stronger fact in the sequel.} in $\mathbb{R}^n$.
\item \label{prop:smooth locus:item2} $X_{\operatorname{sm}}$ is a nonempty Zariski open subset of $X$ and $\dim(X_{\operatorname{sm}}) = \dim (X)$. Hence $X\setminus X_{\operatorname{sm}}$ is an algebraic variety of dimension {strictly} smaller than $\dim (X)$.
\end{enumerate}
\end{proposition}

\subsection{Basics of differential geometry}
In this subsection, we give a brief overview of constructions and tools from differential geometry, which are essential to the analysis of the dynamics of the optimization algorithm proposed in this paper.
\subsubsection{Oblique and Stiefel manifolds}\label{subsec:stiefel}
Let $m \le n$ be two positive integers. We define
\begin{equation}\label{eq:fixed-norm}
\operatorname{B}(m,n) \coloneqq \{A\in\mathbb R^{n\times m}\colon (A^\tp A)_{jj}= 1,\ 1\le j \le m\}.
\end{equation}
By definition, a matrix $A\in \mathbb{R}^{n\times m}$ lies in $\operatorname{B}(m,n)$ if and only if each column vector of $A$ is a unit vector. This implies that $\B(m,n)$ is simply the Cartesian product of $m$ spheres in $\mathbb{R}^{n}$:
\begin{equation}\label{eq:prod-sp}
\B(m,n) =  \mathbb S^{n-1}\times\dots\times\mathbb S^{n-1},
\end{equation}
which induces a smooth manifold structure on $\B(m,n)$. Therefore, $\operatorname{B}(m,n)$ is both a smooth submanifold and a closed subvariety of $\mathbb{R}^{n\times m}$.

Inside $\operatorname{B}(m,n)$ there is a well-known {\textit{Oblique manifold} \cite{AMS-08}
\begin{equation}\label{eq:oblique}
\operatorname{OB}(m,n) \coloneqq \{A\in\operatorname{B}(m,n)\colon \operatorname{det}(A^\tp A)\neq 0\}.
\end{equation}
Equivalently, $\operatorname{OB}(m,n)$ consists of full rank matrices in $\operatorname{B}(m,n)$. We notice that $\OB(m,n)$ is an open submanifold of $\B(m,n)$. From the product structure \eqref{eq:prod-sp} of $\B(m,n)$, the tangent spaces of both $\B(m,n)$ and $\OB(m,n)$ can be easily computed. Indeed, if we write $A\in \operatorname{B}(m,n)$ (resp. $B\in \operatorname{OB}(m,n)$) as $A = [\mathbf{a}_1,\dots, \mathbf{a}_m]$ (resp. $B = [\mathbf{b}_1,\dots, \mathbf{b}_m]$), then we have {
\begin{align}
\operatorname{T}_{\B(m,n)} (A)  &= \operatorname{T}_{ \mathbb{S}^{n-1}}(\mathbf{a}_1) \oplus \cdots \oplus \operatorname{T}_{\mathbb{S}^{n-1}}(\mathbf{a}_m), \label{eq:prod-sp:oblique} \\
\operatorname{T}_{\OB(m,n)} (B)  &= \operatorname{T}_{\mathbb{S}^{n-1}}(\mathbf{b}_1) \oplus \cdots \oplus \operatorname{T}_{\mathbb{S}^{n-1}}(\mathbf{b}_m). \label{q:tangent:oblique}
\end{align}
}

In the sequel, we also need the notion of \emph{Stiefel manifold} consisting of $n\times m$ orthonormal matrices:
\begin{equation}\label{eq:stf}
\V(m,n) \coloneqq \{A \in \mathbb R^{n \times m}\colon A^\tp A=I_m \},
\end{equation}
where $I_m$ is the $m\times m$ identity matrix. It is obvious from the definition that $\V(m,n)$ is a closed submanifold of $\operatorname{OB}(m,n)$. In particular, $\V(n,n)$ is simply the orthogonal group $\O (n)$.

In general, if $M$ is a Riemannian embeded submanifold of $\mathbb{R}^k$, then the \emph{normal space} $\operatorname{N}_M(\x)$ of $M$ at a point $\x\in M$ is defined to be the orthogonal complement of its tangent space $\operatorname{T}_M(\x)$ in $\mathbb{R}^k$, i.e., $\operatorname{N}_M (\x) \coloneqq \operatorname{T}_M(\x)^{\perp}$. By definition, $\V(m,n)$ is naturally a Riemannian embeded submanifold of $\mathbb{R}^{n \times m}$ and hence its normal space is well-defined. Indeed, it follows from \cite[Chapter~6.C]{RW-98} and \cite{EAT-98,AMS-08} that {
\begin{equation}\label{eq:normal-stf}
\operatorname{N}_{\V(m,n)}(A)=\{AX: X\in \SS^m\},
\end{equation}
where $\SS^m\subseteq \mathbb{R}^{m \times m}$ is the space of $m\times m$ symmetric matrices.}
For a fixed $A\in \V(m,n)$, projection maps from $\mathbb{R}^{n\times m}$ onto $\operatorname{N}_{\V(m,n)}(A)$ and $\operatorname{T}_{\V(m,n)}(A)$ are respectively given by:
\begin{align}
\pi_{\operatorname{N}_{\V(m,n)}(A)}: \mathbb R^{n \times m} &\to \operatorname{N}_{\V(m,n)}(A),\quad B\mapsto A \left( \frac{A^\tp B+B^\tp A}{2} \right), \label{eq:normal projection} \\
\pi_{\operatorname{T}_{\V(m,n)}(A)}: \mathbb R^{n \times m} &\to \operatorname{T}_{\V(m,n)}(A),\quad B\mapsto (I-\frac{1}{2}AA^\tp )(B-AB^\tp A). \label{eq:tangent-form}
\end{align}

Given a function $f : \mathbb R^k\rightarrow \mathbb R\cup\{\infty\}$, the \emph{subdifferential} \cite{RW-98} of $f$ at $\x\in\mathbb R^k$ is defined as
\begin{equation}\label{eqn:subdifferential}
\partial f(\x) \coloneqq \Bigg\{\mathbf v\in\mathbb R^k\colon\liminf_{\x\neq \y\rightarrow \x}\frac{f(\y)-f(\x)-\langle\mathbf v, \y-\x\rangle}{\|\y-\x\|}\geq 0\Bigg\}.
\end{equation}
Elements in $\partial f(\x)$ are called \textit{subgradients} of $f$ at $\x$.
If $\mathbf 0\in\partial f(\mathbf x)$, then $\mathbf x$ is a \textit{critical point} of $f$. {Given a set $K\subseteq\mathbb R^k$, the function $\delta_{K}:\mathbb{R}^k \to \mathbb{R} \cup \{\infty\}$ defined by
\begin{equation}\label{eq:indicator}
\delta_{K}(\x) \coloneqq \begin{cases}0&\text{if } \x \in K,\\ +\infty &\text{otherwise}\end{cases}
\end{equation}
is called the \emph{indicator function} of $K$. It is an important fact \cite[Section~8.D]{RW-98} that for a submanifold $M$ of $\mathbb R^k$ and any $\x\in M$, we have
\begin{equation}\label{eq:normal-sub}
\partial \delta_{M}(\x)=\operatorname{N}_{M}(\x).
\end{equation}
In particular, since $\V(m,n)$, $\B(m,n)$ and $\OB(m,n)$ are all submanifolds of $\mathbb{R}^{n \times m}$, one may apply \eqref{eq:normal-sub} to compute the subdifferentials of their indicator functions respectively.
\subsubsection{Morse functions}\label{subsec:morse}
In the following, we recall the
notion of a Morse function and some of its basic properties. On a smooth manifold $M$, a {smooth} function $f : M\rightarrow \mathbb R$ is called a \textit{Morse function} if each critical point of $f$ on $M$ is nondegenerate, i.e., the Hessian matrix of $f$ at each critical point is nonsingular. The following result is well-known.
\begin{lemma}[Projection is Generically Morse]\cite[Theorem~6.6]{M-63}\label{lem:generic-morse}
Let $M$ be a {submanifold} of $\mathbb R^n$. For {a generic} $\mathbf a = (a_1,\dots,a_n)^\tp \in\mathbb R^n$, the squared Euclidean distance function $f(\mathbf x) \coloneqq  \|\mathbf x-\mathbf a\|^2
$ is a Morse function on $M$.
\end{lemma}
We will also need the following property of nondegenerate critical points in the sequel.
\begin{lemma}\cite[Corollary~2.3]{M-63}\label{lem:isolated critical points}
Let $M$ be a manifold and let $f:M \to \mathbb{R}$ be a smooth function. Nondegenerate critical points of $f$ are isolated.
\end{lemma}

{To conclude this subsection, we briefly discuss how critical points behave under a local diffeomorphism. {For this purpose,} we recall that {a smooth manifold $M_1$ is \emph{locally diffeomorphic} to another smooth manifold $M_2$} \cite{dC-92} if there is a smooth map $\varphi : M_1\rightarrow M_2$ such that for each point $\x\in M_1$ there exists a neighborhood $U\subseteq M_1$ of $\x$ and a neighborhood $V\subseteq M_2$ of $\varphi(\x)$ such that the restriction $\varphi|_U : U\rightarrow V$ is a diffeomorphism. In this case, the corresponding $\varphi$ is called a \emph{local diffeomorphism} {from $M_1$ to $M_2$}. {It is clear from the definition that {if $M_1$ is locally diffeomorphic to $M_2$ then the two manifolds $M_1$ and $M_2$} must have the same dimension. For ease of reference, we record the following two simple results about local diffeomorphism and critical points.
\begin{lemma}\cite[Corollary~II.6.7]{B-75}\label{lem:local}
Let $M\subseteq\mathbb R^m$ be a manifold and let $f : M \to \mathbb R^n$ be a smooth, injective map. Then $f$ is a differemorphism between $M$ and $f(M)$ if and only if the differential map $d_\x f:\operatorname{T}_M(\x) \to \operatorname{T}_{f(M)}{f(\x)}$ is nonsingular for every $\x\in M$.
\end{lemma}
\begin{proposition}\cite[Proposition~5.2]{HL-18}\label{prop:critical}
Let $M_1,M_2$ be smooth manifolds and let $\varphi : M_1\rightarrow M_2$ be a local diffeomorphism.
Assume that $f : M_2\rightarrow \mathbb R$ is a smooth function. A point $\x\in M_1$ is a (nondegenerate) critical point of $f\circ \varphi$ if and only if $\varphi(\mathbf x)$ is a (nondegenerate) critical point of $f$.
\end{proposition}

\subsubsection{Gradient of a function}\label{subsubsec:gradient}
When $f$ is a smooth function on $\mathbb{R}^n$ and $M$ is a submanifold of $\mathbb{R}^n$, we denote by $\nabla f$ the Euclidean gradient of $f$ as a vector field on $\mathbb{R}^n$, while we denote by $\operatorname{grad}(f)$ the Riemannian gradient of $f$ as a vector field on $M$. The two gradients are related by the formula:
\begin{equation}\label{eqn:gradients}
{\operatorname{grad}(f)(\x) = \pi_{\operatorname{T}_M(\x)} (\nabla f(\x))},\quad \x\in M,
\end{equation}
where $\pi_{\operatorname{T}_M(\x)}: \mathbb{R}^n \to \operatorname{T}_M(\x)$ is the projection map from $\operatorname{T}_{\mathbb{R}^n}(\x) \simeq \mathbb{R}^n$ onto the tangent space $\operatorname{T}_M(\x)$.

\subsection{KKT points and LICQ}\label{subsec:KKT}
Let us consider a general optimization problem
\begin{equation}\label{eq:abstract}
\begin{array}{rl}
\max& f(\x)\\
\text{s.t.}&g_i(\x)=0,\ i=1,\dots,p,
\end{array}
\end{equation}
where $f, g_i: \mathbb R^n\rightarrow \mathbb R$ are continuously differentiable functions. A feasible point $\x$ of \eqref{eq:abstract} is called a \emph{Karush-Kuhn-Tucker point} (KKT point) if the following \emph{KKT condition} is satisfied
\begin{equation}\label{eq:abstract-kkt}
\nabla f(\x)+\sum_{i=1}^pu_i\nabla g_i(\x)=\mathbf 0
\end{equation}
for some Lagrange multiplier vector $\uu = (u_1,\dots, u_p)^\tp \in\mathbb R^p$.

A fundamental result in optimization theory \cite{B-99} is that under some {constraint} qualifications, an optimizer of \eqref{eq:abstract} must be a KKT point.
A word of caution is imperative before we proceed: Constraint qualifications are usually nontrivial sufficient conditions imposing on $f$ and $g_i$'s in \eqref{eq:abstract}, and there exist examples \cite{B-99} for which the KKT condition \eqref{eq:abstract-kkt} does not hold even at global optimizers of \eqref{eq:abstract}. A classic one is \textit{Robinson's constraint qualification} \cite{R-76}. In our case, it becomes the commonly used \textit{linear independence constraint qualification} (LICQ), which requires the gradients $\nabla g_1(\x),\dots,\nabla g_p(\x)$ at the given feasible point $\x$ are linearly independent. It follows from \eqref{eq:abstract-kkt} that if LICQ holds, then the Lagrange multiplier vector $\uu$ is uniquely determined.

\subsection{Tools for convergence analysis}\label{sec:lojasiewicz}
The Kurdyka-\L ojasiewicz property and the \L ojasiewicz inequality are two imperative components of convergence analysis of an optimization algorithm. In this subsection, we provide some necessary definitions and related facts.
\subsubsection{The Kurdyka-\L ojasiewicz property}\label{subsubsec:kl}
In this subsection, we review some basic facts about the Kurdyka-\L ojasiewicz property, which was first established for a smooth definable function in \cite{K-98} by
Kurdyka and later was even generalized to non-smooth definable case in \cite{BDLS-07} by Bolte, Daniilidis, Lewis
and Shiota. Interested readers are also referred to \cite{ ABS-13,LP-16,ABRS-10,BDLM-10} for more discussions and applications of the Kurdyka-\L ojasiewicz property.

Let $f$ be an extended real-valued function on $\mathbb{R}^n$ and let $\partial f(\x)$ be the subdifferential of $f$ at $\x$ defined in \eqref{eqn:subdifferential} \footnote{In the general definition of KL property, $\partial f({\x})$ is actually the limiting subdifferential of $f$ at $\x$. However, for functions considered in this paper, the two notions of subdifferentials coincide. For this reason, we do not distinguish them in the sequel.}. We define $\operatorname{dom}(\x) \coloneqq \{\x\colon\partial f(\x)\neq \emptyset\}$ and take $\x^*\in \operatorname{dom}(\partial f)$. If there exist some $\eta\in(0,+\infty]$, a neighborhood $U$ of $\x^*$, and a continuous concave function $\varphi:[0,\eta)\rightarrow \mathbb{R}_+$, such that
\begin{enumerate}[label=(\roman*)]
\item $\varphi  (0)=0$ and $\varphi$ is {continuously differentiable} on $(0,\eta)$,
\item for all $s\in (0,\eta)$, $\varphi^{\prime}(s)>0$, and
\item for all $\x\in U\cap {\{\y\colon f(\x^*)<f(\y)<f(\x^*)+\eta \}}$, the Kurdyka-\L ojasiewicz inequality holds
\begin{equation}\label{eq:KL property}
\varphi^{\prime}(f(\x) - f(\x^*)) \operatorname{dist}( \mathbf{0},\partial f(\x))\geq 1,
\end{equation}
\end{enumerate}
then we say that $f$ has the \emph{Kurdyka-\L ojasiewicz (abbreviated as KL) property} at $\x^*$. Here $\operatorname{dist}(\mathbf 0,\partial f(\x))$ denotes the distance from $\mathbf 0$ to the set $\partial f({\x} )$. If $f$ is proper, lower semicontinuous, and has the KL property at each point of $\operatorname{dom}(\partial f)$, then $f$ is said to be a \emph{KL function}. Examples of KL functions include real subanalytic functions and semi-algebraic functions \cite{BST-14}.

In this paper, semi-algebraic functions will be involved, we refer to \cite{BCR-98} and references {herein} for more details. In particular, polynomial functions and indicator functions of semi-algebraic sets are semi-algebraic functions\cite{BCR-98,BST-14}. Also, a finite sum of semi-algebraic functions is again semi-algebraic {\cite{BCR-98,BST-14}}. We assemble these facts to derive the following lemma which will be crucial to the analysis of the global convergence of our algorithm.
\begin{lemma}\label{lem:KL functions}
A finite sum of polynomial functions and indicator functions of semi-algebraic sets is a KL function.
\end{lemma}
\subsubsection{The \L ojasiewicz inequality}
The \L ojasiewicz inequality discussed in this subsection is an essential ingredient in convergence rate analysis. The classical \L ojasiewicz inequality for analytic functions is stated as follows (cf.\ \cite{L-63}):
\begin{itemize}
\item[] {\bf (Classical \L ojasiewicz's Gradient Inequality)} {If $f$ is an analytic function and $\nabla f(\mathbf x^*)=\mathbf 0$, then there exist positive constants $\mu, \kappa$ and $\epsilon$
such that
\begin{equation}\label{eqn:classicalKL}
\|\nabla f(\x)\|\ge \mu |f(\x)-f(\x^*)|^{\kappa}\;\mbox{ for all }\;\|\mathbf x-\x^*\|\le\epsilon.
\end{equation}}
\end{itemize}
We remark that \eqref{eqn:classicalKL} is the KL property discussed in Subsection~\ref{subsubsec:kl} spcialized with 
the auxiliary function being $\varphi(s)=\frac{1}{(1-\kappa)\mu}s^{1-\kappa}$. As pointed out in \cite{ABRS-10,BDLM-10}, it is often difficult to determine the corresponding exponent $\kappa$ in \L ojasiewicz's gradient inequality, and it is unknown for a general function. Fortunately, when $f$ is a polynomial function, an estimate of $\kappa$ is obtained by D'Acunto and Kurdyka. We record this result in the next lemma, which plays a key role in our sublinear convergence rate analysis.

\begin{lemma}[\L ojasiewicz's Gradient Inequality for Polynomials]\cite[Theorem~4.2]{DK-05}\label{lemma:loja1}
Let $f$ be a real polynomial of degree $d$ in $n$ variables. Suppose that $\nabla f(\x^*)=\mathbf 0$. There exist positive constants $\mu$ and $\epsilon$ such that for all $\|\mathbf x-\x^*\|\le\epsilon$, we have
\[
\|\nabla f(\x)\|\ge \mu|f(\x)-f(\x^*)|^{\kappa},
\]
where $\kappa \coloneqq 1-\frac{1}{d(3d-3)^{n-1}}$.
\end{lemma}

Below is a generalization of the classical \L ojasiewicz gradient inequality \eqref{eqn:classicalKL} to a manifold, in which we can take the exponent $\kappa$ to be $1/2$ at a nondegenerate critical point. Functions with this property are usually called \textit{Polyak-\L ojasiewicz functions}. 
\begin{proposition}[\L ojasiewicz's Gradient Inequality]\cite{dC-92}\label{prop:lojasiewicz}
Let $M$ be a smooth manifold and let $f: M \to \mathbb R$ be a smooth function for which $\z^*$ is a nondegenerate critical point. Then there exist a neighborhood $U$ in $M$ of $\z^*$ and some $\mu >0$ such that for all $\z\in U$
\[
\|\operatorname{grad}(f)(\z)\| \geq \mu  |f(\z)-f(\z^*)|^{\frac{1}{2}}.
\]
\end{proposition}

\subsection{Properties of polar decompositions}\label{sec:polar}
In this subsection, we present an error bound property for the polar decomposition.
\begin{lemma}[Polar Decomposition]\label{lem:polar}
Let $A\in\mathbb R^{n\times m}$ with $n\geq m$. Then there exist an orthonormal matrix $U\in \V(m,n)$ and a unique $m\times m$ symmetric positive semidefinite matrix $H$ such that $A=UH$ and
\begin{equation}\label{eq:polar-optimal}
U\in\operatorname{argmax}\{\langle Q,A\rangle\colon Q\in \V(m,n)\}.
\end{equation}
Moreover, if $A$ is of full rank, then $U$ is uniquely determined and $H$ is positive definite.
\end{lemma}

The decomposition $A=UH$ as in Lemma~\ref{lem:polar} is called the \textit{polar decomposition} of $A$  \cite{GV-13}. Correspondingly, $U$ is called {a} \textit{polar orthonormal factor matrix} and $H$ is called a  \textit{polar positive semidefinite factor matrix}. We denote by $\operatorname{Polar}(A)$ the subset of $\V(m,n)$ consisting of all polar orthonormal factor matrices of $A$.

We observe that the description \eqref{eq:polar-optimal}  of $U$ comes from rewriting the optimization problem
\[
\min_{Q\in \V(m,n)}\ \|B-QC\|^2
\]
for two given matrices $B$ and $C$ of appropriate sizes. Theorem~\ref{thm:error} below provides us a global error bound for this problem.
\begin{theorem}[Global Error Bound in Frobenius Norm]\cite{HY-19}\label{thm:error}
Let $p, m, n$ be positive integers with $n\ge m$ and let $B\in\mathbb R^{n\times p}$ and $C\in\mathbb R^{m\times p}$ be two given matrices. We set $A = BC^\tp \in\mathbb R^{n\times m}$ and suppose that $A$ has a polar decomposition $A=WH$ for some orthonormal matrix $W$ and symmetric positive semidefinite matrix $H$. Then for any $Q\in \V(m,n)$, we have
\begin{equation}\label{eq:polar-error}
\|B-QC\|_F^2-\|B-WC\|_F^2\geq \sigma_{\min}(A)\|W-Q\|^2_F,
\end{equation}
where $\sigma_{\min}(A)$ denotes the smallest singular value of $A$.
\end{theorem}

In the sequel, we also need the following inequality:
\begin{lemma}{\cite{HY-19}}\label{lem:distance}
For any orthonormal matrices $U,V\in \V(m,n)$, we have
\begin{equation}\label{eq:distance}
\|U^\tp V-I_m\|_F^2 \leq \|U-V\|^2_F.
\end{equation}
\end{lemma}
\section{Partially Orthogonal Tensors}\label{sec:podt}
Let $k,n_1,\dots, n_k$ and $1\le s \le k$ be positive integers and let $\mathcal A\in\mathbb R^{n_1}\otimes\dots\otimes\mathbb R^{n_k}$ be a tensor with a decomposition
\begin{equation}\label{eq:low-rank-orth}
\mathcal A=\sum_{i=1}^r\sigma_i\mathbf a^{(1)}_i\otimes\dots\otimes\mathbf a^{(k)}_i,
\end{equation}
such that
\begin{equation}\label{eq:factor-matrix}
A^{(i)}:=\begin{bmatrix}\mathbf a^{(i)}_1,\dots,\mathbf a^{(i)}_r\end{bmatrix}\in\operatorname{B}(r,n_i)\ \text{for all }i=1,\dots,k
\end{equation}
and
\begin{equation}\label{eq:factor-matrix-orth}
A^{(i)}\in\V(r,n_i)\ \text{for all }i=1,\dots,s.
\end{equation}
The matrix $A^{(i)} \in \mathbb{R}^{n_i \times r}$ is called the $i$-th \emph{factor (loading) matrix} of the decomposition \eqref{eq:low-rank-orth} for $i=1,\dots, k$. Obviously, we must have
\[
r\leq\min\{n_i\colon i=1,\dots,s\}.
\]

The tensor $\mathcal{A}$ is called a {\emph{partially orthogonal tensor (with  $s$ orthonormal factors)} of rank at most $r$} and a decomposition of the form \eqref{eq:low-rank-orth} with the smallest possible $r$ is called a \textit{partially orthogonal rank decomposition} of $\mathcal{A}$. We denote by $P_s (\mathbf{n},r)$ the set of all such tensors, where $\mathbf{n}:= (n_1,\dots, n_k)$. We remark that tensors in $P_k(\mathbf n,r)$ are orthogonally decomposable tensors studied in the literature \cite{HL-18,HY-19,ZG-01}.
The following is a simple observation directly obtained from the definition of $P_s (\mathbf{n},r)$.
\begin{proposition}
For fixed $\mathbf n$ and $r$, $P_s(\mathbf n,r)$'s are nested:
\[
P_k(\mathbf n,r) \subsetneq P_{k-1}(\mathbf n,r) \subsetneq \cdots \subsetneq P_1(\mathbf n,r).
\]
\end{proposition}

Before we proceed to a more detailed study of partially orthogonal tensors, we investigate the identifiability of them. It is well-known \cite{ZG-01} that tensors in $P_k(\mathbf n,r)$ are all identifiable, but this is no longer true for $P_s(\mathbf n,r)$ when $s < k$. For instance, we have
\[
\mathbf{u}_1 \otimes \mathbf{u}_1 \otimes \mathbf{v} + \mathbf{u}_2 \otimes \mathbf{u}_2 \otimes \mathbf{v} = \frac{1}{2}  (\mathbf{u}_1 + \mathbf{u}_2) \otimes (\mathbf{u}_1 + \mathbf{u}_2) \otimes \mathbf{v}  + \frac{1}{2}  (\mathbf{u}_1 - \mathbf{u}_2) \otimes (\mathbf{u}_1 - \mathbf{u}_2) \otimes \mathbf{v},
\]
where $\{\mathbf{u}_1,\mathbf{u}_2\}$ is an orthonormal basis of $\mathbb{R}^2$ and $\mathbf{v}\in \mathbb{R}^2$ is a unit vector. However, an application of Kruskal's uniqueness theorem provides us an explicit condition to guarantee the identifiability.

\begin{proposition}[Identifiability]\label{prop:unique}
Suppose $k\geq 3$ and $s\ge 1$. If in a decomposition \eqref{eq:low-rank-orth} of a tensor $\mathcal{A}\in P_{s}(\mathbf{n},r)$, factor matrices $A^{(s+1)},\dots, A^{(k)}$ are of full column ranks, then $\mathcal{A}$ is identifiable. Moreover, if $s\ge 3$, then any tensor in $P_s(\mathbf n,r)$ is identifiable.
\end{proposition}
\begin{proof}
Without loss of generality, we may assume that all $\sigma_1,\dots, \sigma_r$ are nonzero, since otherwise we can remove the zero components in \eqref{eq:low-rank-orth} and the rest decomposition satisfies all the hypotheses. The existence of decomposition \eqref{eq:low-rank-orth} of $\mathcal{A}$ ensures that $\rank (\mathcal{A}) \le r$ and hence Kruskal's uniqueness theorem \cite{K-77,SB-00} applies.
\end{proof}

A tensor $\mathcal{A} \in P_s(\mathbf{n},r)$ is called \textit{independently partially orthogonal} if factor matrices $A^{(s+1)},\dots, A^{(k)}$ in decomposition \eqref{eq:low-rank-orth} are respectively of the full rank $r$. The set of all independently partially orthogonal tensors with at most (resp. exactly) $r$ nonzero $\sigma_i$'s is denoted by $Q_s(\mathbf n,r)$ (resp. $R_s(\mathbf n,r)$). By definition, we have
\begin{equation}\label{eqn:inclusion PQR}
R_s(\mathbf n,r) \subsetneq Q_s(\mathbf n,r) \subseteq P_s(\mathbf n,r),\quad Q_s(\mathbf n,r) = \bigsqcup_{t=0}^r R_{{s}}(\mathbf n,{t}).
\end{equation}
Here the second inclusion is proper for $s<k$, but it becomes an equality for $s = k$. Figure~\ref{fig:nested structure} provides an intuitive picture for the nested structure described in \eqref{eqn:inclusion PQR}. Moreover, Figure~\ref{fig:nested structure} also serves as a pictorial interpretation of a more refined geometric relation among  $P_s(\mathbf n,r),Q_s(\mathbf n,r)$ and $R_s(\mathbf n,r)$, i.e., $R_s(\mathbf n,r)$ and $Q_s(\mathbf n,r)$ are dense subsets of $P_s(\mathbf n,r)$, which is the content of  Corollary~\ref{cor:Q_s}.

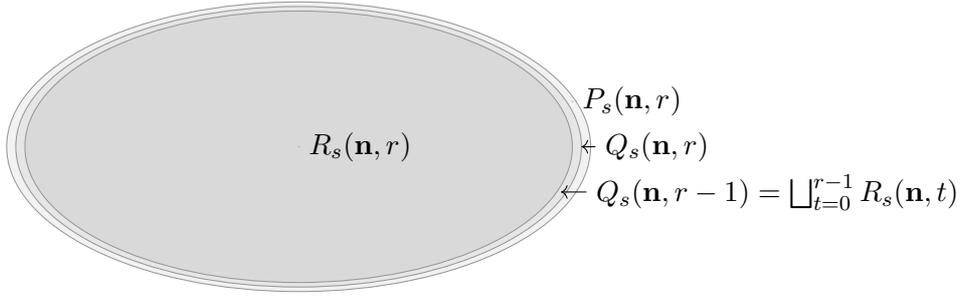
\begin{figure}
\begin{tikzpicture}[decoration=zigzag,scale=0.6]
  \filldraw[fill=gray!10, draw=gray!80] (0,0) ellipse (6.4cm and 3.2cm);
  \path[] (0,0) ++(110:6.4cm and 3.2cm) arc(110:70:6.4cm and 3.2cm);

  \filldraw[fill=gray!20, draw=gray!80] (0,0) ellipse (6.2cm and 3.1cm);
  \path[] (0,0) ++(110:4.8cm and 2.4cm) arc(110:70:6cm and 3.1cm);

  \filldraw[fill=gray!30, draw=gray!80] (0,0) ellipse (6cm and 3cm);
  \path[] (0,0) ++(110:3.2cm and 1.6cm) arc(110:70:6cm and 3cm);


  \filldraw[draw=black!80] (6,1) circle (0pt) node[anchor=west] {$P_s(\mathbf{n},r)$};

\filldraw[draw=black!80] (6.5,0) circle (0pt) node[anchor=west] {$Q_s(\mathbf{n},r)$};
    \draw[->] (6.5,0) -- (6.2,0);

    \filldraw[draw=black!80] (6.3,-1) circle (0pt) node[anchor=west] {$Q_s(\mathbf{n},r-1) = {\bigsqcup_{t=0}^{r-1} R_s(\mathbf{n},t)}$};
        \draw[->] (6.3,-1) -- (5.75,-1);

    \filldraw[draw=black!80] (0,0) circle (0pt) node[anchor=west] {$R_s(\mathbf{n},r)$};

\end{tikzpicture}
\caption{{Nested structure of $P_s(\mathbf n,r)$}}
\label{fig:nested structure}
\end{figure}

For any $\mathcal{A} \in R_s(\mathbf{n},r)$, we have $\rank (\mathcal{A}) = r$
and Proposition~\ref{prop:unique} implies that its {rank decomposition coincides with its partially orthogonal rank decomposition}, which is essentially unique. Therefore it is not necessary to distinguish {rank decomposition and partially orthogonal rank decomposition} of a tensor in $Q_s(\mathbf{n},r)$ and hence in $R_s(\mathbf{n},r)$.

The existence of decomposition \eqref{eq:low-rank-orth} of $\mathcal{A}\in P_{s}(\mathbf{n},r)$ trivially implies that $\rank (\mathcal{A}) \le r$. In fact, we can even determine $\rank (\mathcal{A})$ directly from any decomposition of the form \eqref{eq:low-rank-orth}.
\begin{proposition}[Rank]\label{prop:rank}
Let $k\geq 3$ and $s\geq 2$. If $\mathcal A\in P_s(\mathbf n,r)$ has a decomposition \eqref{eq:low-rank-orth}, then $\rank(\mathcal A)=\# \{i\colon \sigma_i\neq 0\}$ and a partially orthogonal rank decomposition is a rank decomposition.
\end{proposition}
\begin{proof}
Suppose without loss of generality that $\sigma_i\neq 0$ for all $i=1,\dots, r$. We want to prove that $\rank (\mathcal{A}) = r$. It is clear from the definition of the tensor rank that $r \ge \rank (\mathcal{A})$, due to the existence of the decomposition \eqref{eq:low-rank-orth}. Hence it suffices to prove the reverse inequality. To that end, since $s \ge 2$, we may flatten $\mathcal{A}$ as a matrix $A$ of size $n_1 \times \prod_{j=2}^k n_j$ so that the partially orthogonal decomposition \eqref{eq:low-rank-orth} of $\mathcal{A}$ translates into a singular value decomposition of $A$. In particular, nonzero numbers $\sigma_1,\dots, \sigma_r$ are singular values of $A$ and hence $\rank(\mathcal{A}) \ge \rank (A) = r$, which completes the proof.
\end{proof}

We remark that Proposition~\ref{prop:rank} is not true for $P_1(\mathbf{n},r)$. For example, we consider
\[
\mathcal{T} = \mathbf{u}_1 \otimes \mathbf{v} \otimes \mathbf{w} + \mathbf{u}_2 \otimes \mathbf{v} \otimes \mathbf{w},
\]
where $\{\mathbf{u}_1,\mathbf{u}_2\}$ is an orthonormal base of $\mathbb{R}^2$ and $\mathbf{v},\mathbf{w}$ are unit vectors in $\mathbb{R}^2$. It is clear that $\mathcal{T} \in P_1((2,2,2),2)$ and $\sigma_1 = \sigma_2 = 1$ in the above partially orthogonal decomposition, while $\rank (\mathcal{T}) = 1$ since $\mathcal{T}  = (\mathbf{u}_1 + \mathbf{u}_2) \otimes \mathbf{v} \otimes \mathbf{w}$.

There is another simple method to compute the rank of $\mathcal{A} \in P_s(\mathbf{n},r)$, without first knowing any partially orthogonal decomposition of $\mathcal{A}$.
\begin{proposition}\label{prop:matrix rank}
Let $k\geq 3$ and $s\geq 1$. The rank of $\mathcal{A} \in P_s(\mathbf{n},r)$ is equal to $\rank (A_1)$, where $A_1 \in \mathbb{R}^{n_1 \times \prod_{i=2}^k n_i}$ is the flattening of $\mathcal{A}$ with respect to its first factor.
\end{proposition}

\begin{proof}
{We may assume that $\mathcal A\neq 0$.}
Since $\mathcal{A} \in P_s(\mathbf{n},r)$, there exists some {$1\le t \le r$} such that $\mathcal{A} $ can be written as
\[
\mathcal{A} = \sum_{i=1}^t \mathbf{u}_i \otimes \mathcal{T}_i,
\]
where $\{\mathbf{u}_1,\dots,\mathbf{u}_t \}$  is an orthonormal set and $\rank(\mathcal{T}_i) = 1$ for $1 \le i \le t$. It is clear that $\rank (A_1)$ is the {maximal} cardinality of linearly independent subsets of $\{\mathcal{T}_1,\dots, \mathcal{T}_t\}$. Without loss of generality, we assume $\rank (A_1) = r_0$ and $\{\mathcal{T}_1,\dots, \mathcal{T}_{r_0}\}$ is a {maximal} linearly independent subset of $\{\mathcal{T}_1,\dots, \mathcal{T}_t\}$. Therefore, $\mathcal{T}_{i}$ is a linear combination of $\mathcal{T}_1,\dots, \mathcal{T}_{r_0}$ for each $i = r_0 + 1, \dots, t$, from which we may obtain a decomposition of $\mathcal{A}$ as the sum of $r_0$ rank one tensors. This implies that $\rank (\mathcal{A}) \le r_0$ and hence $\rank (\mathcal{A}) = r_0 = \rank (A_1)$ since we also have $\rank(A_1) \le \rank (\mathcal{A})$ by the definition of flattening matrices of a tensor.
\end{proof}

Before we proceed, we make a clarification on the applicability of Proposition~\ref{prop:matrix rank}: It only applies to $\mathcal{A}$ which lies in $P_s(\mathbf{n},r)$. In fact, for a general tensor $\mathcal{A}$, $\rank(A_1)$ can be strictly smaller than $\rank(\mathcal{A})$. For example, for any linearly independent $\mathbf{u},\mathbf{v}\in \mathbb{R}^2$, the tensor
\[
\mathcal{A} = \mathbf{u} \otimes \mathbf{u} \otimes \mathbf{v} + \mathbf{u} \otimes \mathbf{v} \otimes \mathbf{u} + \mathbf{v} \otimes \mathbf{u} \otimes \mathbf{u} \in \mathbb{R}^2 \otimes \mathbb{R}^2 \otimes \mathbb{R}^2
\]
is not partially orthogonal. It is straightforward to verify that $\rank(\mathcal{A}) = 3$ but $\rank(A_1) = 2$. Moreover, $\rank(\mathcal{A}) = \rank(A_1)$ does not guarantee {that} $\mathcal{A}$ is partially orthogonal, which can be easily seen from the tensor
\[
\mathcal{A} = \mathbf{u}_1 \otimes \mathbf{v}_1 \otimes \mathbf{w}_1 + \mathbf{u}_2 \otimes \mathbf{v}_2 \otimes \mathbf{w}_2 \in \mathbb{R}^2 \otimes \mathbb{R}^2 \otimes \mathbb{R}^2,
\]
where $\{\mathbf{u}_1,\mathbf{u}_2\}$, $\{\mathbf{v}_1,\mathbf{v}_2\}$ and $\{\mathbf{w}_1,\mathbf{w}_2\}$ are non-orthogonal bases of $\mathbb{R}^2$ respectively.
\subsection{Low rank partially orthogonal tensor approximation problem and its algorithm}\label{subsec:lowrank}
We begin with the statement of low rank partially orthogonal tensor approximation (LRPOTA) problem: Given a tensor $\mathcal A\in\mathbb R^{n_1}\otimes\dots\otimes\mathbb R^{n_k}$, find a partially orthogonal tensor $\mathcal B\in\mathbb R^{n_1}\otimes\dots\otimes\mathbb R^{n_k}$ of rank at most $r\leq\min\{n_1,\dots,n_s\}$ such that the residual $\|\mathcal A-\mathcal B\|$ is minimized. The above problem can be formulated as the following optimization problem:
\begin{flalign}\label{eq:original}
&\text{(LRPOTA)}\ \ \ \ \ \ \quad \quad\quad\quad\quad\quad\quad\quad\quad  \begin{array}{rl}
\min& \frac{1}{2} \|\mathcal A- \mathcal{B} \|^2\\
\text{s.t.}& \mathcal{B} \in P_s(\mathbf{n},r).
\end{array}&
\end{flalign}
We notice that problem \eqref{eq:original} is actually a constrained optimization problem in the tensor space $\mathbb{R}^{n_1} \otimes \cdots \otimes \mathbb{R}^{n_k} \simeq \mathbb{R}^N$ where $N = \prod_{j=1}^k n_j$. It is tempting to solve problem  \eqref{eq:original} by {methods for constrained optimization problems, such as} the Lagrange multiplier method \cite{B-82}, but such a method is not practical as $P_s(\mathbf{n},r)$ has a huge number of complicated defining equations from the algebraic perspective (cf. Lemmas~\ref{lem:P_s>1} and \ref{lem:P_s=1}), and it is actually an exponentially small subset of $\mathbb{R}^N$ from the geometric perspective (cf. Proposition~\ref{prop:smooth}-\eqref{prop:smooth:item5}). Moreover, there even lacks a theoretical guarantee for the applicability of the Lagrange multiplier method to problem \eqref{eq:original} since $P_s(\mathbf{n},r)$ is not smooth (cf. Proposition~\ref{prop:smooth}-\eqref{prop:smooth:item6}). A brief discussion along this line can be found in Section~\ref{sec:projection}. To address the above two issues, we can parametrize $P_s(\mathbf{n},r)$ and reformulate the problem~\eqref{eq:original} as:

\begin{flalign}\label{eq:sota}
&\text{(LRPOTA)}\ \ \ \ \ \ \begin{array}{rl}\min&\|\mathcal A- (U^{(1)},\dots,U^{(k)})\cdot \mathcal{D} \|^2\\
\text{s.t.}& \mathcal{D}=\operatorname{diag}_k(\lambda_1,\dots,\lambda_r),\ {\lambda_j\in\mathbb R\ \text{for all }j\in\{1,\dots,r\}},\\
& \big(U^{(i)}\big)^\tp U^{(i)}=I_r \  \text{for all }i\in\{1,\dots,s\},\\
& U^{(i)}\in\B(r,n_i)\ \text{for all }i\in\{s+1,\dots,k\}.
\end{array}&
\end{flalign}
Here we remind readers that $\operatorname{diag}_k(\lambda_1,\dots,\lambda_r)\in (\mathbb{R}^r)^{\otimes k}$ is the diagonal tensor defined in \eqref{eqn:diag} and $(U^{(1)},\dots,U^{(k)})\cdot \mathcal{D}\in \mathbb{R}^{n_1} \otimes \cdots \otimes \mathbb{R}^{n_k}$ is the matrix-tensor product of $(U^{(1)},\dots,U^{(k)})$ and $\mathcal{D}$ defined in \eqref{eq:matrix-tensor}. Some extreme cases of problem \eqref{eq:sota} are intensively considered in the literature: For $s=1$, the problem is investigated in \cite{WTY-15}; For $s=k$, it is the completely orthogonal low rank tensor approximation problem discussed thoroughly in \cite{HL-18,HY-19}; For $r=1$, the problem simply reduces to the best rank one tensor approximation problem studied in {\cite{HL-18,ZG-01,DDV-00}}.

\begin{proposition}[Maximization Equivalence]\label{prop:sota-max}
The approximation problem~\eqref{eq:sota} is equivalent to the maximization problem
\begin{flalign}\label{eq:sota-max}
&\text{(mLRPOTA)}\ \ \ \ \ \ \quad\quad\quad
\begin{array}{rl}
\max& \sum_{j=1}^r\Big(\big(\big(U^{(1)}\big)^\tp ,\dots,\big(U^{(k)}\big)^\tp \big)\cdot\mathcal A\Big)_{{j,\dots,j}}^2\\
 \text{s.t.} &\big(U^{(i)}\big)^\tp U^{(i)}=I_r \  \text{for all }i\in\{1,\dots,s\},\\
& U^{(i)}\in\B(r,n_i)\ \text{for all }i\in\{s+1,\dots,k\},
\end{array}&
\end{flalign}
in the following sense:
\begin{enumerate}[label=(\roman*)]
\item an optimizer
\[
(\widehat{U}, \widehat{\mathcal{D}}) \coloneqq \big( (\widehat{U}^{(1)},\dots,\widehat{U}^{(k)}) ,\operatorname{diag}_k( \widehat{\lambda}_1,\dots, \widehat{\lambda}_r) \big)
\]
of \eqref{eq:sota} gives an optimizer $\widehat{ U}$ of \eqref{eq:sota-max} with respective optimal values {$\|\mathcal A\|^2-\sum_{j=1}^r \widehat{\lambda}_j^2$ and $\sum_{j=1}^r \widehat{\lambda}_j^2$};
\item conversely, an optimizer $\widehat{ U}$ of \eqref{eq:sota-max}, together with
\[
\widehat{\mathcal{D}}=\diag_k\circ\Diag_k
\left(
( (\widehat{U}^{(1)})^\tp,\dots,(\widehat{U}^{(k)})^\tp) \cdot\mathcal A
\right),
\]
gives an optimizer of \eqref{eq:sota}.
\end{enumerate}
\end{proposition}

\begin{proof}
It is by a direct calculation to {obtain} that
 \begin{align*}
 \|\mathcal A-(U^{(1)},\dots,U^{(k)})\cdot \mathcal{D} \|^2&=\|\mathcal A\|^2+\sum_{j=1}^r\lambda_j^2-2\langle\mathcal A,(U^{(1)},\dots,U^{(k)})\cdot \mathcal{D} \rangle\\
 &=\|\mathcal A\|^2+\sum_{j=1}^r\lambda_j^2-2\Big\langle\big(\big(U^{(1)}\big)^\tp ,\dots,\big(U^{(k)}\big)^\tp \big)\cdot\mathcal A,\mathcal{D}\Big\rangle\\
 &=\|\mathcal A\|^2+\sum_{j=1}^r\lambda_j^2-2\sum_{j=1}^r\lambda_j\Big[\big(\big(U^{(1)}\big)^\tp ,\dots,\big(U^{(k)}\big)^\tp \big)\cdot\mathcal A\Big]_{j,\dots,j}.
 \end{align*}
Note that $\lambda_j$ in the minimization problem \eqref{eq:sota} is unconstrained for all $j\in\{1,\dots,r\}$, and they are mutually independent.
Thus, at an optimizer $(\widehat{U},\widehat{\mathcal{D}}) \coloneqq \big(( \widehat{U}^{(1)},\dots, \widehat{U}^{(k)}), \operatorname{diag}_k( \widehat{\lambda}_1,\dots, \widehat{\lambda}_r)  \big)$ of \eqref{eq:sota}, we must have by the optimality that
 \begin{equation}\label{eq:lambda}
 \widehat{\lambda}_j=\Big[\big(\big(\widehat{U}^{(1)}\big)^\tp ,\dots,\big(\widehat{U}^{(k)}\big)^\tp \big)\cdot\mathcal A\Big]_{j,\dots,j}\ \text{for all }j\in\{1,\dots,r\}
 \end{equation}
 with the optimal value being
 \[
 \|\mathcal A\|^2-\sum_{j=1}^r \widehat{\lambda}_j^2.
 \]
Therefore, problem~\eqref{eq:sota} is equivalent to \eqref{eq:sota-max}. The other statement follows by a similar calculation.
\end{proof}

We remark that \eqref{eq:sota-max} actually depends on the value of $r$, while the name ``mLRPOTA" does not indicate this dependence. In the most discussion of this paper, the value of $r$ is always understood from the context, hence we use the term ``mLRPOTA" for simplicity. However, in Section~\ref{sec:kkt} we need to deal with  problem \eqref{eq:sota-max} for different values of $r$. Therefore we instead use the term ``mLRPOTA($r$)" for clarification.

Next we investigate the relation between KKT points of problem~\eqref{eq:sota} and those of its maximization reformulation \eqref{eq:sota-max}.
\begin{lemma}\label{lem:kkt-equiv}
A feasible point $(U,\mathcal{D})$ of problem \eqref{eq:sota} is a KKT point with a multiplier $P$ if and only if $U$ is a KKT point of problem \eqref{eq:sota-max} with a multiplier $P$ and $\mathcal{D} = \operatorname{diag}_k\big(\operatorname{Diag}_k\big(((U^{(1)})^\tp ,\dots,(U^{(k)})^\tp )\cdot\mathcal A\big)\big)$.
\end{lemma}

\begin{proof}
Noticing that variables $\lambda_j$'s in \eqref{eq:sota} are unconstrained and also the objective function \eqref{eq:sota} is quadratic with leading coefficient $1$ for each $\lambda_j$, we may conclude that for fixed $U^{(i)}$'s, the objective function of \eqref{eq:sota} has a unique critical point for $\lambda_j$'s. Thus, when applying the general theory of calculating the KKT points of \eqref{eq:sota} (cf.\ Section~\ref{subsec:KKT}), we see that $\mathcal{D}$ is uniquely determined by $U$. The desired correspondence between KKT points of \eqref{eq:sota} and \eqref{eq:sota-max} then follows.
\end{proof}

Motivated by the SVD-based algorithm presented in \cite{GC-19}, we propose Algorithm~\ref{algo} to numerically solve the optimization problem \eqref{eq:sota}. The goal of this paper is to analyse the convergence behaviour of Algorithm~\ref{algo}.
\begin{algorithm}
\caption{iAPD-ALS algorithm for low rank partially orthogonal tensor approximation problem}\label{algo}
\begin{algorithmic}[0]
\State{Input: a nonzero tensor $\mathcal A\in\mathbb R^{n_1}\otimes\dots\otimes\mathbb R^{n_k}$, a positive integer $r$ and a proximal parameter $\epsilon > 0$.}
\State{Output: a partially orthogonal tensor approximation of $\mathcal{A}$.}
\State{Initialization: Choose $ U_{[0]}:=(U^{(1)}_{[0]},\dots,U^{(k)}_{[0]})\in \V(r,n_1)\times\dots\times \V(r,n_s)\times\B(r,n_{s+1})\times\dots\times \B(r,n_k)$ such that
$f( U_{[0]})>0$, a truncation parameter $\kappa\in(0,\sqrt{f( U_{[0]})/r})$ and set $p:=1$.}

\algrule
\While{not converge}
\For{$i=1,\dots, k$}
\If {$i\le s$} \Comment{alternating polar decompositions}
\State{Compute $U_{[p]}^{(i)}\in\operatorname{Polar}(V^{(i)}_{[p]}\Lambda^{(i)}_{[p]})$ and $S^{(i)}_{[p]}:=(U^{(i)}_{[p]})^\tp V^{(i)}_{[p]}\Lambda^{(i)}_{[p]}$.}

\If{$\sigma^{(i)}_{r,[p]}=\lambda_{\min}(S^{(i)}_{[p]})<\epsilon$}  \Comment{proximal correction}
\State {Update $U_{[p]}^{(i)} \in \operatorname{Polar}(V^{(i)}_{[p]}\Lambda^{(i)}_{[p]}+\epsilon U^{(i)}_{[p-1]})$ and $S^{(i)}_{[p]} \coloneqq (U^{(i)}_{[p]})^\tp {\big(V^{(i)}_{[p]}\Lambda^{(i)}_{[p]}+\epsilon U^{(i)}_{[p-1]}\big)}$} 
\EndIf
\EndIf
\If{$|\big((U^{(s)}_{[p]})^\tp V^{(s)}_{[p]}\big)_{jj}|<\kappa$ for some $j\in J\subseteq\{1,\dots,r\}$}  \Comment{truncation}
\If{$i\in\{1,\dots,s\}$}
\State{Update $U^{(i)}_{[p]}:=\big(U^{(i)}_{[p]}\big)_{\{1,\dots,r\}\setminus J}$.}
\Else
\State{Update $U^{(i)}_{[p-1]}:=\big(U^{(i)}_{[p-1]}\big)_{\{1,\dots,r\}\setminus J}$.}
\EndIf
\State{Update $r:=r-|J|$.}
\EndIf

\If{$i=s+1,\dots,k$}
\Comment{alternating least squares}
\For{$j=1,\dots,r$}
\State{Compute
\begin{equation}\label{eq:iteration0}
\mathbf u^{(i)}_{j,[p]}:=
\sgn(\lambda^{i-1}_{j,[p]})\frac{\mathcal A\tau_i(\mathbf x^{i}_{j,[p]})}{\|\mathcal A\tau_i(\mathbf x^{i}_{j,[p]})\|}.
\end{equation}
}
\EndFor
\EndIf
\EndFor
\State{Update $p:=p+1$.} \Comment{iteration number}
\EndWhile
\end{algorithmic}
\end{algorithm}

To conclude this subsection, we explain Algorithm~\ref{algo} in details. To do that, we fix some notations. For each $i=1,\dots,k$, $j=1,\dots,r$ and $p\in \mathbb{N}$, we define
\begin{align}
\mathbf x^{i}_{j,[p]} &\coloneqq (\mathbf u^{(1)}_{j,[p]},\dots,\mathbf u^{(i-1)}_{j,[p]}, \mathbf u^{(i)}_{j,[p-1]}, \mathbf u^{(i+1)}_{j,[p-1]},\dots,\mathbf u^{(k)}_{j,[p-1]}), \label{eq:inter-vector} \\
\lambda^{i-1}_{j,[p]} &\coloneqq \mathcal A\tau(\mathbf x^{i}_{j,[p]}),\quad
\Lambda^{(i)}_{[p]} \coloneqq \operatorname{diag}_2(\lambda^{i-1}_{1,[p]},\dots,\lambda^{i-1}_{r,[p]}), \label{eq:lambda-orth}\\
\mathbf v^{(i)}_{j,[p]} &\coloneqq \mathcal A\tau_i(\mathbf x^{i}_{j,[p]}),\quad V^{(i)}_{[p]} \coloneqq \begin{bmatrix}\mathbf v^{(i)}_{1,[p]}&\dots&\mathbf v^{(i)}_{r,[p]}\end{bmatrix}, \label{eq:marixv}
\end{align}
where $\mathbf u^{(i)}_{j,[p]}$ is the $j$-th column of the factor matrix $U^{(i)}_{[p]}$.

Furthermore, for a matrix $B\in\SS^r$, $\lambda_{\min}(B)$ denotes the smallest eigenvalue of $B$. Given a matrix $A\in\mathbb R^{n\times r}$ and a nonempty subset $J\subseteq\{1,\dots,r\}$, $A_J\in\mathbb R^{n\times |J|}$ is the submatrix of $A$ formed by columns of $A$ indexed by $J$.

We are now in the position to decode Algorithm~\ref{algo}. 
We denote by $f( U)$ the objective function of \eqref{eq:sota-max}.
Generally speaking, the algorithm starts with an initialization $U_{[0]}:=(U^{(1)}_{[0]},\dots,U^{(k)}_{[0]})$ and $p:=1$. If $U_{[p-1]}$ is already obtained, in the $p$-th iteration, Algorithm~\ref{algo} updates matrices in $U_{[p-1]}$ sequentially as
\[
U^{(1)}_{[p-1]} \rightarrow U^{(1)}_{[p]},\ U^{(2)}_{[p-1]}\rightarrow U^{(2)}_{[p]},\dots,U^{(k)}_{[p-1]}\rightarrow U^{(k)}_{[p]}.
\]
To obtain $U_{[p]}$ from $U_{[p-1]}$, Algorithm~\ref{algo} performs two different types of operations, depending on whether $i \le s$ or $i > s$. For $i=1,\dots,s$, $U^{(i)}_{[p-1]}$ is updated by computing the polar decomposition 
\begin{equation}\label{eq:algorithm1-polar}
U^{(i)}_{[p]} S^{(i)}_{[p]}  = V^{(i)}_{[p]}\Lambda^{(i)}_{[p]}+ \alpha U^{(i)}_{[p-1]},
\end{equation}
where $\alpha = \epsilon$ or $0$ depending on whether or not there is a proximal correction. The polar decomposition of a matrix $X\in\mathbb{R}^{m\times n}$ can be performed by the following two steps:
\begin{enumerate}[label=(\roman*)]
\item  We compute the SVD of {$X$}:
\begin{equation}\label{eq:svd}
X = G \Sigma H^\tp
\end{equation}
where $G\in \V(n,m)$, $H \in \O(n)$ and singular values of $X$ are sorted nonincreasingly in the diagonal of $\Sigma$.
\item The desired polar decomposition is:
\begin{equation}\label{eq:polar}
X = US,
\end{equation}
where $U \coloneqq GH^\tp$, $S \coloneqq H \Sigma H^\tp$.
\end{enumerate}

To ensure the convergence, we modify {orthonormal} factors in these updates by proximal correction (resp. truncation) with respect to $\epsilon$ (resp. $\kappa$). For $i = s+1,\dots, k$,  Algorithm~\ref{algo} simply updates $U^{(i)}_{[p-1]}$  by the usual ALS method {together with a sign coherence choice}. It is worthy to remark that although Algorithm~\ref{algo} is designed to solve problem~\eqref{eq:sota}, whose feasible set is a parameter space of $P_s(\mathbf n,r)$, the convergence analysis of Algorithm~\ref{algo} requires {an in-depth} understanding of $P_s(\mathbf n,r)$ itself.

\subsection{Algebraic geometry of $P_s(\mathbf n,r)$}\label{sec:part-orth}
In general, a numerical algorithm solving the optimization problem \eqref{eq:sota} (or equivalently its maximization reformulation \eqref{eq:sota-max}) is usually designed in the parameter space ${V_{\mathbf{n},r} \coloneqq \V(r,n_1)\times \dots\times \V(r,n_s)\times\B(r,n_{s+1})\times\dots\times\B(r,n_k)\times\mathbb R^r}$. See for example, \cite{CS-09,GC-19,Y-19,WTY-15}. However, from a more geometric perspective, we can also regard problem \eqref{eq:sota} as the projection of a given tensor $\mathcal A$ onto $P_s(\mathbf n,r)$ in the tensor space $\mathbb{R}^{n_1} \otimes \cdots \otimes \mathbb{R}^{n_k}$. A key ingredient in our study of problem \eqref{eq:sota} is the connection between these two viewpoints. Once such a connection is established, we will be able to implement an algorithm in the parameter space $V_{\mathbf{n},r}$ but analyse it in the tensor space $P_s(\mathbf n,r)$.

Note that we have a natural vector space isomorphism $\mathbb R^m\otimes\mathbb R^n \simeq  \mathbb R^{mn}$. Because of this isomorphism, in the sequel we may simply identify $\mathbb{R}^m\otimes\mathbb R^n$ with $\mathbb R^{mn}$ when tensor structure is not of concern. In the following, unless otherwise stated, we always assume $k \ge s \ge 1$ and cases $s\geq 2$ and $s=1$ are treated separately.}
\begin{lemma}\label{lemma:induction}
Let $k \ge 4$ and $k > s \ge 3$ be positive integers. A tensor ${\mathcal T}\in \mathbb{R}^{n_1} \otimes \cdots \otimes \mathbb{R}^{n_k}$ lies in $P_s(\mathbf{n},r)$ if and only if
\[
\mathcal T \in P_s(\mathbf{m},r) \bigcap P_s(\mathbf{m}',r),
\]
where $\mathbf{m} \coloneqq (n_1,\dots, n_{s-1},n_{s}n_{s+1},n_{s+2},\dots,n_{k})$ and $\mathbf{m}' \coloneqq (n_1,\dots, n_{s-1}n_{s+1}, n_s, n_{s+2} \dots, n_{k})$.
\end{lemma}
\begin{proof}
It is straightforward to verify the inclusion $P_s(\mathbf{n},r) \subseteq P_s(\mathbf{m},r) \bigcap P_s(\mathbf{m}',r)$ and hence it suffices to prove the reverse inclusion. Suppose that $\mathcal T \in P_s(\mathbf{m},r)$.
We may assume that $\rank (\mathcal{T}) = r$. Otherwise, we simply replace $r$ by $\rank (\mathcal{T})$. We decompose $\T$ as
\[
\mathcal T = \sum_{j=1}^r \mathbf{a}^{(1)}_j \otimes \cdots \otimes \mathbf{a}^{(s-1)}_j \otimes M_j \otimes \mathbf{a}^{(s+2)}_j \otimes \cdots  \otimes \mathbf{a}^{(k)}_j,
\]
where $A^{(i)} = [\mathbf{a}^{(i)}_1,\dots, \mathbf{a}^{(i)}_r] \in {\V(r,n_i)}$ for $i \in \{1,\dots, s-1\}$, $A^{(s)} = [M_1,\dots, M_r]\in \V(r,n_{s}n_{s+1})$ and $\mathbf{a}_j^{(l)} \in {\mathbb{R}^{n_{l}}\setminus\{\mathbf 0\}}$ for $l \in\{ s+2,\dots, k\}$. Here we adopt the convention that if $s = k - 1$ then $n_l = 0$ and hence $\mathbf{a}_j^{(l)}$ is just a nonzero scalar. Similarly, if $\T \in P_s(\mathbf{m}',r)$, then we have
\[
\mathcal T = \sum_{j=1}^r \mathbf{b}^{(1)}_j \otimes \cdots \otimes \mathbf{b}^{(s-2)}_j \otimes N_j \otimes \mathbf{b}_j^{(s)} \otimes  \mathbf{b}^{(s+2)}_j \otimes \cdots \otimes \mathbf{b}^{(k)}_j,
\]
where $B^{(i)} = [\mathbf{b}^{(i)}_1,\dots, \mathbf{b}^{(i)}_r] \in \V(r,n)$ for $i\in \{1,\dots, s-2, s\}$, $A^{(s-1)} = [N_1,\dots, N_r]\in \V(r,n_{s-1}n_{s+1})$ and $\mathbf{b}_j^{(l)} \in {\mathbb{R}^{n_{l}}\setminus\{\mathbf 0\}}$, for $l \in \{ s+2,\dots, k\}$. For $p,q\in \{1,\dots, r\}$, we contract $\mathcal T$ with $\mathbf{a}_p^{(1)} \otimes \cdots \otimes \mathbf{a}_p^{(s-2)} \otimes \mathbf{b}_q^{(s)}$ to obtain
\[
\mathbf{a}_p^{(s-1)} \otimes \langle M_p, \mathbf{b}_q^{(s)} \rangle \otimes \mathbf{a}^{(s+2)}_p \otimes \cdots \otimes \mathbf{a}^{(k)}_p =\left( \prod_{i=1}^{s-2} \langle \mathbf{b}_q^{(i)},  \mathbf{a}_p^{(i)} \rangle \right)  N_q \otimes \mathbf{b}^{(s+2)}_q \otimes \cdots \otimes \mathbf{b}^{(k)}_q.
\]

We observe that $\langle M_p, \mathbf{b}_q^{(s)} \rangle \ne \mathbf 0$ implies that $\mathbf{a}^{(l)}_p$ and $\mathbf{b}^{(l)}_q$ are only differed by a nonzero scalar multiple for $l \in \{s+2,\dots, k\}$ and $N_q = \mathbf{a}^{(s-1)}_p \otimes \widetilde{\mathbf{b}}^{(s+1)}_q$ for some $\widetilde{\mathbf{b}}^{(s+1)}_q \in \mathbb{R}^{n_{s+1}} \setminus \{\mathbf{0}\}$, while  $\langle M_p, \mathbf{b}_q^{(s)} \rangle = 0$ implies $\prod_{i=1}^{s-2} \langle \mathbf{b}_q^{(i)},  \mathbf{a}_p^{(i)} \rangle = 0$.

For each $q$ there exists some $p$ such that $\langle M_p, \mathbf{b}_q^{(s)} \rangle \ne \mathbf 0$ since $\langle \T, \mathbf{b}_q^{(s)} \rangle \ne 0$. Moreover, if for some $q$ there exist $p \ne p'$ such that $\langle M_p, \mathbf{b}_q^{(s)} \rangle \ne \mathbf 0$ and $\langle M_{p'}, \mathbf{b}_q^{(s)} \rangle \ne \mathbf 0$, then
\[
N_q = \mathbf{a}^{(s-1)}_p \otimes \widetilde{\mathbf{b}}^{(s+1)}_q = \mathbf{a}^{(s-1)}_{p'} \otimes \widehat{\mathbf{b}}^{(s+1)}_q,
\]
which contradicts the assumption that $\langle \mathbf{a}^{(s-1)}_p, \mathbf{a}^{(s-1)}_{p'} \rangle = 0$. Therefore, we may reorder summands in $\T$ so that $\langle M_p, \mathbf{b}_p^{(s)} \rangle \ne 0$ for each $p\in \{1,\dots,r \}$ and hence we may write
\[
\T = \sum_{j=1}^r \mathbf{b}^{(1)}_j \otimes \cdots \otimes \mathbf{b}^{(s-2)}_j \otimes \mathbf{a}^{(s-1)}_j  \otimes \mathbf{b}_j^{(s)} \otimes \widetilde{\mathbf{b}}_j^{(s+1)} \otimes  \mathbf{b}^{(s+2)}_j \otimes \cdots \otimes \mathbf{b}^{(k)}_j,
\]
from which we may conclude that $\mathcal T\in P_s(\mathbf{n},r)$.
\end{proof}

\begin{lemma}\label{lem:P_s>1}
Let $k \ge 3$ and $k \ge s \ge 3$ be positive integers. For any $\mathbf{n} = (n_1,\dots, n_k)$ and {$r\le \min\{n_1,\dots, n_s\}$}, the set $P_s(\mathbf{n},r)$ is a real algebraic variety defined by quadratic polynomials.
\end{lemma}
\begin{proof}
For each fixed $s \ge 3$, we proceed by induction on $k\ge s$. First we consider the case $k = s$ and $\mathbf{n} = (n_1,\dots, n_s)$. {Note that in this case, $P_s(\mathbf{n},r)$ is the set consisting of all completely orthogonal tensors of rank at most $r$ in $\mathbb{R}^{n_1} \otimes \cdots \otimes \mathbb{R}^{n_s}$. According to \cite[Proposition 31]{BDHR17}, $P_s(\mathbf{n},r)$ is a real algebraic variety cut out by quadratic polynomials.}

Next suppose that the statement is proved for some $k> s$. Then for any $\mathbf{n} = (n_1,\dots, n_{k+1})$, Lemma~\ref{lemma:induction} implies that
\[
P_s(\mathbf{n},r) = P_s(\mathbf{m},r) \bigcap P_s(\mathbf{m}',r),
\]
where $\mathbf{m} \coloneqq (n_1,\dots, n_{s-1},n_{s}n_{s+1},n_{s+2},\dots,n_{k+1}) $ and
$\mathbf{m}' \coloneqq (n_1,\dots, n_{s-1}n_{s+1}, n_s, n_{s+2}, \dots, n_{k+1})$. Therefore, $P_s(\mathbf{n},r)$ is a real algebraic variety defined by quadratic polynomials since both $\mathbf{m}$ and $\mathbf{m}'$ have $k$ components and  $V(I + J) = V(I) \bigcap V(J)$ \cite{H-77}, where $I,J$ are ideals in a polynomial ring {and $V(I)$ denotes the algebraic set defined by the ideal $I$}.
\end{proof}

As a direct consequence of the proof of Lemma~\ref{lem:P_s>1} and \cite[Proposition 31]{BDHR17}, we can even obtain an explicit description of quadratic polynomials defining $P_s(\mathbf{n},r)$. However, these quadratic polynomials are complicated and will not be needed in the sequel. Next we deal with the case $s \le 2$.
\begin{lemma}\label{lem:P_s=1}
Let $k\ge 3$ be a positive integer and $s\in\{1,2\}$. For any $\mathbf{n} = (n_1,\dots, n_k)$ and $r\le \min\{n_i\colon i=1,\dots,s\}$, the set $P_s(\mathbf{n},r)$ is a real algebraic variety.
\end{lemma}

\begin{proof}
The conclusion for $s=1$ follows from the conclusion for $s=2$. In turn, the conclusion for $s=2$ follows from that for $s=3$, which is given by Lemma~\ref{lem:P_s>1}. As the rationale is the same, we only present the proof for $s=1$ by assuming the conclusion holds for $s=2$.

Let $\{\mathbf{e}_1,\dots, \mathbf{e}_r\}$ be a fixed orthonormal basis of $\mathbb{R}^r$. We first consider the set $Z$ consisting of tensors $\mathcal{T} \in \mathbb{R}^{n_1} \otimes \cdots \otimes \mathbb{R}^{n_k} \otimes \mathbb{R}^r$ satisfying
\[
\rank(\langle \mathcal{T}, \mathbf{e}_i \rangle) \le 1,\quad  \langle \langle \mathcal{T}, \mathbf{e}_i \rangle,\langle \mathcal{T}, \mathbf{e}_j \rangle \rangle = 0,\quad 1\le i < j \le r.
\]
We claim that $Z$ is a real algebraic variety. In fact, we may construct a polynomial map{
\begin{align*}
\psi : \mathbb{R}^{n_1} \otimes \cdots \otimes \mathbb{R}^{n_k} \otimes \mathbb{R}^r &\to \left( \mathbb{R}^{n_1} \otimes \cdots \otimes\mathbb{R}^{n_k} \right)^{\times r},\\
\T &\mapsto (\langle \mathcal{T}, \mathbf{e}_1 \rangle,\dots, \langle \mathcal{T}, \mathbf{e}_r \rangle).
\end{align*}}
On the one hand, it is straightforward to verify that $Z \subseteq \psi^{-1}(\Sigma_r)$ where $\Sigma_r$ is the real subvariety of $(\mathbb{R}^{n_1} \otimes \cdots \otimes \mathbb{R}^{n_k})^{\times r}$ consisting of all $r$-tuples $(\mathcal{T}_1,\dots, \mathcal{T}_r)$ such that $\rank (\mathcal{T}_i) \le 1$ and $\langle \mathcal{T}_i,\mathcal{T}_j \rangle = 0$ for all $1\le i < j \le r$. On the other hand, every $\mathcal{T}$ in $\mathbb{R}^{n_1} \otimes \cdots \otimes \mathbb{R}^{n_k} \otimes \mathbb{R}^r$ can be written as $\mathcal{T} = \sum_{i=1}^r \mathcal{T}_i \otimes \mathbf{e}_i$ for some tensors $\mathcal{T}_1, \dots, \mathcal{T}_r\in \mathbb{R}^{n_1} \otimes \cdots \otimes \mathbb{R}^{n_k}$, which are not necessarily of rank at most $1$. However, if $\mathcal{T}$ lies in $\psi^{-1}(\Sigma_r)$, then we may conclude that $\{ \mathcal{T}_i := \langle \mathcal{T}, \mathbf{e}_i \rangle, i =1,\dots, r \}$ must satisfy the desired property, which implies that $Z = \psi^{-1}(\Sigma_r)$ and this proves our claim.

Next we let $\widehat{\mathbf{n}}$ denote the sequence $(n_1,n_2,\dots, n_k,r)$.
We remark that for any $\mathcal{T} \in Z \bigcap P_2(\widehat{\mathbf{n}},r)$,\footnote{Here we identify $P_2(\widehat{\mathbf{n}},r)$ with $P_2(\tilde{\mathbf{n}},r)$ where $\tilde{\mathbf n}=(n_1,r,n_2,\dots,n_k)$.} there exist $s \le r$ uniquely determined orthogonal rank one tensors $\mathcal{T}_1,\dots, \mathcal{T}_s\in \mathbb{R}^{n_1} \otimes \cdots \otimes\mathbb{R}^{n_k}$ and indices $1\le j_1 < \dots <  j_s\le r$ such that $\mathcal{T} = \sum_{i=1}^s \mathcal{T}_i \otimes \mathbf{e}_{j_i}$. Existence is clear from the definition of $Z$. Suppose that there are other orthogonal rank one tensors $\mathcal{T}'_1,\dots, \mathcal{T}'_t$ together with indices $1\le j'_1 < \dots <  j'_t\le r$ such that $\mathcal{T} = \sum_{i=1}^t \mathcal{T}'_i \otimes \mathbf{e}_{j'_i}$. If $\{j_1,\dots, j_s\} \ne \{j'_1,\dots, j'_t\}$, we may assume without loss of generality that
\[
j_1 \in \{j_1,\dots, j_s\} \setminus \{j'_1,\dots, j'_t\}.
\]
Then $\mathcal{T}_1 = \langle \mathcal{T}, \mathbf{e}_{j_1} \rangle  = 0$ which contradicts the assumption $\rank (\mathcal{T}_1) = 1$. Therefore we have $s = t$ and $\{j_1,\dots, j_s\} = \{j'_1,\dots, j'_s\}$ so that (up to some re-ordering of indices)
\[
\sum_{i=1}^s \mathcal{T}_i \otimes \mathbf{e}_{j_i} = \mathcal{T} =  \sum_{i=1}^s \mathcal{T}'_i \otimes \mathbf{e}_{j_i},
\]
from which we further have
\[
\mathcal{T}_i = \langle \mathcal{T}, \mathbf{e}_{j_i} \rangle = \mathcal{T}'_i,\quad i=1,\dots, s.
\]

Let $\varphi:Z \bigcap P_2(\widehat{\mathbf{n}},r) \to P_1(\mathbf{n},r)\subseteq \mathbb{R}^{n_1} \otimes \cdots \otimes \mathbb{R}^{n_k}$ be the polynomial map defined by
\[
\varphi(\mathcal{T}) = \langle \mathcal{T}, \mathbf{e}_1 + \cdots + \mathbf{e}_r \rangle.
\]
{It is straightforward to verify that $\varphi$ is surjective onto its image $\im(\varphi)= P_1(\mathbf{n},r)$.}
If $\varphi(\mathcal{T}) = 0$, then we may conclude from the unique decomposition of $\mathcal{T}$ discussed in the last paragraph that $\mathcal{T} = 0$. Moreover, we also have $\lambda \mathcal{T} \in Z \bigcap P_2(\widehat{\mathbf{n}},r)$ (resp. $\lambda \mathcal{S} \in P_1(\mathbf{n},r)$) whenever $\lambda\in \mathbb{R}$ and $\mathcal{T} \in Z \bigcap P_2(\widehat{\mathbf{n}},r)$ (resp. $\mathcal{S} \in P_1(\mathbf{n},r)$). Therefore, $\varphi$ descends to a well-defined map
\[
\overline{\varphi}: \mathbb{P}_{\mathbb{R}} \left( Z \bigcap P_2(\widehat{\mathbf{n}},r) \right) \to \mathbb{P}_{\mathbb{R}} \left( \mathbb{R}^{n_1} \otimes \cdots \otimes \mathbb{R}^{n_k} \right),
\]
where $\mathbb{P}_{\mathbb{R}} (W)$ is a projectivization of an affine subvariety $W$ in $\mathbb{R}^N$. Since $Z \bigcap P_2(\widehat{\mathbf{n}},r)$ is an affine variety stable under scaling, $\mathbb{P}_{\mathbb{R}} \left( Z \bigcap P_2(\widehat{\mathbf{n}},r) \right)$ is a projective subvariety of $\mathbb{P}_{\mathbb{R}}(\mathbb{R}^{n_1} \otimes \cdots \otimes \mathbb{R}^{n_k} \otimes \mathbb{R}^r)$. This implies that $\im (\overline{\varphi}) = \mathbb{P}_{\mathbb{R}} \left(P_1(\mathbf{n},r) \right)$ is a closed subvariety of $\mathbb{P}_{\mathbb{R}} \left( \mathbb{R}^{n_1} \otimes \cdots \otimes \mathbb{R}^{n_k} \right)$ \cite[Theorem~1.10]{S-13}. Inasmuch as the set $P_1(\mathbf{n},r)$ is the affine cone over $\im (\overline{\varphi})$, $P_1(\mathbf{n},r)$ is a closed subvariety of $\mathbb{R}^{n_1} \otimes \cdots \otimes \mathbb{R}^{n_k}$.
\end{proof}

We summarize the above results.
\begin{theorem}\label{thm:P_s}
Let $k\ge 3$ and $s\ge 1$ be positive integers. For any $\mathbf{n} = (n_1,\dots, n_k)$ and {$r \le \min\{n_1,\dots, n_s\}$}, the set $P_s(\mathbf{n},r)$ is an irreducible real algebraic variety. Moreover if $s\ge 3$, then  $P_s(\mathbf{n},r)$ is defined by quadratic polynomials.
\end{theorem}
\begin{proof}
The irreducibility of $P_s(\mathbf{n},r)$ follows from the fact that $P_s(\mathbf{n},r) = \im (\varphi_{\mathbf{n},r})$ where $\varphi_{\mathbf{n},r}$ is the polynomial map defined by
\begin{equation}\label{eq:tensor-map}
\varphi_{\mathbf n,r} : \V(r,n_1)\times\dots\times\V(r,n_s)\times \operatorname{B}(r,n_{s+1})\times\dots\times\operatorname{B}(r,n_k)\times\mathbb R^r\rightarrow \mathbb R^{n_1}\otimes\dots\otimes\mathbb R^{n_k}
\end{equation}
and $\V(r,n_1)\times\dots\times\V(r,n_s)\times \operatorname{B}(r,n_{s+1})\times\dots\times\operatorname{B}(r,n_k)\times\mathbb R^r$ is an irreducible\footnote{Notice that when $n_j = r$, $\V(r,n_j) = \O(r)$ is disconnected. In this case, we simply replace $\O(r)$ by its component $\operatorname{SO}(r)$, the special orthogonal group in dimension $r$. This does not change the image of $\varphi_{\mathbf{n},r}$.} real algebraic variety. Other assertions directly follow from Lemmas~\ref{lem:P_s>1} and \ref{lem:P_s=1}.
\end{proof}
}

\begin{corollary}[{Zariski Open Subsets}]\label{cor:Q_s}
Let $r\leq\min\{n_1,\dots,n_k\}$. The set $P_s(\mathbf{n},r) \setminus R_s(\mathbf{n},r)$ is a real algebraic variety and hence $R_s(\mathbf{n},r)$ is a Zariski open dense subset of $P_s(\mathbf{n},r) $. In particular, $Q_s(\mathbf{n},r)$ contains a Zariski open dense subset of  $P_s(\mathbf{n},r) $. Moreover, if $k>s \ge 3$, then the real algebraic set $P_s(\mathbf{n},r) \setminus R_s(\mathbf{n},r)$ is defined by polynomials of degrees $(r+1)^{k-s}$ and $2$.
\end{corollary}

\begin{proof}
We only need to prove that $P_s(\mathbf{n},r) \setminus R_s(\mathbf{n},r)$ is a real algebraic variety with the claimed defining equations. The case $k=s$ follows from Lemma~\ref{lem:P_s>1} and \cite{HY-19}. In the following, we assume $k>s$.
Let $\mathbf{t} := (t_1,\dots, t_k)$ be a sequence of positive integers and let $\operatorname{Sub}_{\mathbf{t}}(\mathbf{n})$ be the subset of $\mathbb{R}^{n_1} \otimes \cdots \otimes \mathbb{R}^{n_k}$ consisting of tensors $\T$ such that
\[
\T \in \mathbb{E}_1 \otimes \cdots \otimes \mathbb{E}_k,
\]
where $\mathbb{E}_j \subseteq \mathbb{R}^{n_j}$ is a subspace and $\dim \mathbb{E}_j = t_j$ for $j=1,\dots, k$. It is a well-known fact \cite{L-12} that $\operatorname{Sub}_{\mathbf{t}}(\mathbf{n})$ is a real algebraic variety defined by polynomials of degrees $t_1+1,\dots, t_k +1$.

Now for each $j=s+1,\dots, k$, we take $\mathbf{t}^{(j)}_i \in \mathbb{N}^k$ to be the sequence whose entries are all equal to $r$, except the $j$-th one, which is equal to $(r-1)$. Then by definition we have
\[
P_s(\mathbf{n},r) \setminus R_s(\mathbf{n},r) \subseteq \left( \bigcup_{j=s+1}^k \operatorname{Sub}_{\mathbf{t}^{(j)}}(\mathbf{n}) \right) \bigcap P_s(\mathbf{n},r).
\]
We claim that the reverse inclusion also holds. Indeed, an element $\T$ in $ P_s(\mathbf{n},r)$ can be decomposed as
\[
\T = \sum_{i=1}^r \lambda_i\mathbf{a}^{(1)}_i \otimes \cdots \otimes \mathbf{a}^{(s)}_i \otimes \mathbf{a}_i^{(s+1)} \otimes \cdots \otimes \mathbf{a}^{(k)}_i.
\]
If some $\lambda_i=0$, then definitely $\T\notin R_s(\mathbf{n},r)$. Thus, we can assume that $\lambda_i\neq 0$ for all $i$.
If $\T$ also lies in $ \bigcup_{j=s+1}^k \operatorname{Sub}_{\mathbf{t}^{(j)}}(\mathbf{n})$ then it lies in some $\operatorname{Sub}_{\mathbf{t}^{(j_0)}}(\mathbf{n})$ for some $j_0$. For simplicity, we may assume that $j_0 = s+1$. By definition, we must have
\[
\mathbf{a}^{(s+1)}_i =\frac{1}{\lambda_i}  \langle \mathcal{T}, \mathbf{a}^{(1)}_i \otimes \cdots \otimes \mathbf{a}^{(s)}_i \otimes \mathbf{a}_i^{(s+2)} \otimes \cdots \otimes \mathbf{a}^{(k)}_i \rangle \in \mathbb{E}_{s+1}
\]
for some subspace $\mathbb{E}_{s+1} \subseteq \mathbb{R}^{n_{s+1}}$ of dimension at most $(r-1)$. Hence the $(s+1)$-th factor matrix
\[
A^{(s+1)} =\begin{bmatrix}
\mathbf a^{(s+1)}_1,\dots,\mathbf a^{(s+1)}_r
\end{bmatrix}
\]
of $\T$ cannot have the full column rank and this implies again $\T\not\in R_s(\mathbf{n},r)$. Thus, $R_s(\mathbf{n},r)$ is Zariski open dense subset of $P_s(\mathbf{n},r) $.

Regarding degrees of polynomials defining $P_s(\mathbf{n},r) \setminus R_s(\mathbf{n},r)$ when $s\ge 3$, we recall that
\begin{equation}\label{cor:Q_s:eqn}
V(I) \bigcup V(J) = V(IJ),\quad V(I) \bigcap V(J) = V(I + J),
\end{equation}
where $I,J$ are ideals in a polynomials ring. Lemma~\ref{lem:P_s>1} and the fact that $\operatorname{Sub}_{\mathbf{t}^{(j)}}(\mathbf{n})$ is defined by polynomials of degree $(r+1)$, together with \eqref{cor:Q_s:eqn},  supply $P_s(\mathbf{n},r) \setminus R_s(\mathbf{n},r)$ with a set of defining equations of the desired degrees.
\end{proof}

We conclude this subsection by a remark: Corollary~\ref{cor:Q_s} implies that, roughly speaking,  $R_s(\mathbf{n},r)$ captures almost all points of $P_s(\mathbf{n},r)$, which is pictorially indicated in Figure~\ref{fig:nested structure}. It turns out that due to this density of $R_s(\mathbf{n},r)$ in $P_s(\mathbf{n},r)$, for almost all choices of $\mathcal{A}$, it is reasonable to analyse Algorithm~\ref{algo} in $R_s(\mathbf{n},r)$, which has a better differential geometric property than $P_s(\mathbf{n},r)$.

\subsection{Differential geometry of $Q_s(\mathbf{n},r)$ and $R_s(\mathbf{n},r)$}\label{sec:geometry} Let $k,n_1,\dots,n_k, r\le \min \{n_1,\dots,n_k\}$ be positive integers and let $\mathbf{n} = (n_1,\dots,n_k)$. We recall from \eqref{eqn:inclusion PQR} that
\[
R_s(\mathbf{n},r) \subsetneq Q_s(\mathbf{n},r) \subseteq P_s(\mathbf{n},r),\quad Q_s(\mathbf{n},r) = \bigsqcup_{t=0}^{{r}} R_{{s}}(\mathbf{n},{t}).
\]
The structure of $P_s(\mathbf n,r)$ is complicated. For instance, $P_s(\mathbf n,r)$ may contain tensors which are not identifiable and $P_s(\mathbf n,r)$ is not even smooth. Nonetheless, subsets $Q_s(\mathbf{n},r)$ and $R_s(\mathbf n, r)$ of $P_s(\mathbf n,r)$ have sufficiently nice geometric structures, which are essential to the analysis of Algorithm~\ref{algo}. This subsection is devoted to obtain these refined geometric structures. To do this, we first fix some notations. We denote by $\mathbb R_{\ast}$ (resp. $\mathbb R_+$, $\mathbb R_{++}$) the set of nonzero (reps. nonnegative, positive) real numbers and we define the following sets:
\begin{align}
V_{\mathbf n,r}& \coloneqq \V(r,n_1)\times\dots\times\V(r,n_s)\times \operatorname{B}(r,n_{s+1})\times\dots\times\operatorname{B}(r,n_k)\times\mathbb R^r, \label{eq:set-vnr} \\
W_{\mathbf n,r} & \coloneqq \V(r,n_1)\times\dots\times\V(r,n_s)\times \OB(r,n_{s+1})\times\dots\times\OB(r,n_k)\times\mathbb R^r, \label{eq:set-wnr} \\
W_{\mathbf n,r,\ast} & \coloneqq \V(r,n_1)\times\dots\times\V(r,n_s)\times \operatorname{OB}(r,n_{s+1})\times\dots\times\operatorname{OB}(r,n_k)\times\mathbb R_*^r,  \label{eq:set-unr} \\
W_{\mathbf n,r,+} &\coloneqq \V(r,n_1)\times\dots\times\V(r,n_s)\times \OB(r,n_{s+1})\times\dots\times\OB(r,n_k)\times\mathbb R_+^r, \label{eq:set-wnrp} \\
W_{\mathbf n,r,++} &\coloneqq \V(r,n_1)\times\dots\times\V(r,n_s)\times \operatorname{OB}(r,n_{s+1})\times\dots\times\operatorname{OB}(r,n_k)\times\mathbb R_{++}^r. \label{eq:set-unrp}
\end{align}
Here for given positive integers $r\le n$, $\V(r,n)$ is the Stiefel manifold defined in \eqref{eq:stf}, $\B(r,n)$ is the manifold defined in  \eqref{eq:fixed-norm} and $\OB(r,n)$ is the Oblique manifold defined in \eqref{eq:oblique}.

Let $D_r$
be the group of $r\times r$ diagonal matrices with $\pm 1$ diagonal elements and $S_r$ be the permutation group on $r$ elements. We define
\[
D_{r,k} \coloneqq \bigg\lbrace (E_1,\dots,E_k)\in D_r\times\dots\times D_r\colon \prod_{i=1}^k E_i=I_r \bigg\rbrace.
\]
For any positive integers $r \le n$, the group $S_r$ naturally acts on $\mathbb{R}^{n \times r}$ by permuting columns. More precisely, we have
\begin{align*}
 \mathbb{R}^{n \times r} \times S_r   &\to \mathbb{R}^{n \times r},\\
 (U,\sigma) &\mapsto U P_{\sigma},
\end{align*}
where $P_\sigma$ is the $r\times r$ permutation matrix determined by $\sigma\in S_r$. In particular, the action of $S_r$ on $\mathbb{R}^{n \times r}$ induces an action on $\V(r,n)$ and $\B(r,n)$ respectively.

{We consider the group homomorphism $\eta:S_r \to \Aut(D_{r,k})$ defined by
\[
\eta(\sigma) (E_1,\dots, E_k) \coloneqq (P_{\sigma} E_1 P_{\sigma^{-1}},\dots, P_{\sigma} E_k P_{\sigma^{-1}}),
\]
where $\Aut(D_{r,k})$ is the group of automorphisms of $D_{r,k}$. This homomorphism defines the semi-direct product $D_{r,k} \rtimes_{\eta} S_r$  \cite{L-02}.

It follows that the semi-direct product ${D_{r,k}} \rtimes_{\eta} S_r$ acts on $V_{\mathbf{n},r}$ via
\begin{align*}
V_{\mathbf{n},r} \times \left( {D_{r,k}}  \rtimes_{\eta} S_r \right)  &\to V_{\mathbf{n},r} \\
( (U,\lambda), (E,\sigma) ) &\mapsto (U,\lambda)\cdot (E,\sigma) \coloneqq  (U^{(1)}E_1P_{\sigma},\dots,U^{(k)}E_kP_{\sigma}, \lambda_{\sigma(1)},\dots, \lambda_{\sigma(r)}),
\end{align*}
where $(U,\lambda)= \left( U^{(1)},\dots, U^{(k)}, (\lambda_1,\dots,\lambda_r) \right) \in V_{\mathbf{n},r}$ and $(E,\sigma) = \left( E_1,\dots, E_k,\sigma \right) \in  {D_{r,k}} \rtimes_{\eta} S_r$.
}
\begin{proposition}\label{prop:smooth}
Let ${k\geq 3},n_1,\dots,n_k, r\le \min \{n_1,\dots,n_k\}, 1\le s \le k$ be positive integers and let $\mathbf{n} = (n_1,\dots,n_k)$. The map
\begin{align*}
\varphi_{\mathbf{n},r}: V_{\mathbf n,r} &\to {P_s(\mathbf{n},r)\subsetneq\mathbb{R}^{n_1} \otimes \cdots \otimes \mathbb{R}^{n_k}},\\
 (U^{(1)},\dots, U^{(k)}, (\lambda_1,\dots,\lambda_r)) &\mapsto   (U^{(1)},\dots,U^{(k)})\cdot \operatorname{diag}_k(\lambda_1,\dots, \lambda_r)
\end{align*}
is a smooth surjective map and we have the following:
\begin{enumerate}[label=(\roman*)]
\item {The group ${D_{r,k}}  \rtimes_{\eta}  S_r$ acts on $V_{\mathbf{n},r}$ and $\varphi_{\mathbf{n},r}$ is $D_{r,k}\rtimes_{\eta} S_r$-invariant.} \label{prop:smooth:item1}
\item Each tensor in $Q_s(\mathbf n,r)$ is identifiable and we have $\varphi_{\mathbf{n},r}^{-1} (Q_s(\mathbf n, r)) = W_{\mathbf{n},r}$ and $\varphi_{\mathbf{n},r}^{-1} (R_s(\mathbf n, r)) = W_{\mathbf n,r,\ast}$. \label{prop:smooth:item2}
\item $W_{\mathbf{n},r}$, $W_{\mathbf n,r,\ast}$ and  $W_{\mathbf n,r,++}$ are open submanifolds of $V_{\mathbf{n},r}$; $W_{\mathbf n,r,+}$ is a submanifold of $V_{\mathbf{n},r}$ with boundary. \label{prop:smooth:item3}
\item $W_{\mathbf{n},r}$, $W_{\mathbf n,r,\ast}$, $W_{\mathbf n,r,+}$, $W_{\mathbf n,r,++}$ are invariant under the induced action of $D_{r,k}\rtimes_{\eta} S_r$ on $V_{\mathbf{n},r}$ and in particular, $W_{\mathbf n,r,++}$ is a principal $D_{r,k}\rtimes_{\eta} S_r$-bundle on $R_s(\mathbf{n},r)$, i.e.,
\[
W_{\mathbf n,r,++}/ \left(D_{r,k}\rtimes_{\eta} S_r \right) \simeq R_s(\mathbf{n},r).
\]
\label{prop:smooth:item4}
\item $R_s(\mathbf n, r)$ is a smooth manifold of dimension
\[
d_{\mathbf{n},r} \coloneqq r\left( \sum_{i=1}^k n_i-{\frac{s(r-1)}{2}-k}+1\right).
\]
\label{prop:smooth:item5}
\item $Q_s(\mathbf{n},r) = \bigsqcup_{t = 0}^r R_s(\mathbf{n},t)$ is a dense subset of $P_s(\mathbf n,r)$ and $Q_s(\mathbf{n},r-1) = \bigsqcup_{t = 0}^{r-1} R_s(\mathbf{n},t)$ is the singular locus of $Q_s(\mathbf{n},r)$. \label{prop:smooth:item6}
\end{enumerate}
\end{proposition}

\begin{proof}
The smoothness and surjectivity of $\varphi_{\mathbf{n},r}$ follows directly from the definition of $V_{\mathbf{n},r}$ and $P_s(\mathbf{n},r)$. The rest of the proof is arranged item by item.
\begin{enumerate}[label=(\roman*)]
\item The invariance of $\varphi_{\mathbf{n},r}$ under the action of ${D_{r,k}}  \rtimes_{\eta} S_r$ is straightforward to verify.

\item The identifiability of a tensor in $Q_s(\mathbf{n},r)$ follows from Proposition~\ref{prop:unique}. By definition of $Q_s(\mathbf n, r)$ and $R_s(\mathbf n, r)$, we may conclude $\varphi_{\mathbf{n},r}^{-1} (Q_s(\mathbf n, r)) = W_{\mathbf{n},r}$ and $\varphi_{\mathbf{n},r}^{-1} (R_s(\mathbf n, r)) = W_{\mathbf{n},r,\ast}$.

\item It is direct to verify that $W_{\mathbf n,r}, W_{\mathbf n,r,\ast}$ are open subsets of $V_{\mathbf{n},r}$ and hence they are open submanifolds of $V_{\mathbf{n},r}$. We notice that {$\mathbb{R}_+ = [0,\infty)$} is a submanifold of $\mathbb{R}$ with boundary $\{0\}$. This implies that $W_{\mathbf n,r,+} $ is a submanifold of $V_{\mathbf{n},r}$ with boundary $\partial W_{\mathbf n,r,+} $ consisting of points $(U^{(1)},\dots, U^{(k)},(\lambda_1,\dots,\lambda_r)) \in {W_{\mathbf n,r,+}}$ such that $\prod_{i=1}^r \lambda_i = 0$.

\item The $D_{r,k}\rtimes_{\eta} S_r$-invariance of $W_{\mathbf{n},r}$, $W_{\mathbf n,r,\ast}$, $W_{\mathbf n,r,+}$, $W_{\mathbf n,r,++}$ is trivial to verify. Hence it is left to prove that $W_{\mathbf n,r,++}$ is a principal $D_{r,k}\rtimes_{\eta} S_r$-bundle on $R_s(\mathbf{n},r)$. To that end,  we first observe that the action of $D_{r,k}\rtimes_{\eta} S_r$ on $V_{\mathbf{n},r}$ induces a free action on $W_{\mathbf n,r,++}$, i.e., $(U,\lambda) \cdot (E,\sigma) = (U,\lambda)$ for some $(U,\lambda) \in W_{\mathbf n,r,++}$ implies that $(E,\sigma) = (I_r,\dots, I_r,e)\in {D_{r,k}}  \rtimes_{\eta} S_r$ where $e$ denotes the identity element in $S_r$. Indeed, since $(U,\lambda) = (U^{(1)},\dots, U^{(k)},(\lambda_1,\dots, \lambda_r)) \in W_{\mathbf n,r,++}$, column vectors of $U^{(i)}$ are linearly independent for each $i=1,\dots, k$. If $(U,\lambda) \cdot (E,\sigma) = (U,\lambda)$ holds, then we must have
\[
{
U^{(i)} E_i P_{\sigma} = U^{(i)},\quad i =1,\dots, k}
\]
and the linear independence of column vectors of $U^{(i)}$ forces $E_i = I_r$ and $\sigma = e$. Moreover, it is straightforward to verify that the action of $D_{r,k}\rtimes_{\eta} S_r$ on any fiber of $\varphi_{\mathbf{n},r}: W_{\mathbf{n},r,++} \to R_s(\mathbf{n},r)$ is also free and transitive. This implies that $W_{\mathbf n,r,++}$ is a principal $D_{r,k}\rtimes_{\eta} S_r$-bundle on $R_s(\mathbf{n},r)$.
\item
Since $D_{r,k}\rtimes_{{\eta}} S_r$ is a finite group, the dimension of $R_s(\mathbf{n},r) \simeq  W_{\mathbf n,r,++}/ (D_{r,k}\rtimes_{\eta} S_r)$ is the same as $\dim W_{\mathbf n,r,++} = \dim V_{\mathbf{n},r}$, which can be computed by
\begin{align*}
\dim (V_{\mathbf{n},r}) &= \sum_{i=1}^s \dim (\V(r,n_i) )+\sum_{j=s+1}^k\dim( \OB(r,n_j)) + \dim (\mathbb{R}^r) \\
&=\sum_{i=1}^s \left( r(n_i-r) + \binom{r}{2} \right) +  \sum_{j=s+1}^k \left( rn_j {-r}\right) + r \\
&= r\left( \sum_{i=1}^k n_i-{\frac{s(r-1)}{2}-k}+1\right).
\end{align*}

\item According to Corollary~\ref{cor:Q_s}, $Q_s(\mathbf n,r)$ is a dense subset of $P_s(\mathbf n,r)$ containing an open subset even in Zariski topology, hence it is in particular a dense subset of $P_s(\mathbf n,r)$ in the Euclidean topology. In the following, we show that $Q_s(\mathbf n,r-1) = \bigsqcup_{t = 0}^{r-1} R_s(\mathbf{n},t)$ is the singular locus of $Q_s(\mathbf n,r)$. To check that $Q_s(\mathbf n,r-1)$ is the singular locus of $Q_s(\mathbf{n},r)$, we calculate the dimension of $\operatorname{T}_{Q_s(\mathbf{n},r)}(\mathcal{B}) $ for $\mathcal{B} \in Q_s(\mathbf n,r-1)$. Without loss of generality, we may suppose that $\mathcal{B} \in R_s(\mathbf{n},r-1)$. If $r = 1$ then $\mathcal{B} = 0$ and it is the vertex and hence is the singularity of the cone $Q_s(\mathbf{n},1) =R_s(\mathbf{n},1) \bigsqcup \{ 0 \} \subsetneq \mathbb{R}^{n_1} \otimes \cdots \otimes \mathbb{R}^{n_k}$.

Next we assume that $r \ge 2$. A curve $\mathcal{B}(t)$ in $Q_s(\mathbf{n},r)$ such that $\mathcal{B}(0) = \mathcal{B}$ can be written as
\[
\mathcal{B}(t) = \left( U^{(1)}(t),\dots,U^{(k)}(t) \right) \cdot \operatorname{diag}_k(\lambda_1(t),\dots, \lambda_r(t)),\quad t\in (-\epsilon,\epsilon)
\]
for some small $\epsilon > 0$. We may also assume that {$\lambda_1(0)\geq \cdots\geq  \lambda_r(0)\ge 0$}. Here we must have $\lambda_r(0) = 0$. We observe that if $\mathcal{B}(t) \in R_s(\mathbf{n},r-1)$ for all $t\in (-\epsilon,\epsilon)$, i.e., $\lambda_r(t) \equiv 0$, then $\mathcal{B}'(0) \in \operatorname{T}_{R_s(\mathbf{n},r-1)}(\mathcal{B}) \subseteq \operatorname{T}_{Q_s(\mathbf{n},r)}(\mathcal{B})$. Moreover, if the curve has the form $\mathcal{B}(t) = \mathcal{B}_0(t) + \mathcal{B}_1(t) \in R_s(\mathbf{n},r) \subseteq Q_s(\mathbf{n},r)$ where $\mathcal{B}_0(t) \in R_s(\mathbf{n},r-1)$,
$\mathcal{B}_0(0) = \mathcal{B}$, $\mathcal{B}_1(t) = \lambda_r(t) \mathbf u^{(1)}_r(t) \otimes \cdots \otimes \mathbf u^{(k)}_r(t)\ne 0$ if $t \ne 0$ and $\mathcal{B}_1(0) = 0$, then $\mathcal{B}'(0)\in \operatorname{T}_{Q_s(\mathbf{n},r)}(\mathcal{B})$ satisfies
\[
\mathcal{B}'(0) = \mathcal{B}_0'(0) + \mathcal{B}_1'(0) \in  \operatorname{T}_{R_s(\mathbf{n},r-1)}(\mathcal{B}) + \operatorname{T}_{Q_s(\mathbf{n},1)} (0).
\]
It is also clear that we may vary the component $\mathcal{B}_1(t)$ of $\mathcal{B}(t)$ (and accordingly $\mathcal{B}_0(t)$) such that $ \mathcal{B}_1'(0)$ runs through the whole space $\operatorname{T}_{Q_s(\mathbf{n},1)} (0)$. This implies that
\[
\operatorname{T}_{Q_s(\mathbf{n},1)} (0) \hookrightarrow \operatorname{T}_{Q_s(\mathbf{n},r)}(\mathcal{B})/\operatorname{T}_{R_s(\mathbf{n},r-1)}(\mathcal{B}).
\]
If $\mathcal{B}$ is a smooth point of $Q_s(\mathbf{n},r)$, then $\operatorname{T}_{Q_s(\mathbf{n},r)}(\mathcal{B})$ is a vector space and $\dim\big(\operatorname{T}_{Q_s(\mathbf{n},r)}(\mathcal B)\big)= d_{\mathbf{n},r}$.
However, we also have
\[
\dim\big(\operatorname{T}_{Q_s(\mathbf{n},1)}(0)\big) + \dim\big(\operatorname{T}_{R_s(\mathbf{n},r-1)}(\mathcal{B})\big) \ge d_{\mathbf{n},1} + d_{\mathbf{n},r-1} > d_{\mathbf{n},r},
\]
which is a contradiction. Therefore we may conclude that $Q_s(\mathbf n,r-1)$ is the singular locus of $Q_s(\mathbf{n},r)$.
\end{enumerate}
\end{proof}

For reader's convenience, we summarize in Figure~\ref{fig:relations} relations among various sets discussed in this subsection.
\begin{figure}
  \begin{tikzcd}
      W_{\mathbf{n},r,++}  \arrow[d,symbol=\subseteq] \arrow[r,symbol=\subseteq] & W_{\mathbf{n},r,+} \arrow[d,symbol=\subseteq] &  \\
    W_{\mathbf{n},r,\ast}  \arrow{d}{\varphi_{\mathbf{n},r}} \arrow[r,symbol=\subseteq]{good} & W_{\mathbf{n},r} \arrow{d}{\varphi_{\mathbf{n},r}} \arrow[r,symbol=\subseteq] & V_{\mathbf{n},r} \arrow[r,symbol=\subseteq] \arrow{d}{\varphi_{\mathbf{n},r}} & \mathbb{R}^{n_1\times r}\times \cdots \times \mathbb{R}^{n_k \times r} \times \mathbb{R}^r \arrow{d}{\varphi_{\mathbf{n},r}} \\
     R_{s}(\mathbf{n},r) \arrow[r,symbol=\subseteq] \arrow[d,symbol=\simeq]   & Q_{s}(\mathbf{n},r)  \arrow[r,symbol=\subseteq] \arrow[d,symbol=\simeq]  & P_{s}(\mathbf{n},r) \arrow[r,symbol=\subseteq] & \mathbb{R}^{n_1} \otimes \cdots \otimes \mathbb{R}^{n_k}  \\
  W_{\mathbf n,r,++}/ (D_{r,k}\rtimes_{\eta} S_r) \arrow[r,symbol=\subseteq] & \bigsqcup_{t=1}^r W_{\mathbf n,t,++}/ (D_{t,k}\rtimes_{\eta} S_t) \\
  \end{tikzcd}
  \caption{Relations among tensor spaces and their parameter spaces}
    \label{fig:relations}
\end{figure}


\subsection{Geometry of the projection}\label{sec:projection}
We recall that problem~\eqref{eq:original} can be viewed as a constrained optimization problem in $\mathbb{R}^{N}$ where $N = \prod_{j=1}^k n_j$ and hence it is natural to solve problem~\eqref{eq:original} by methods for constrained optimization problems such as the Langrange multiplier method. On the one hand, Lemmas~\ref{lem:P_s>1}, \ref{lem:P_s=1} and Proposition~\ref{prop:smooth}-\eqref{prop:smooth:item5} imply that these methods are not practical. On the other hand, these methods are designed for KKT points \cite{RW-98} while Proposition~\ref{prop:smooth}-\eqref{prop:smooth:item6} indicates that a minimizer of problem~\eqref{eq:original} is not necessarily be a KKT point. Putting these issues aside, we briefly discuss  in this subsection some properties of KKT points of \eqref{eq:original}, which might be of independent interest.

Let $Y$ be a real algebraic variety in $\mathbb{R}^n$ and let $f(\x,\y)$ be the restriction of a differentiable function on $\mathbb{R}^n \times \mathbb{R}^n$ to $\mathbb{R}^n \times Y$. We suppose that the ideal $I(Y)$ of $Y$ is generated by polynomials $f_1,\dots, f_s \in \mathbb{R}[x_1,\dots,x_n]$. Let $\y_0\in Y$ and let
$J_{\y_0}(f_1,\dots, f_s)$ be the Jacobian matrix defined by
\[
J_{\y_0}(f_1,\dots, f_s) = \begin{bmatrix}
\frac{\partial f_1}{\partial x_1}(\y_0) &  \cdots & \frac{\partial f_1}{\partial x_n}(\y_0) \\
\vdots & \ddots & \vdots  \\
\frac{\partial f_s}{\partial x_1}(\y_0) & \cdots & \frac{\partial f_s}{\partial x_n}(\y_0) \\
\end{bmatrix} \in \mathbb{R}^{s \times n}.
\]
We denote by $\ker (J_{\y_0}(f_1,\dots, f_s))\subseteq \mathbb{R}^n$ the right null space of $J_{\y_0}(f_1,\dots, f_s)$, which can be identified with the \emph{Zariski tangent space} $\Tt_Y (\y_0)$ of $Y$ at $\y_0$\cite[Proposition~2.5]{L-02-A}.
Given a point $\x\in \mathbb{R}^n$, a point $\y_0\in Y$ is called a \emph{critical projection} of $\x$ to $Y$ (with respect to $f$) if
\begin{equation}\label{lem:distance minimizer:relation}
\langle \ker (J_{\y_0}(f_1,\dots, f_s)), \nabla_\y f(\x,\y_0) \rangle = 0.
\end{equation}
\begin{lemma}\label{lem:distance minimizer}
Let $Y_{\text{sm}}\subseteq Y$ be the smooth locus of $Y$. For a fixed $\x\in \mathbb{R}^n$, if $\y_0\in Y_{\text{sm}}$ is a minimizer of the function $f(\x,\cdot): Y\to \mathbb{R}$, then $\y_0$ is a critical projection of $\x$ to $Y$ with respect to $f$.
\end{lemma}

\begin{proof}
We consider the following optimization problem:
\begin{equation}\label{lem:distance minimizer:op problem}
\begin{array}{rl}\min & f(\x,\y) \\
 \text{s.t.} & f_1(\y) = \cdots = f_s(\y) = 0.
\end{array}
\end{equation}
Since $\y_0$ is a solution to \eqref{lem:distance minimizer:op problem} and $\y_0$ is a smooth point of $Y$, then $\nabla_\y f(\x,\y_0)$ must be orthogonal to the tangent space of $Y$ at $\y_0$. Thus, the desired relation \eqref{lem:distance minimizer:relation} is satisfied.
\end{proof}
Note that once the condition \eqref{lem:distance minimizer:relation} is satisfied,
there exist multipliers $\mu_1,\dots,\mu_s\in \mathbb{R}$ such that $\y_0$ is a critical point of $f(\x,\y) + \sum_{j=1}^s \mu_s f_s(\y)$, from which we can conclude that $\nabla_\y f(\x,\y_0)$ is a linear combination of $\nabla_\y f_1(\y_0),\dots, \nabla_\y f_s(\y_0)$ and hence the KKT condition holds for \eqref{lem:distance minimizer:op problem}. Conversely, if $y_0$ is a KKT point of \eqref{lem:distance minimizer:op problem}, then \eqref{lem:distance minimizer:relation} must be satisfied. Thus, critical projections are exactly KKT points of problem \eqref{lem:distance minimizer:op problem}. It is also notable that linear independence of $\nabla_\y f_1(\y_0),\dots, \nabla_\y f_s(\y_0)$ implies that $\y_0\in Y_\text{sm}$ but the converse may not be true.


According to Proposition~\ref{prop:smooth locus}--\ref{prop:smooth locus:item2}, if $\y\in Y_{\text{sm}}$, then  $\Tt_Y (\y)$ coincides with the tangent space $\Tt_{Y_\text{sm}} (\y)$ of the manifold $Y_\text{sm}$ at $\y$.
Geometrically, Lemma~\ref{lem:distance minimizer} means that if $\y\in Y_\text{sm}$ is a minimizer of $f(\x,\cdot)$, then the gradient $\nabla_\y f(\x,\y)$ is perpendicular to the Zariski tangent space $\Tt_Y (\y)\coloneqq \ker (J_{\y}(f_1,\dots, f_s))$ of $Y$ at $\y$. Moreover, the Jacobian criterion also implies that if $\y$ belongs to the singular locus $Y_{\operatorname{sing}}$ of $Y$, then $\dim \Tt_Y (\y) > \dim Y = \dim Y_{\text{sm}}$.

\begin{corollary}\label{cor:distance minimizer}
For a fixed $\x\in \mathbb{R}^n$, a point $\y_0\in Y_{\text{sm}}$ satisfies \eqref{lem:distance minimizer:relation} if and only if the Riemannian gradient of $f(\x,\cdot)$ on $Y_{\text{sm}}$ vanishes at $\y_0$.
\end{corollary}

\begin{lemma}\label{lem:projection}
Let $Z$ be a closed proper subvariety of $Y$ and let $f$ be the restriction of the squared distance function to $\mathbb{R}^n \times Y$. For a generic $\x\in \mathbb{R}^n$, any critical projection $\y$ of $\x$ to $Y$ lies in $Y\setminus Z$. In other words, the set of $\x \in \mathbb{R}^n$ such that some critical projection of $\x$ lies in $Z$ is contained in a closed proper subvariety of $\mathbb{R}^n$.
\end{lemma}
\begin{proof}
We prove by contradiction. Suppose that there exists a dense subset $U\subseteq \mathbb{R}^n$ such that for any $\x\in U$,
\[
\{\y\in Y: \langle \ker (J_{\y}(f_1,\dots, f_s)), \x-\y\rangle = 0\} \cap Z \ne \emptyset.
\]
We consider the following set
\[
X \coloneqq \{\z + \vv\in \mathbb{R}^n: \z\in Z,\ \langle \ker (J_{\z}(f_1,\dots, f_s)), \vv \rangle = 0\},
\]
which is the image of $W \coloneqq \{(\z,\vv):\z\in Z, \langle \ker (J_{\z}(f_1,\dots, f_s)), \vv \rangle = 0\}$ under the polynomial map
\begin{align*}
\alpha: Z \times \mathbb{R}^n &\to \mathbb{R}^n,\\
(\z,\vv)&\mapsto \alpha(\z,\vv) = \z+ \vv.
\end{align*}

It is clear from the construction that $U\subseteq X$ and hence $\overline{X} = \mathbb{R}^n$. Since $W$ is a projection of an algebraic variety (hence it is a semi-algebraic set) and $\alpha$ is a polynomial map, Proposition~\ref{thm:image dimension} implies that $X = \alpha (W)$ is semi-algebraic. Moreover, according to Proposition~\ref{prop:closure dimension}, we may conclude that $\dim X = \dim \overline{X}  = n$.
However, we notice that
\[
\dim X \le \dim Z + \max_{\z\in Z} \{ \rank J_\z (f_1,\dots, f_s)\}  \le \dim Z + (n - \dim Y) < n,
\]
from which we obtain a contradiction. Here the first inequality follows from a simple dimension counting, the second is due to the fact that
\[
n - \rank J_\z(f_1,\dots, f_s) = \dim \Tt_Y(\z)\ge \dim Y,
\]
and the last inequality is obtained from the assumption that $Z$ is a proper closed subvariety of $Y$.
\end{proof}

\begin{theorem}\label{thm:projection}
For $k\ge 3$, $k\ge s \ge 1$, $\mathbf{n} = (n_1,\dots, n_k)$ and $r \le \min\{n_1,\dots, n_k\}$, any critical projection of a generic tensor $\T\in \mathbb{R}^{n_1}\otimes \cdots \otimes \mathbb{R}^{n_k}$ on $P_s(\mathbf{n},r)$ lies in $R_s(\mathbf{n},r)$.
\end{theorem}

\begin{proof}
Let $Y = P_s(\mathbf{n},r) \subseteq \mathbb{R}^{n_1} \otimes \cdots \otimes \mathbb{R}^{n_k}$ and let $Z =P_s(\mathbf{n},r) \setminus  R_s(\mathbf{n},r)$. According to Theorem~\ref{thm:P_s} and Corollary~\ref{cor:Q_s}, $Y$ is a real algebraic subvariety of $\mathbb{R}^N$ where $N = \prod_{j=1}^k n_j$ and $Z$ is a proper closed subvariety of $Y$. Lemma~\ref{lem:projection} applies and this completes the proof.
\end{proof}

We conclude this subsection with an emphasis on the relation between Theorem~\ref{thm:projection} and the LRPOTA problem formulated in \eqref{eq:original}. In the lingo of optimization theory, Theorem~\ref{thm:projection} means that {generically} all KKT points of problem~\eqref{eq:original} are contained in $R_s(\mathbf{n},r)$. However, frequently used constraint qualifications such as linear independence constraint qualification (LICQ) mentioned in Section~\ref{subsec:KKT} are not easy to verify for $P_s(\mathbf{n},r)$. Hence there is no guarantee that a minimizer of problem~\eqref{eq:original} must be a KKT point and finding KKT points of \eqref{eq:original} does not completely solve our original LRPOTA problem. Nonetheless, as we have mentioned at the beginning of Section~\ref{subsec:lowrank}, we may parametrize $P_s(\mathbf{n},r)$ by a smooth submanifold $V_{\mathbf{n},r} \subseteq \mathbb{R}^{n_1 \times r} \times \cdots \times \mathbb{R}^{n_k \times r} \times \mathbb{R}^r$. Consequently, problem \eqref{eq:original} is reformulated as \eqref{eq:sota}, by which above issues are resolved. The goal of the next subsection is to locate all KKT points of \eqref{eq:sota}.

\subsection{Geometry of the parametrized projection}\label{subsec:geometry of parametrized projection}
We observe that $V_{\mathbf{n},r}$ is a smooth real subvariety of $\left( \prod_{i=1}^k \mathbb{R}^{{n_i}\times r} \right) \times \mathbb{R}^r$ which is defined by some polynomials $g_1,\dots, g_{{t}}$ on $ \left( \prod_{i=1}^k \mathbb{R}^{{n_i}\times r} \right) \times \mathbb{R}^r$ where ${t} = \operatorname{codim} V_{\mathbf{n},r}$. For a given tensor $\mathcal{A} \in \mathbb{R}^{n_1} \otimes \cdots \otimes \mathbb{R}^{n_k} $, the function $g: V_{\mathbf{n},r} \to \mathbb{R}_+$ defined by
\[
g(y):=\frac{1}{2} \lVert \mathcal{A} - \varphi_{\mathbf{n},r} (y) \rVert^2
\]
is obviously a differentiable function, where $\varphi_{\mathbf{n},r}: V_{\mathbf{n},r} \to \mathbb{R}^{n_1} \otimes \cdots \otimes \mathbb{R}^{n_k}$ is the smooth map defined in Proposition~\ref{prop:smooth}. Here we denote a point $(U^{(1)},\dots, U^{(k)}, {(\lambda_1,\dots, \lambda_r)})$ in $V_{\mathbf{n},r}$ simply by $y$ when specific components of $y$ are not involved in the context.
\begin{lemma}\label{lem:KKT parameter}
A point $y_{\ast} \in V_{\mathbf{n},r}$ is a KKT point of the optimization problem
\begin{equation}\label{lem:KKT parameter:optproblem}
\begin{array}{rl}
\min & g(y)  \\
\text{s.t.}& y \in V_{\mathbf{n},r}
\end{array}
\end{equation}
if and only if $y_\ast$ satisfies the equation
\begin{equation}\label{eq:incidence}
\langle \mathcal{A} - \varphi_{\mathbf{n},r} (y_\ast), d_{{y_\ast}}\varphi_{\mathbf{n},r} (\Tt_{V_{\mathbf{n},r}} (y_\ast)) \rangle  = 0,
\end{equation}
where $d_{y_\ast} \varphi_{\mathbf{n},r}$ is the differential of $\varphi_{\mathbf{n},r}$ at the point $y_\ast$.
\end{lemma}

Before we proceed to the proof, we remind the reader that problem \eqref{lem:KKT parameter:optproblem} is our LRPOTA problem \eqref{eq:sota} written in a more compact form.
\begin{proof}
We recall that if $V_{\mathbf{n},r}$ is defined by polynomials $g_1,\dots, g_{t}$, then according to Lemma~\ref{lem:distance minimizer}, KKT points of \eqref{lem:KKT parameter:optproblem} are characterized by the equation
\begin{equation}\label{lem:KKT parameter:eqn1}
\langle \ker (J_y (g_1,\dots, g_t)), \nabla_y g \rangle = 0,
\end{equation}
since $V_{\mathbf{n},r}$ is smooth.

Considering that $g(y) =\frac{1}{2} \lVert \mathcal{A} - \varphi_{\mathbf{n},r} (y) \rVert^2$, by the chain rule we have
\[
\nabla_y g = - (\mathcal{A} - \varphi_{\mathbf{n},r}(y)) \cdot J_y(\varphi_{\mathbf{n},r}),
\]
from which we may rewrite \eqref{lem:KKT parameter:eqn1} as
 \begin{equation}\label{lem:KKT parameter:eqn2}
\langle \ker (J_y (g_1,\dots, g_t)), (\mathcal{A} - \varphi_{\mathbf{n},r}(y)) \cdot J_y(\varphi_{\mathbf{n},r})\rangle = 0.
\end{equation}
{We may further rewrite \eqref{lem:KKT parameter:eqn2} as}
 \begin{equation}\label{lem:KKT parameter:eqn3}
\langle \mathcal{A} - \varphi_{\mathbf{n},r}(y),  J_y(\varphi_{\mathbf{n},r}) \cdot \ker (J_y (g_1,\dots, g_t))  \rangle = 0.
\end{equation}
We notice that $V_{\mathbf{n},r}$ is a smooth complete intersection of $g_1,\dots, g_t$ {and hence $\ker (J_y (g_1,\dots, g_t))  = \Tt_{V_{\mathbf{n},r}} (y)$}, which implies that $J_y(\varphi_{\mathbf{n},r}) \cdot \ker (J_y (g_1,\dots, g_t))$ can be re-interpreted as $d_{{y}}\varphi_{\mathbf{n},r} (\Tt_{V_{\mathbf{n},r}} (y))$ and this completes the proof.
\end{proof}

Condition \eqref{eq:incidence} determines an incidence variety in the Cartesian product of the tensor space and the parameter space. In the following, we study this incidence variety and characterize the linear subspace $d_{{y}}\varphi_{\mathbf{n},r} (\Tt_{V_{\mathbf{n},r}} (y)) \subseteq \mathbb{R}^{n_1} \otimes \cdots \otimes \mathbb{R}^{n_k}$. To do that, we first recall from \eqref{eq:set-vnr} that
\[
V_{\mathbf n,r} = \V(r,n_1)\times\dots\times\V(r,n_s)\times \operatorname{B}(r,n_{s+1})\times\dots\times\operatorname{B}(r,n_k)\times\mathbb R^r.
\]
Hence the tangent space $\Tt_{V_{\mathbf{n},r}} (y)$ is spanned by bases of the following vector spaces:
\[
\Tt_{\V(r,n_1)} (U^{(1)}), \dots, \Tt_{\V(r,n_s)} (U^{(s)}), \Tt_{\B(r,n_{s+1})} (U^{(s+1)}), \dots, \Tt_{\B(r,n_{k})} (U^{(k)}),  \Tt_{\mathbb{R}^r} ( (\lambda_1,\dots, \lambda_r) ).
\]
The next two lemmas are devoted to determine images of the spanning vectors described above under the map $d_y \varphi_{\mathbf{n},r}$.
According to \cite{AMS-08}, {for $A\in\V(r,n)$}, we have
\[
\Tt_{\V(r,n)}(A) =A^\perp\cdot \mathbb{R}^{(n-r)\times r} + A \cdot \operatorname{Sk}^r,
\]
where $A^{
\perp}$ is any $n\times (n-r)$ matrix whose column vectors are orthonormal to those of $A$ and $\operatorname{Sk}^r:=\{B\in\mathbb{R}^{r\times r}\colon B=-B^\tp\}$ is the space of $r\times r$ skew-symmetric matrices.
\begin{lemma}\label{lem:differential phi}
Let $y = (U^{(1)},\dots, U^{(k)}, {(\lambda_1,\dots, \lambda_r)})$ be a point in $V_{\mathbf{n},r}$. For $1\le i \le s$ (resp. $s+1 \le i \le k$) and $1\le j \le r$, we denote by $\mathbf{x}^{(i)}_j$ a vector in $\mathbb{R}^{n_i}$ orthogonal to $\{ \mathbf{u}^{(i)}_1,\dots, \mathbf{u}^{(i)}_r\}$ (resp. $\mathbf{u}^{(i)}_j$).
{Let $z^{(i)}_j\in\mathbb{R}^{n_1\times r}\times\dots\times\mathbb{R}^{n_k\times r}\times\mathbb R^r$ be a vector whose components are all zero, except the $j$-th column of the $i$-th factor matrix that equals $\x^{(i)}_j$.}
Then
\begin{align*}
d_y \varphi_{\mathbf{n},r} ({z^{(i)}_j}) &= \lambda_j \mathbf{u}^{(1)}_j \otimes \mathbf{u}^{(i-1)}_j \otimes \mathbf{x}^{(i)}_j \otimes \mathbf{u}^{(i+1)}_j \otimes \cdots \otimes \mathbf{u}^{(k)}_j, \\
d_y \varphi_{\mathbf{n},r} (0,\dots, 0, {\ee_j}) &= \mathbf{u}^{(1)}_j \otimes \cdots \otimes \mathbf{u}^{(k)}_j.
\end{align*}
Here for each $1\le j \le r$, $\ee_j \in \mathbb{R}^r$ denotes the vector with a $1$ in the $j$-th coordinate and $0$'s elsewhere, i.e.,
$\ee_j = (0,\dots, 0, 1 ,0,\dots, 0)^\tp$.
\end{lemma}
\begin{proof}
Without loss of generality, we may assume that $j = r$. We consider the following curves in $V_{\mathbf{n},r}$:
\begin{align*}
c_1(t) &:= (U^{(1)},\dots,U^{(i-1)}, [\widehat{U}^{(i)}, \mathbf{y}^{{(i)}}_r(t)],U^{(i+1)},\dots, U^{(k)},(\lambda_1,\dots, \lambda_r)), \\
c_2(t) &:= (U^{(1)},\dots, U^{(k)},(\lambda_1,\dots,\lambda_{j-1}, \lambda_j + t,\lambda_{j+1},\dots, \lambda_r)),
\end{align*}
where $\widehat{U}^{(i)}$ is the submatrix of ${U}^{(i)}$ obtained by taking the first $(r-1)$ columns and
\begin{enumerate}[label=(\roman*)]
\item If $1\le i \le s$, $\mathbf{y}^{(i)}_r(t)$ is a curve in $\left(\operatorname{span}(\mathbf{u}^{(i)}_1,\dots, \mathbf{u}^{(i)}_{r-1})\right)^\perp$ such that $\lVert \mathbf{y}^{(i)}_r(t) \rVert \equiv 1$, $\mathbf{y}^{(i)}_r(0) = \mathbf{u}^{(i)}_r$ and $\dot{\mathbf{y}}^{(i)}_r(0) = \mathbf{x}^{(i)}_r$.
\item If $s+1 \le i \le k$, $\mathbf{y}^{(i)}_r(t)$ is a curve in $\mathbb{R}^{n_i}$ such that $\lVert \mathbf{y}^{(i)}_r(t) \rVert \equiv 1$, $\mathbf{y}^{(i)}_r(0)= \mathbf{u}^{(i)}_r$ and $\dot{\mathbf{y}}^{(i)}_r(0) = \mathbf{x}^{(i)}_r$.
\end{enumerate}
It is obvious that $c_1(0) = c_2(0) = y$ and
\begin{align*}
\dot{c}_1(0) &= {z^{(i)}_j}, \\
\dot{c}_2(0) &= (0,\dots, 0, \ee_j).
\end{align*}
Using the formula $d_y \varphi_{\mathbf{n},r} (\dot{c}_i(0)) = \frac{d \varphi_{\mathbf{n},r}(c_i(t)) }{dt}\mid_{t = 0}$, we obtain the desired {expressions}.
\end{proof}

\begin{lemma}\label{lem:differential phi 1}
Let $y = (U^{(1)},\dots, U^{(k)},(\lambda_1,\dots, \lambda_r))$ be a point in $V_{\mathbf{n},r}$ and let $A$ be an $r\times r$ skew-symmetric matrix. For $1\le i \le s$, we have $U^{(i)} A \in \Tt_{\V(r,n_i)} ( U^{(i)} )$ and
\[
d_y \varphi_{\mathbf{n},r} (0,\dots, 0, U^{(i)}A, 0 ,\dots, 0, (0,\dots, 0) ) = \sum_{j=1}^{{r}} \lambda_j \mathbf{u}^{(1)}_j \otimes \cdots \mathbf{u}^{(i-1)}_j  \otimes \mathbf{v}^{(i)}_j \otimes \mathbf{u}^{(i+1)}_j \otimes \cdots \otimes \mathbf{u}^{(k)}_j,
\]
where $\mathbf{v}^{(i)}_j$ is the $j$-th column vector of $U^{(i)}A$.
\end{lemma}
\begin{proof}
We notice that the tangent vector $(0,\dots, 0,  U^{(i)}A, 0 ,\dots, 0, (0,\dots, 0) ) \in \Tt_{V_{\mathbf{n},r}} (y)$ is obtained by differentiating the curve
\[
c(t) := (U^{(1)},\dots, U^{(i-1)},  U^{(i)}\exp(tA), U^{(i+1)}, \dots, U^{(k)}, (\lambda_1,\dots, \lambda_r))
\]
at $t = 0$. The expression for $d_y \varphi_{\mathbf{n},r} (0,\dots, 0,  U^{(i)}A, 0 ,\dots, 0, (0,\dots, 0) )$ can be derived by differentiating $\varphi_{\mathbf{n},r} (c(t))$ at $t= 0$.
\end{proof}

We observe that for $1\le l < m \le r$, if $A_{lm}$ denotes the $r\times r$ skew-symmetric matrix whose entries are all zeros except the $(l,m)$-th and $(m,l)$-th, which are $1$ and $-1$ respectively, then
\[
U^{(i)} A_{lm} = \begin{bmatrix}
\mathbf{u}^{(i)}_1,\dots, \mathbf{u}^{(i)}_{{r}}
\end{bmatrix} A_{lm} = \begin{bmatrix}
0,\dots, 0, -\mathbf{u}^{(i)}_{m},0,\dots,0, \mathbf{u}^{(i)}_l, 0,\dots, 0
\end{bmatrix}.
\]
This implies that $d_y \varphi_{\mathbf{n},r} (0,\dots, 0, U^{(i)} A_{lm}, 0 ,\dots, 0, (0,\dots, 0) ) $ is equal to
\begin{equation}\label{eqn:basis Ckl}
\lambda_m \mathbf{u}^{(1)}_m \otimes \cdots \mathbf{u}^{(i-1)}_m  \otimes \mathbf{u}^{(i)}_l \otimes \mathbf{u}^{(i+1)}_m \otimes \cdots \otimes \mathbf{u}^{(k)}_m - \lambda_l \mathbf{u}^{(1)}_l \otimes \cdots \otimes \mathbf{u}^{(i-1)}_l  \otimes \mathbf{u}^{(i)}_m \otimes \mathbf{u}^{(i+1)}_l \otimes \cdots \otimes \mathbf{u}^{(k)}_l.
\end{equation}
For $1\le l < m \le r$ and $1\le i \le s$, let $\mathcal{T}^{(i)}_{lm}$ be the tensor in \eqref{eqn:basis Ckl} and let $C_{lm}$ be the set $\{\mathcal{T}^{(i)}_{lm}: 1 \le i \le s \}$.

Let $y = (U^{(1)},\dots, U^{(k)},(\lambda_1,\dots, \lambda_{{r}}))$ be a point in $V_{\mathbf{n},r}$. Let $\{ \mathbf{x}^{(i)}_{p} \}_{p= 1}^{n_i - r}$ be an orthonormal basis of $\left(\operatorname{span}( \mathbf{u}^{(i)}_1,\dots, \mathbf{u}^{(i)}_r) \right)^\perp$ for $1 \le i \le s$ and let $\{ \mathbf{x}^{(i')}_{j,p'} \}_{p'=1}^{n_{i'}-1}$ be an orthonormal basis of $\left(\operatorname{span}( \mathbf{u}^{(i')}_j) \right)^\perp$ for $s+1 \le i' \le k $.

Let $J\coloneqq \{j: \lambda_j\neq 0, j =1,\dots,r \}$. We define a subset $B\subseteq \mathbb{R}^{n_1} \otimes \cdots \otimes \mathbb{R}^{n_k}$ to be
\begin{equation}\label{eqn:basis B}
\begin{split}
\big\{
&\mathbf{u}^{(1)}_j \otimes \cdots \otimes \mathbf{u}^{(i-1)}_j \otimes \mathbf{x}^{(i)}_{p} \otimes \mathbf{u}^{(i+1)}_j \otimes \cdots \otimes \mathbf{u}^{(k)}_j, \\
& \mathbf{u}^{(1)}_j \otimes \cdots \otimes \mathbf{u}^{(i'-1)}_j \otimes \mathbf{x}^{(i')}_{j,p'} \otimes \mathbf{u}^{(i'+1)}_j \otimes \cdots \otimes \mathbf{u}^{(k)}_j, \\
& \mathbf{u}^{(1)}_{j'} \otimes \cdots \otimes \mathbf{u}^{(k)}_{j'}
\big\},
\end{split}
\end{equation}
for all $1\le i \le s$, $s+1 \le i' \le k$, $1\le p \le n_i -r$, $1 \le p' \le n_{i'}-1$, $j\in J$ and $j'\in\{1,\dots,r\}$.
\begin{lemma}\label{lem:basis dphi}
For $s\ge 2$, we have the following:
\begin{enumerate}[label=(\roman*)]
\item $B $ is an orthonormal subset of $d_y \varphi_{\mathbf{n},r} (\Tt_{V_{\mathbf{n},r}})$. In particular,
\[
\dim (\operatorname{span} (B ) ) = \Big( \big( \sum_{i=1}^k n_i \big) - sr - (k - s)  \Big)q+r,
\]
where $q$ is the cardinality of the set $J$.
\label{lem:basis dphi:item1}

\item Elements in $B$ are orthogonal to those in $C_{lm}$. \label{lem:basis dphi:item2}

\item Elements in $C_{lm}$ are orthogonal to those in $C_{l_1 m _1}$ whenever $(l,m) \ne (l_1,m_1)$. \label{lem:basis dphi:item3}

\item Elements in $C_{lm} \setminus \{0\}$ are linearly independent unless $s = 2$ and $\lambda_l = \beta \lambda_m$ and $\mathbf{u}^{j}_l = \gamma_j \mathbf{u}^{j}_m$, where $\beta,\gamma_j \in \{-1,1\}$ for $3\le j \le k$.
 \label{lem:basis dphi:item4}

\item Assume that $(\lambda_l,\lambda_m) \ne (0,0)$. Then $\dim (\operatorname{span} (C_{lm})) = s$ if either {$s\geq 3$} or $s = 2$ with an exceptional case: $\lambda_l = \pm \lambda_m$ and $\mathbf{u}^{(j)}_l = \pm \mathbf{u}_m^{(j)}$ for all $3\le j \le k$, in which $\dim (\operatorname{span} (C_{lm})) = 1$.
 \label{lem:basis dphi:item5}

\end{enumerate}
\end{lemma}

\begin{proof}

{Unless otherwise stated, all indices in the following argument range over those sets specified in \eqref{eqn:basis B} and indices in distinct notations should be understood to take distinct values.}
\begin{enumerate}[label=(\roman*)]
\item It is clear that
\[
\left( \mathbf{u}^{(1)}_j {\otimes\cdots}\otimes \mathbf{u}^{(i-1)}_j \otimes \mathbf{x}^{(i)}_{p} \otimes \mathbf{u}^{(i+1)}_j \otimes \cdots \otimes \mathbf{u}^{(k)}_j \right) \perp  \left( \mathbf{u}^{(1)}_{j'} \otimes \cdots \otimes \mathbf{u}^{(k)}_{j'} \right)
\]
and
\[
\left( \mathbf{u}^{(1)}_j {\otimes\cdots}\otimes \mathbf{u}^{(i'-1)}_j \otimes \mathbf{x}^{(i')}_{j,p} \otimes \mathbf{u}^{(i'+1)}_j \otimes \cdots \otimes \mathbf{u}^{(k)}_j \right) \perp  \left( \mathbf{u}^{(1)}_{j'} \otimes \cdots \otimes \mathbf{u}^{(k)}_{j'} \right)
\]
since $\langle \mathbf{u}^{({i})}_{j}, \mathbf{u}^{({i})}_{{j'}} \rangle = {\delta_{j{j'}}}$, $\langle \mathbf{x}^{(i)}_{p}, \mathbf{u}^{(i)}_{{j}} \rangle = 0$ and $\langle \mathbf{x}^{(i')}_{j,p}, \mathbf{u}^{(i')}_{j} \rangle = 0$. By the same reason, we also have
\footnotesize
\[
\left( \mathbf{u}^{(1)}_j {\otimes\cdots}\otimes \mathbf{u}^{(i-1)}_j \otimes \mathbf{x}^{(i)}_{p} \otimes \mathbf{u}^{(i+1)}_j \otimes \cdots \otimes \mathbf{u}^{(k)}_j  \right) \perp \left( \mathbf{u}^{(1)}_{j_1} {\otimes\cdots}\otimes \mathbf{u}^{(i'-1)}_{j_1} \otimes \mathbf{x}^{(i')}_{{j_1},p'} \otimes \mathbf{u}^{(i'+1)}_{j_1} \otimes \cdots \otimes \mathbf{u}^{(k)}_{j_1} \right)
\]
\normalsize
for all $j,j_1 \in J$.
By the choice of $\mathbf{x}^{(i')}_{j,p}$, we may conclude that
\footnotesize
\[
\left(  \mathbf{u}^{(1)}_j {\otimes\cdots}\otimes \mathbf{u}^{(i'-1)}_j \otimes \mathbf{x}^{(i')}_{j,p} \otimes \mathbf{u}^{(i'+1)}_j \otimes \cdots \otimes \mathbf{u}^{(k)}_j  \right) \perp \left(  \mathbf{u}^{(1)}_{j_1} {\otimes\cdots} \otimes \mathbf{u}^{(i'_1-1)}_{j_1} \otimes \mathbf{x}^{(i'_1)}_{{j_1},p_1} \otimes \mathbf{u}^{(i_1+1)}_{j_1} \otimes \cdots \otimes \mathbf{u}^{(k)}_{j_1}  \right).
\]
\normalsize
By the choice of $\mathbf{x}^{(i)}_{p}$, it is obvious that
\footnotesize
\[
\left( \mathbf{u}^{(1)}_j {\otimes\cdots}\otimes \mathbf{u}^{(i-1)}_j \otimes \mathbf{x}^{(i)}_{p} \otimes \mathbf{u}^{(i+1)}_j \otimes \cdots \otimes \mathbf{u}^{(k)}_j  \right) \perp \left(  \mathbf{u}^{(1)}_{j_1}{\otimes\cdots} \otimes \mathbf{u}^{(i_1-1)}_{j_1} \otimes \mathbf{x}^{(i_1)}_{p_1} \otimes \mathbf{u}^{(i_1+1)}_{j_1} \otimes \cdots \otimes \mathbf{u}^{(k)}_{j_1}  \right).
\]
\normalsize
Thus, the vectors given in \eqref{eqn:basis B} are mutually orthogonal, and hence the conclusion follows.

\item According to \eqref{eqn:basis Ckl}, for $1\le l < m \le r$, $\mathcal{T}^{(i)}_{lm} $ is simply
\small
\[
\lambda_m \mathbf{u}^{(1)}_m \otimes \cdots \mathbf{u}^{(i-1)}_m  \otimes \mathbf{u}^{(i)}_l \otimes \mathbf{u}^{(i+1)}_m \otimes \cdots \otimes \mathbf{u}^{(k)}_m - \lambda_l \mathbf{u}^{(1)}_l \otimes \cdots \mathbf{u}^{(i-1)}_l  \otimes \mathbf{u}^{(i)}_m \otimes \mathbf{u}^{(i+1)}_l \otimes \cdots \otimes \mathbf{u}^{(k)}_l.
\]
\normalsize
For notational simplicity, we denote $\mathcal{A}_m^{(i)} := \mathbf{u}^{(1)}_m \otimes \cdots \mathbf{u}^{(i-1)}_m   \otimes \mathbf{u}^{(i+1)}_m \otimes \cdots \otimes \mathbf{u}^{(k)}_m$ so that
\[
\mathcal{T}^{(i)}_{lm} \simeq \lambda_m \mathcal{A}_m^{(i)} \otimes \mathbf{u}^{(i)}_l - \lambda_l \mathcal{A}_l^{(i)} \otimes \mathbf{u}^{(i)}_m.
\]
Relations
\[
\left( \mathbf{u}^{(1)}_j \otimes \cdots \otimes \mathbf{u}^{(i-1)}_j \otimes \mathbf{x}^{(i)}_{p} \otimes \mathbf{u}^{(i+1)}_j \otimes \cdots \otimes \mathbf{u}^{(k)}_j \right) \perp \left( \lambda_m \mathcal{A}_m^{(i_1)} \otimes \mathbf{u}^{(i_1)}_l - \lambda_l \mathcal{A}_l^{(i_1)} \otimes \mathbf{u}^{(i_1)}_m \right)
\]
and
\[
\left( \mathbf{u}^{(1)}_j \otimes \cdots \otimes \mathbf{u}^{(i'-1)}_j \otimes \mathbf{x}^{(i')}_{j,p'} \otimes \mathbf{u}^{(i'+1)}_j \otimes \cdots \otimes \mathbf{u}^{(k)}_j \right) \perp \left( \lambda_m \mathcal{A}_m^{(i_1)} \otimes \mathbf{u}^{(i_1)}_l - \lambda_l \mathcal{A}_l^{(i_1)} \otimes \mathbf{u}^{(i_1)}_m \right)
\]
easily follow from the choice of $\mathbf{x}^{(i)}_{p}$ and $\mathbf{x}^{(i')}_{j,p'}$ respectively. Moreover, we notice that
\[
\langle\mathbf{u}^{(1)}_{j'} \otimes \cdots \otimes \mathbf{u}^{(k)}_{j'},\mathcal{A}_m^{(i_1)} \otimes \mathbf{u}^{(i_1)}_l\rangle = \delta_{j'm}\delta_{j'l} = 0
\]
since $l<m$. This implies
\[
\left( \mathbf{u}^{(1)}_{j'} \otimes \cdots \otimes \mathbf{u}^{(k)}_{j'} \right) \perp \left( \lambda_m \mathcal{A}_m^{(i_1)} \otimes \mathbf{u}^{(i_1)}_l - \lambda_l \mathcal{A}_l^{(i_1)} \otimes \mathbf{u}^{(i_1)}_m \right)
\]
and therefore the conclusion follows.

\item For $\mathcal{T}^{(i)}_{lm} \in C_{lm}$ and $\mathcal{T}^{(i_1)}_{l_1 m_1} \in C_{l_1 m_1}$, we recall that
\begin{align*}
\mathcal{T}^{(i)}_{lm} &\simeq \lambda_m \mathcal{A}_m^{(i)} \otimes \mathbf{u}^{(i)}_l - \lambda_l \mathcal{A}_l^{(i)} \otimes \mathbf{u}^{(i)}_m, \\
\mathcal{T}^{(i_1)}_{l_1m_1} &\simeq \lambda_{m_1} \mathcal{A}_{m_1}^{(i_1)} \otimes \mathbf{u}^{(i_1)}_{l_1} - \lambda_{l_1} \mathcal{A}_{l_1}^{(i_1)} \otimes \mathbf{u}^{(i_1)}_{m_1}.
\end{align*}
We discuss with respect to the following cases:
\begin{enumerate}[label=(\roman*)]
\item $l \ne m_1, l \ne l_1, m \ne l_1, m \ne m_1$: this case is obvious.
\item $l = m_1$: this implies that $l_1 < m_1 = l < m$.
\begin{enumerate}[label=(\roman*)]
\item $i = i_1$: we have $\mathcal{A}_m^{(i)} \otimes \mathbf{u}^{(i)}_l\perp \mathcal{A}_{l}^{(i)} \otimes \mathbf{u}^{(i)}_{l_1}$ since $\mathbf{u}^{(i)}_l \perp \mathbf{u}^{(i)}_{l_1}$; $\mathcal{A}_m^{(i)} \otimes \mathbf{u}^{(i)}_l \perp \mathcal{A}_{l_1}^{(i)} \otimes \mathbf{u}^{(i)}_{l}$ since $\mathcal{A}_m^{(i)} \perp \mathcal{A}_{l_1}^{(i)}$; $\mathcal{A}_l^{(i)} \otimes \mathbf{u}^{(i)}_m \perp \mathcal{A}_{l}^{(i)} \otimes \mathbf{u}^{(i)}_{l_1}$ since $\mathbf{u}^{(i)}_m \perp \mathbf{u}^{(i)}_{l_1}$; $\mathcal{A}_l^{(i)} \otimes \mathbf{u}^{(i)}_m \perp \mathcal{A}_{l_1}^{(i)} \otimes \mathbf{u}^{(i)}_{l}$ since $\mathbf{u}^{(i)}_m \perp \mathbf{u}^{(i)}_{l}$.
\item $i \ne i _1$: we have $\mathcal{A}_m^{(i)} \otimes \mathbf{u}^{(i)}_l \perp \mathcal{A}_{l}^{(i_1)} \otimes \mathbf{u}^{(i_1)}_{l_1}$ since $\mathbf{u}^{(i_1)}_m \perp \mathbf{u}^{(i_1)}_{l_1}$; $\mathcal{A}_m^{(i)} \otimes \mathbf{u}^{(i)}_l \perp \mathcal{A}_{l_1}^{(i_1)} \otimes \mathbf{u}^{(i_1)}_{l}$ since $\mathbf{u}^{(i)}_l \perp \mathbf{u}^{(i)}_{l_1}$; $\mathcal{A}_l^{(i)} \otimes \mathbf{u}^{(i)}_m \perp \mathcal{A}_{l}^{(i_1)} \otimes \mathbf{u}^{(i_1)}_{l_1}$ since $\mathbf{u}^{(i)}_m \perp \mathbf{u}^{(i)}_{l}$; $\mathcal{A}_l^{(i)} \otimes \mathbf{u}^{(i)}_m \perp \mathcal{A}_{l_1}^{(i_1)} \otimes \mathbf{u}^{(i_1)}_{l}$ since $\mathbf{u}^{(i)}_m \perp \mathbf{u}^{(i)}_{l_1}$.
\end{enumerate}
\item $l = l_1$: this implies that $l_1 = l < m, m_1$ and $m \ne m_1$.
\begin{enumerate}[label=(\roman*)]
\item $i = i_1$: we have $\mathcal{A}_m^{(i)} \otimes \mathbf{u}^{(i)}_l \perp \mathcal{A}_{m_1}^{(i)} \otimes \mathbf{u}^{(i)}_{l}$ since $\mathcal{A}_m^{(i)}  \perp \mathcal{A}_{m_1}^{(i)} $;
$\mathcal{A}_m^{(i)} \otimes \mathbf{u}^{(i)}_l \perp \mathcal{A}_{l}^{(i)} \otimes \mathbf{u}^{(i)}_{m_1}$ since $\mathbf{u}^{(i)}_l  \perp \mathbf{u}^{(i)}_{m_1}$; $\mathcal{A}_l^{(i)} \otimes \mathbf{u}^{(i)}_m \perp \mathcal{A}_{m_1}^{(i)} \otimes \mathbf{u}^{(i)}_{l}$ since $\mathbf{u}^{(i)}_{m} \perp \mathbf{u}^{(i)}_{l}$; $\mathcal{A}_l^{(i)} \otimes \mathbf{u}^{(i)}_m \perp \mathcal{A}_{l}^{(i)} \otimes \mathbf{u}^{(i)}_{m_1}$ since $\mathbf{u}^{(i)}_m \perp \mathbf{u}^{(i)}_{m_1}$.
\item $i \ne i _1$: we have $\mathcal{A}_m^{(i)} \otimes \mathbf{u}^{(i)}_l \perp \mathcal{A}_{m_1}^{(i_1)} \otimes \mathbf{u}^{(i_1)}_{l}$ since $\mathbf{u}^{(i_1)}_m \perp \mathbf{u}^{(i_1)}_{l}$;
$\mathcal{A}_m^{(i)} \otimes \mathbf{u}^{(i)}_l \perp \mathcal{A}_{l}^{(i_1)} \otimes \mathbf{u}^{(i_1)}_{m_1}$ since $\mathbf{u}^{({i_1})}_{{m_1}} \perp \mathbf{u}^{({i_1})}_{{m}}$; $\mathcal{A}_l^{(i)} \otimes \mathbf{u}^{(i)}_m \perp \mathcal{A}_{m_1}^{(i_1)} \otimes \mathbf{u}^{(i_1)}_{l}$ since $\mathbf{u}^{(i)}_m \perp \mathbf{u}^{(i)}_{m_1}$; $\mathcal{A}_l^{(i)} \otimes \mathbf{u}^{(i)}_m \perp \mathcal{A}_{l}^{(i_1)} \otimes \mathbf{u}^{(i_1)}_{m_1}$ since $\mathbf{u}^{(i)}_m \perp \mathbf{u}^{(i)}_{l}$.
\end{enumerate}
\item $m = m_1$: this implies that $l,l_1 < m = m_1$ and $l \ne l _1$. This case can be proved in a similar pattern as the last case.
\end{enumerate}
Hence we have $\mathcal{T}^{(i)}_{lm} \perp \mathcal{T}^{(i)}_{l_1 m_1}$ whenever $(l,m) \ne (l_1,m_1)$.

\item  For $\mathcal{T}^{(i)}_{lm}, \mathcal{T}^{(i_1)}_{l m} \in C_{lm}, i \ne i_1$, we have
\begin{align*}
\mathcal{T}^{(i)}_{lm} &\simeq \lambda_m \mathcal{A}_m^{(i)} \otimes \mathbf{u}^{(i)}_l - \lambda_l \mathcal{A}_l^{(i)} \otimes \mathbf{u}^{(i)}_m, \\
\mathcal{T}^{(i_1)}_{lm} &\simeq \lambda_{m} \mathcal{A}_{m}^{(i_1)} \otimes \mathbf{u}^{(i_1)}_{l} - \lambda_{l} \mathcal{A}_{l}^{(i_1)} \otimes \mathbf{u}^{(i_1)}_{m}.
\end{align*}
If $s\ge 3$, then it is not hard to check that the four summands in the above are pairwise orthogonal. If $s = 2$, then $i = 1$ and $i_1 = 2$. Suppose that $\mathcal{T}^{(2)}_{lm} = c \mathcal{T}^{(1)}_{lm}$ for some $c \ne 0$. Since $\lVert \mathcal{T}^{(1)}_{lm} \rVert^2 = \lVert \mathcal{T}^{(2)}_{lm} \rVert^2 = \lambda_l^2 + \lambda_m^2$, it is clear that $c = \pm 1$. We also have
\[
c (\lambda_l^2 + \lambda_m^2) = \langle \mathcal{T}^{(1)}_{lm}, \mathcal{T}^{(2)}_{lm} \rangle = -2 \lambda_l \lambda_m \prod_{j =3}^k \langle \mathbf{u}^{(j)}_m, \mathbf{u}^{(j)}_l \rangle,
\]
from which we may derive that $\lambda_l = \beta \lambda_m$ and $\prod_{j =3}^k \langle \mathbf{u}^{(j)}_m, \mathbf{u}^{(j)}_l \rangle = -c\beta$ {for some $\beta=\pm 1$}. This implies that $\mathbf{u}^{(j)}_l = \gamma_j \mathbf{u}^{(j)}_m$ for some $\gamma_j \in \{-1,1\}$, $3 \le j \le k$. It is easily checked from \eqref{eqn:basis Ckl} that this condition is also sufficient.

\item {It follows from \eqref{eqn:basis Ckl} that $\mathcal{T}^{(i)}_{lm} \ne 0$ if $(\lambda_l,  \lambda_m) \ne (0,0)$.} Hence we must have $\#(C_{lm}) = s$. According to \eqref{lem:basis dphi:item4}, we have $\dim (\operatorname{span} ( C_{lm})) = s$ if either {$s\geq 3$} or $s= 2$ but not in the exceptional case. For $s = 2$ with the exceptional case, we have $\dim (\operatorname{span} ( C_{lm})) = 1$.
\end{enumerate}
\end{proof}

\begin{corollary}\label{cor:basis dphi 1}
If $s \ge 3$, then
\begin{align*}
\dim (d_y \varphi_{\mathbf{n},r} (\Tt_{V_{\mathbf{n},r}} (y)) ) &= \Big(  \big(\sum_{i=1}^k n_i\big)  -\frac{q-1}{2}s - k    \Big) q+r, \\
\dim (\ker(d_y \varphi_{\mathbf{n},r}) ) &=  \Big( \big( \sum_{i=1}^k n_i \big) - \frac{(r + q -1)s}{2}-k \Big)(r - q) ,
\end{align*}
where $q$ is the cardinality of the set $J = \{j: \lambda_j \ne 0, j =1,\dots, r\}$. In particular if $(\lambda_1,\dots, \lambda_r) \in \mathbb{R}_{\ast}^r$, then $d_y \varphi_{\mathbf{n},r} : \Tt_{V_{\mathbf{n},r}} (y) \to \Tt_{\varphi_{\mathbf{n},r} (y)} (\mathbb{R}^{n_1} \otimes \cdots \otimes \mathbb{R}^{n_k})$ is injective.
\end{corollary}

\begin{proof}
According to Lemmas~\ref{lem:differential phi}, \ref{lem:differential phi 1} and \ref{lem:basis dphi}, we have
\begin{equation}\label{eq:direct}
d_y \varphi_{\mathbf{n},r} (\Tt_{V_{\mathbf{n},r}} (y) ) = \operatorname{span}(B) \bigoplus \Big( \bigoplus_{1\le l < m \le r} \operatorname{span} (C_{lm}) \Big).
\end{equation}
Note that there are
\[
{r\choose 2}-{r-q\choose 2}
\]
nonempty {$C_{lm}$'s} in the right hand side of \eqref{eq:direct}. Thus, we may obtain the desired formula for $\dim (d_y \varphi_{\mathbf{n},r} (\Tt_{V_{\mathbf{n},r} }(y)) )$ from Lemma~\ref{lem:basis dphi}--\eqref{lem:basis dphi:item5}. The dimension of $\ker(d_y \varphi_{\mathbf{n},r})$ is hence obtained {by} recalling that
\[
\dim V_{\mathbf{n},r} = \Big( \big( \sum_{i=1}^k n_i \big) + 1 - k - \frac{(r-1)s}{2} \Big) r.
\]
If in addition that $(\lambda_1,\dots, \lambda_r) \in \mathbb{R}_{\ast}^r$, then ${q} = r$ and hence
$\dim (d_y \varphi_{\mathbf{n},r} (\Tt_{V_{\mathbf{n},r}} (y)) ) = \dim V_{\mathbf{n},r}$ which implies the injectivity of $d_y \varphi_{\mathbf{n},r}$.
\end{proof}

We say that $y\in V_{\mathbf{n},r}$ with $(\lambda_l,\lambda_m) \ne (0,0)$ is \emph{$(l,m)$-exceptional} if $s=2$ and $y$ satisfies the exceptional condition described in Lemma~\ref{lem:basis dphi}--\eqref{lem:basis dphi:item5}. Let $E$ be a subset of
\[
\lbrace (l,m) \in J \times J: l < m \rbrace.
\]
We say that $y$ is \emph{exactly $E$-exceptional} if $s = 2$ and $y$ is $(l,m)$-exceptional if and only if $(l,m) \in E$. According to Lemma~\ref{lem:basis dphi}--\eqref{lem:basis dphi:item5}, if $y$ is both $(l,m)$-exceptional and $(m,{p})$-exceptional for $l < m < p$, then $y$ must also be $(l,p)$-exceptional. This implies that for a fixed exactly-$E$ exceptional $y$, $E$ is a disjoint union of sets of the form
\begin{equation}\label{eq:set-e}
\{(t_i,t_j) \in J \times J: 1 \le i < j \le u \},
\end{equation}
where $2 \le u \le r$ and $1 \le t_1 < \cdots < t_u \le r$ are some fixed integers.
\begin{corollary}\label{cor:basis dphi 2}
If $s = 2$ and $y$ is exactly $E$-exceptional, then
\begin{align*}
\dim (d_y \varphi_{\mathbf{n},r} (T_{V_{\mathbf{n},r} (y)}) ) &= \left( \left( \sum_{i=1}^k n_i \right) + {1} - k - q \right)q+{r} - e, \\
\dim (\ker(d_y \varphi_{\mathbf{n},r}) ) &=  \left( \left( \sum_{i=1}^k n_i \right) + {1}- k - {q} - r \right)(r -{q})  + e,
\end{align*}
where $e$ denotes the cardinality of $E$.
\end{corollary}

\begin{proof}
The argument is similar to that in the proof of Corollary~\ref{cor:basis dphi 1}. The only difference is that now there are some summands in right hand side of \eqref{eq:direct} having dimension one instead of two and the number of such summands is exactly $e$.
\end{proof}

Now we consider the incidence variety \footnote{The fact that the set $K$ in Figure~\ref{fig:incidence variety} is an incidence variety follows easily from the observation that the function $\langle \mathcal{A} - \varphi_{\mathbf{n},r} (y), d_y\varphi_{\mathbf{n},r} (\Tt_{V_{\mathbf{n},r}} (y)) \rangle$ is a polynomial in $\mathcal{A}$ and $y$.
\if discussion in Section~\ref{sec:kkt} and Lemma~\ref{lem:KKT parameter}, which imply that at a KKT point the Lagrange multiplier is unique and can be expressed as polynomial functions of the given tensor and the KKT point. Thus, the set $K$ is indeed defined by polynomial equations over the product of the tensor space and the parameter space.\fi} $K \subseteq (\mathbb{R}^{n_1} \otimes \cdots \otimes \mathbb{R}^n_k ) \times V_{\mathbf{n},r}$ together with its projection maps $\pi_1$ and $\pi_2$, which are interoperated in Figure~\ref{fig:incidence variety}.
\begin{figure}[ht]
  \begin{tikzcd}
     K \coloneqq \{(\mathcal{A},y): \langle \mathcal{A} - \varphi_{\mathbf{n},r} (y), d_y\varphi_{\mathbf{n},r} (\Tt_{V_{\mathbf{n},r}} (y)) \rangle = 0 \} \arrow[r,symbol=\subseteq] \arrow{d}{\pi_1}  \arrow{rd}{\pi_2}  & (\mathbb{R}^{n_1} \otimes \cdots \otimes \mathbb{R}^n_k ) \times V_{\mathbf{n},r}    \\
   \mathbb{R}^{n_1} \otimes \cdots \otimes \mathbb{R}^n_k \supseteq \pi_1(\pi_2^{-1}(Y)) & V_{\mathbf{n},r} \supseteq  Y\\
  \end{tikzcd}
  \caption{Incidence variety and its projections}
  \label{fig:incidence variety}
\end{figure}
We remark that for each $Y \subseteq V_{\mathbf{n},r}$, the  set $\pi_1(\pi_2^{-1}(Y))$ consists of all tensors $\mathcal{A}\in \mathbb{R}^{n_1} \otimes \cdots \otimes \mathbb{R}^{n_k}$ such that \eqref{lem:KKT parameter:optproblem} has at least one KKT point in $Y$. Moreover, given $\mathcal{A}_0 \in \mathbb{R}^{n_1} \otimes \cdots \otimes \mathbb{R}^{n_k}$ and $y_0 \in V_{\mathbf{n},r}$, we respectively have
\begin{equation}\label{eqn:fiber pi_1}
\pi_1^{-1}(\mathcal{A}_0) = \{\mathcal{A}_0\} \times  \{y\in V_{\mathbf{n},r}:
\mathcal{A}_0 \in \varphi_{\mathbf{n},r}(y) + \left( d_y\varphi_{\mathbf{n},r} (\Tt_{V_{\mathbf{n},r}} (y)) \right)^{\perp}
\}
\end{equation}
and
\begin{equation}\label{eqn:fiber pi_2}
\pi_2^{-1}(y_0) = \left( \varphi_{\mathbf{n},r}(y_0) + \left( d_{y_0}\varphi_{\mathbf{n},r} (\Tt_{V_{\mathbf{n},r}} (y_0)) \right)^{\perp} \right) \times \{y_0\}.
\end{equation}

The rest part of this subsection is concerned with $\dim (\pi_1(\pi_2^{-1}(Y)))$ for some specific subset $Y \subseteq V_{\mathbf{n},r}$, from which we may obtain an estimate of locations of KKT points of problem \eqref{lem:KKT parameter:optproblem}. The previous discussion already indicates that $\dim (d_y \varphi_{\mathbf{n},r} (\Tt_{V_{\mathbf{n},r} }(y)))$ varies with respect to different values of $s$. It turns out that the case of $s=1$ is more subtle than the other cases and hence we have to deal with them separately. 

In the following, we proceed with the following three different cases: $s \ge 3$, $s=2$ and $s=1$.
\subsubsection{{The case $s\ge 3$}}\label{sec:generic-sg3}
For integers $i_1,\dots, i_p$ such that $1\le \lvert i_1 \rvert  < \cdots < \lvert i_p \rvert \le r$, we define the subset $V_{\mathbf{n},r} (i_1,\dots, i_p)$ consisting of points $(U,\lambda) {\in V_{\mathbf{n},r}}$ satisfying
\begin{enumerate}[label=(\roman*)]
\item $\lambda_l \ne 0$ if and only if $l\in \{i_1,\dots, i_p\}$.
\item $\sgn(\lambda_l) = \sgn(l)$ for all $l\in \{i_1,\dots, i_p\}$.
\end{enumerate}
In particular, if $p = 0$ then $\{i_1,\dots, i_p\} = \emptyset$ and
\[
V_{\mathbf{n},r}({\emptyset}) = \V(r,n_1)\times\dots\times\V(r,n_s)\times \operatorname{B}(r,n_{s+1})\times\dots\times\operatorname{B}(r,n_k)\times \{0\}^r.
\]
It is clear that
\[
V_{\mathbf{n},r} = \bigcup_{1\le \lvert i_1 \rvert < \cdots < \lvert i_p \rvert \le r} V_{\mathbf{n},r} (i_1,\dots, i_p),
\]
and each $V_{\mathbf{n},r} (i_1,\dots, i_p)$ is a smooth submanifold of $V_{\mathbf{n},r}$.


\begin{lemma}\label{lem:fiber pi2 1}
Let $s\ge 3$ and let $V_{\mathbf{n},r} (i_1,\dots, i_p)$ be defined as before. For any $y \in V_{\mathbf{n},r} (i_1,\dots, i_p)$, the dimension of the fiber $\pi_2^{-1}(y)$ is a constant and hence for any irreducible subvariety $Z \subseteq V_{\mathbf{n},r} (i_1,\dots, i_p)$, $\pi_2^{-1}(Z)$ is irreducible of dimension
\[
\dim (Z) + \prod_{i=1}^k n_i -  \dim (d_y \varphi_{\mathbf{n},r} (\Tt_{V_{\mathbf{n},r}} (y))).
\]
\end{lemma}
\begin{proof}
Without loss of generality, we may assume that $\{i_1,\dots, i_p\} = \{1,\dots, p\}$.\footnote{Here we adopt the convention that if $p=0$ then $\{1,\dots, p\} = \emptyset$.} For simplicity, we denote $Y := V_{\mathbf{n},r} (1,\dots, p)$, $N := \prod_{i=1}^k n_i$ and $\mathbb{R}^N {\simeq} \mathbb{R}^{n_1} \otimes \cdots \otimes \mathbb{R}^{n_k}$.

First by \eqref{eqn:fiber pi_2}, $\pi_2^{-1}(y)$ is an affine linear subspace of $\mathbb{R}^N$, hence $\pi_2^{-1}(y)$ is irreducible for any $y\in V_{\mathbf{n},r}$. Next we compute the dimension of $\pi_2^{-1}(y)$ for each $y \in Y$. We have
\[
\dim (\pi_2^{-1}(y)) = N - \dim (d_y \varphi_{\mathbf{n},r} (\Tt_{V_{\mathbf{n},r}} (y))) .
\]
We notice that
$Y = \V(r,n_1) \times \cdots \times \V(r,n_s) \times \B(r,n_{s+1}) \times \cdots \times \B(r,n_k) \times \mathbb{R}_{{++}}^p \times \{0\}^{r- p}$, hence $Y$ is irreducible and for each $y \in Y$, $d_y \varphi_{\mathbf{n},r} (\Tt_{V_{\mathbf{n},r}} (y))$ has the same dimension due to Corollary~\ref{cor:basis dphi 1}. This implies $\dim (\pi_2^{-1}(y))$ is a constant $d$ for any $y\in Y$.

Now if $Z\subseteq Y$ is an irreducible subvariety, then $\pi_2^{-1}(Z)$ is homeomorphic to a rank $d$ vector bundle on $Z$. Hence by \cite[Lemma~5.8.13]{S-21}
we must have that $\pi_2^{-1}(Z)$ is irreducible of dimension
\begin{equation*}
\dim (\pi_2^{-1}(Z))) = \dim (Z) + d = \dim (Z) + N - \dim (d_y \varphi_{\mathbf{n},r} (T_{V_{\mathbf{n},r}} (y))).
\end{equation*}
\end{proof}

\begin{proposition}\label{prop:KKT location s=3}
Suppose $s\ge 3$. For a generic $\mathcal{A}$, all KKT points of \eqref{lem:KKT parameter:optproblem} are contained in $W_{\mathbf{n},r}$.
\end{proposition}
\begin{proof}
Let $Z$ be the complement of $W_{\mathbf{n},r}$ in $V_{\mathbf{n},r}$. Since
\[
W_{\mathbf{n},r} = \V(r,n_1) \times \cdots \times \V(r,n_s) \times \OB(r,n_{s+1}) \times \cdots \times\OB(r,n_k) \times \mathbb{R}^r,
\]
we have $Z  = V_{\mathbf{n},r} \setminus W_{\mathbf{n},r} = \bigcup_{i=s+1}^{k} Y_i$, where $Y_i$ denotes the set
\footnotesize
\[
\V(r,n_1) \times \cdots \times \V(r,n_s) \times \B(r,n_{s+1})  \times \cdots \times \B(r,n_{i-1})  \times (\B(r,n_i)\setminus \OB(r,n_i)) \times \B(r,n_{i+1})  \times \cdots \times \B(r,n_{k}) \times \mathbb{R}^r.
\]
\normalsize
Moreover, we partition $Y_i$ as
\begin{align*}
Y_i &= Y_i \bigcap \Big(  \bigcup_{1\le \lvert i_1 \rvert < \cdots < \lvert i_p \rvert \le r} V_{\mathbf{n},r} (i_1,\dots, i_p) \Big) \\
&= \bigcup_{1\le \lvert i_1 \rvert < \cdots < \lvert i_p \rvert \le r} \left( Y_i \cap V_{\mathbf{n},r} (i_1,\dots, i_p) \right).
\end{align*}
Hence it is sufficient to prove that $\pi_1 (\pi^{-1}_2(Y_i \cap V_{\mathbf{n},r} (i_1,\dots, i_p) ) )$ has dimension strictly smaller than $N$.

The case that $\overline{\pi_1 (\pi^{-1}_2(V_{\mathbf{n},r} (i_1,\dots, i_p) ) )}$ is a proper subvariety of $\mathbb{R}^N$ is trivial. In the following, we assume that $\overline{\pi_1 (\pi^{-1}_2(V_{\mathbf{n},r} (i_1,\dots, i_p) ) ) }=\mathbb{R}^N$. In this case, for a generic $\mathcal{T}\in\pi_1 (\pi^{-1}_2(V_{\mathbf{n},r} (i_1,\dots, i_p) ) $, scheme theoretically we have
\begin{equation}\label{eq:generic}
\dim(\pi_1^{-1}(\mathcal{T}))=\dim(\pi^{-1}_2(V_{\mathbf{n},r} (i_1,\dots, i_p) ) )-N=\dim V_{\mathbf{n},r}(i_1,\dots, i_p) -\dim (d_y \varphi_{\mathbf{n},r} (T_{V_{\mathbf{n},r}} (y))),
\end{equation}
where the first equality follows from \cite[Corollary~14.5]{E-95} and the second equality follows from Lemma~\ref{lem:fiber pi2 1}. We may rewrite \eqref{eq:generic} as
\begin{equation}\label{eq:generic1}
\dim ( V_{\mathbf{n},r}(i_1,\dots, i_p) ) = \dim(\pi_1^{-1}(\mathcal{T})) + \dim (d_y \varphi_{\mathbf{n},r} (T_{V_{\mathbf{n},r}} (y))).
\end{equation}

For simplicity, we denote $Y \coloneqq Y_i \cap V_{\mathbf{n},r} (i_1,\dots, i_p)$. Suppose on the contrary that $\dim \left( \pi_1 (\pi^{-1}_2( Y ) )  \right) =N$. Then for a generic $\mathcal{T}_1 \in \pi_1 (\pi^{-1}_2( Y ) )$, the relation \eqref{eq:generic1} still holds. We notice that $\dim (d_y \varphi_{\mathbf{n},r} (T_{V_{\mathbf{n},r}} (y)))$ and hence $\dim (\pi_2^{-1}(y))$ is a constant for any $y\in V_{\mathbf{n},r} (i_1,\dots, i_p)$ and that $\dim \left( Y \right) < \dim ( V_{\mathbf{n},r} (i_1,\dots, i_p) )$. Hence we may derive
\begin{align*}
N = \dim \left( \pi_1 (\pi^{-1}_2( Y ) )  \right)  &= \dim  (\pi_2^{-1} ( Y) ) - \dim (\pi_1^{-1} (\mathcal{T}_1)) \\
&= \dim (Y) + \dim (\pi_2^{-1}(y)) - \dim (\pi_1^{-1} (\mathcal{T}_1))  \\
&< \dim (V_{\mathbf{n},r} (i_1,\dots, i_p) ) + \dim (\pi_2^{-1}(y)) - \dim (\pi_1^{-1} (\mathcal{T}_1)) \\
&=  \dim (V_{\mathbf{n},r} (i_1,\dots, i_p) )  + (N - \dim (d_y \varphi_{\mathbf{n},r} (T_{V_{\mathbf{n},r}} (y))) ) - \dim (\pi_1^{-1} (\mathcal{T}_1)) \\
&= N + \left( \dim (V_{\mathbf{n},r} (i_1,\dots, i_p) ) - \dim (\pi_1^{-1} (\mathcal{T}_1)) - \dim (d_y \varphi_{\mathbf{n},r} (T_{V_{\mathbf{n},r}} (y)))  \right) \\
&= N,
\end{align*}
which is absurd and this completes the proof.
\end{proof}

\subsubsection{{The case $s=2$}}\label{sec:generic-s2}
We next discuss locations of  KKT points of the problem \eqref{lem:KKT parameter:optproblem} for  $s= 2$. Let $i_1,\dots, i_p$ be integers such that $1 \le \lvert i_1 \rvert < \cdots < \lvert i_p \rvert \le r$ and let $E$ be a subset of $\{(s,t) \in \mathbb{N} \times \mathbb{N}: 1\le s < t \le r, (\lambda_s,\lambda_t) \ne (0,0)\}$ which is a union of disjoint sets $E_1,\dots, E_m$ defined by
\begin{equation}\label{eq:set-u}
E_q \coloneqq \{(t_{{q,i}},t_{{q,j}}): 1 \le i < j \le u_q\},
\end{equation}
where $q = 1,\dots, m$, $2 \le u_q \le r$ and  $1\le t_{{q,1}} < \cdots < t_{q,u_q} \le r$.
We denote by $V_{\mathbf{n},r} (i_1,\dots,i_p;E)$ the subset of $V_{\mathbf{n},r}(i_1,\dots, i_p)$ consisting of $y$ which is exactly $E$-exceptional. Note that $E$ must be a subset of $\{|i_1|,\dots,|i_p|\}\times \{|i_1|,\dots,|i_p|\}$ by Lemma~\ref{lem:basis dphi}-\eqref{lem:basis dphi:item5}. In particular, we have $V_{\mathbf{n},r} (i_1,\dots,i_p;\emptyset) = V_{\mathbf{n},r}(i_1,\dots, i_p)$.

\begin{lemma}\label{lem:fiber pi2 2}
Let $s = 2$ and let $V_{\mathbf{n},r} (i_1,\dots, i_p;E)$ be defined as before. For any $y \in V_{\mathbf{n},r} (i_1,\dots, i_p;E)$, the dimension of the fiber $\pi_2^{-1}(y)$ is a constant. Hence for an irreducible subvariety $Z$  of  $V_{\mathbf{n},r} (i_1,\dots, i_p;E)$, $\pi_2^{-1}(Z)$ is irreducible of dimension
\[
\dim (Z) + \prod_{i=1}^k n_i -  \dim (d_y \varphi_{\mathbf{n},r} (\Tt_{V_{\mathbf{n},r}} (y))).
\]
\end{lemma}

\begin{proof}
The proof is similar to the one for Lemma~\ref{lem:fiber pi2 1}. By Corollary~\ref{cor:basis dphi 2}, the dimension of $\pi_2^{-1}(y)$ is a constant for any $y\in V_{\mathbf{n},r} (i_1,\dots, i_p;E)$. Since $\pi_2^{-1}(y)$ is an affine linear subspace of $\mathbb{R}^{n_1} \otimes \cdots \otimes \mathbb{R}^{n_k}$, we see that $\pi_2^{-1} (Z)$ is homeomorphic to a vector bundle on $Z$ and hence by \cite[Lemma~5.8.13]{S-21}, it must be irreducible of dimension
\[
\dim (Z) + \prod_{i=1}^k n_i -  \dim (d_y \varphi_{\mathbf{n},r} (\Tt_{V_{\mathbf{n},r}} (y))).
\]
\end{proof}

\begin{lemma}\label{lem:s=2 nonempty E}
Let $s = 2$ and let $V_{\mathbf{n},r} (i_1,\dots, i_p;E)$ be defined as before. Assume that $E$ is nonempty. We have
\[
\dim (\pi_1 (\pi_2^{-1}(V_{\mathbf{n},r} (i_1,\dots, i_p;E)) )) < \prod_{i=1}^k n_i.
\]
\end{lemma}
\begin{proof}
We again assume that $(i_1,\dots, i_p) = (1,\dots, p)$ and denote
\[
Y \coloneqq  V_{\mathbf{n},r} (i_1,\dots, i_p) \supseteq Z \coloneqq V_{\mathbf{n},r} (i_1,\dots, i_p;E),\quad N \coloneqq \prod_{i=1}^k n_i.
\]
Let $y$ (resp. $y'$) be a point in $Z$ (resp. $Y\setminus \big( \bigsqcup_{{E_i} \neq \emptyset} V_{\mathbf{n},r}(i_1,\dots, i_p; {E_i}) \big)$) and let $\mathcal{T}$ be a generic point in $\pi_1 (\pi_2^{-1}(Z))$. We have
\footnotesize
\begin{align*}
\dim (\pi_1 (\pi_2^{-1} (Z)))
&= \dim (\pi_2^{-1} (Z)) - \dim (\pi_1^{-1} (\mathcal{T})) \\
&= \dim(Z) + \dim (\pi_2^{-1}(y)) -  \dim (\pi_1^{-1} (\mathcal{T}))\\
&=\left[ \dim (Y) + \dim(\pi_2^{-1}(y')) - \dim (\pi_1^{-1}(\mathcal{T})) \right] +\left[ ( \dim(\pi_2^{-1}(y)) - \dim (\pi_2^{-1}(y')) ) -( \dim (Y) - \dim(Z)) \right] \\
&=\left[ \dim (Y) + \dim(\pi_2^{-1}(y')) - \dim (\pi_1^{-1}(\mathcal{T})) \right]\\
 &\quad\quad\quad +\big[ ( \dim(d_{{y'}} \varphi_{\mathbf{n},r}{(\Tt_{V(\mathbf{n},r)}(y'))}) - \dim (d_{{y}} \varphi_{\mathbf{n},r}{(\Tt_{V(\mathbf{n},r)}(y))}) )  -( \dim (Y) - \dim(Z)) \big] \\
&\le N + \left[ e  - ( \dim (Y) - \dim(Z)) \right],
\end{align*}
\normalsize
where the first term in the last line follows from the inequality
\[
\dim (Y) + \dim(\pi_2^{-1}(y')) - \dim (\pi_1^{-1}(\mathcal{T})) \le \dim  \pi_1(\pi_2^{-1} (Y)) \le N.
\]

We notice that $E = \bigsqcup_{q =1}^m E_q$ where each $E_q$ is of form \eqref{eq:set-u}. This implies
\[
e =\sum_{q=1}^m \binom{u_q}{2}.
\]
Moreover, according to Lemma~\ref{lem:basis dphi}--\eqref{lem:basis dphi:item5}, we know that
\small
\[
\dim (Y) - \dim(Z) = \sum_{q=1}^m \big( (u_q -1) +  \sum_{j=3}^k (u_q-1)(n_j - 1) \big) =  \sum_{q=1}^m\Big[  (u_q-1) \big( 1 + \sum_{j=3}^k(n_j - 1) \big) \Big].
\]
\normalsize
Since
\[
1 + \sum_{j=3}^k (n_j - 1) \ge  1 + (k-2)(r-1) > \frac{r}{2} \geq \frac{u_q}{2},
\]
we obtain that $\dim (Y) - \dim(Z) > e$ and thus $\dim (\pi_1 (\pi_2^{-1}(V_{\mathbf{n},r} (i_1,\dots, i_p;E)) )) < N$.
\end{proof}

\begin{proposition}\label{prop:KKT location s=2}
Suppose $s = 2$. For a generic $\mathcal{A}$, all KKT points of \eqref{lem:KKT parameter:optproblem} are contained in $W_{\mathbf{n},r}$.
\end{proposition}

\begin{proof}
Let $Z = V_{\mathbf{n},r} \setminus W_{\mathbf{n},r}$. We notice that
\small
\[
V_{\mathbf{n},r} = \bigg( \bigcup_{\substack{E\neq \emptyset\\ {1\le |i_1| < \cdots < |i_p| \le r}}}  V_{\mathbf{n},r} (i_1,\dots, i_p; E) \bigg) \bigcup \bigg(
\bigcup_{{1\le |i_1| < \cdots < |i_p| \le r}} \Big( V_{\mathbf{n},r} (i_1,\dots, i_p) \setminus  \big( \bigcup_{E\neq \emptyset}  V_{\mathbf{n},r} (i_1,\dots, i_p; E) \big) \Big)
\bigg).
\]
\normalsize
Therefore, it is sufficient to prove that
\begin{align*}
\dim \pi_1\left( \pi_2^{-1} \left( Z \cap V_{\mathbf{n},r} (i_1,\dots, i_p; E) \right) \right) &< N, \quad E \neq\emptyset \\
\dim \pi_1\bigg( \pi_2^{-1} \Big( Z \bigcap  \big(
V_{\mathbf{n},r} (i_1,\dots, i_p) \setminus  \big( \bigcup_{E\neq \emptyset}  V_{\mathbf{n},r} (i_1,\dots, i_p; E) \big)
\big)
\Big)
\bigg)&<N.
\end{align*}
The first inequality follows from Lemma~\ref{lem:s=2 nonempty E} and the second is obtained by {a similar argument as Proposition~\ref{prop:KKT location s=3}}.
\end{proof}

\subsubsection{{The case $s=1$}}\label{sec:s=1} If $s=1$, the situation is more complicated than $s \ge 2$. In this case, we let   $y = (U^{(1)},\dots, U^{(k)},(\lambda_1,\dots, \lambda_r))$ be a fixed point in $V_{\mathbf{n},r} (i_1,\dots, i_p)$ where $i_1,\dots, i_p$ are some integers such that $1\le \lvert i_1 \rvert < \cdots < \lvert i_p \rvert \le r$. We define
\begin{align}
B_1 &:= \{ \mathbf{x}_{q} \otimes  \mathbf{u}^{(2)}_j \otimes \cdots \otimes \mathbf{u}^{(k)}_j\}, \label{eqn:basis B1} \\
B_2&:= \{ \mathbf{u}^{(1)}_j \otimes \cdots \otimes \mathbf{u}^{(i'-1)}_j \otimes \mathbf{x}^{(i')}_{j,q'} \otimes \mathbf{u}^{(i'+1)}_j \otimes \cdots \otimes \mathbf{u}^{(k)}_j \}, \label{eqn:basis B2} \\
B_3 &:= \{ \mathbf{u}^{(1)}_{j} \otimes \cdots \otimes \mathbf{u}^{(k)}_{j} \}, \label{eqn:basis B3}   \\
B_4 &:= \{ \mathbf{u}^{(1)}_{j'} \otimes \cdots \otimes \mathbf{u}^{(k)}_{j'} \}. \label{eqn:basis B4}
\end{align}
where $j\in J\coloneqq \{t: \lambda_t \ne 0\}$, $j' \in J^c \coloneqq \{t: \lambda_t = 0\}$, $1\le q \le n_{1} - r$,  $2 \le i' \le k $ and $1\le q' \le n_{i'} - 1$, $\{ \mathbf{x}_{q} \}_{q= 1}^{n_1- r}$ is an orthonormal basis of $\left(\operatorname{span}( \mathbf{u}^{(1)}_1,\dots, \mathbf{u}^{(1)}_r) \right)^\perp$ and $\{ \mathbf{x}^{(i')}_{j,q'} \}_{q'=1}^{n_{i'}-1}$ is an orthonormal basis of $\left(\operatorname{span}( \mathbf{u}^{(i')}_j) \right)^\perp$.

For each pair $1\le l < m \le r$, we also let $\mathcal{T}_{lm}$ be the tensor defined by
\[
\mathcal{T}_{lm} := \lambda_m \mathbf{u}^{(1)}_l \otimes \mathbf{u}^{(2)}_m \otimes  \cdots \otimes \mathbf{u}^{(k)}_m - \lambda_l \mathbf{u}^{(1)}_m \otimes \mathbf{u}^{(2)}_l  \otimes \cdots  \otimes \mathbf{u}^{(k)}_l.
\]
If we denote $\mathcal{A}_l \coloneqq \mathbf{u}_l^{(2)} \otimes \cdots \otimes \mathbf{u}_l^{(k)}$, then we can rewrite $\mathcal{T}_{lm} $ as
\[
\mathcal{T}_{lm}  =  \lambda_m \mathbf{u}^{(1)}_l \otimes \mathcal{A}_m - \lambda_l \mathbf{u}^{(1)}_m \otimes \mathcal{A}_l.
\]
We let
\begin{align*}
C_1 &= \{
\mathcal{T}_{lm}: \lambda_l \ne 0, \lambda_m = 0~\text{or}~
 \lambda_l = 0, \lambda_m \ne 0
\} =\{ -\lambda_l \mathbf{u}_m^{(1)} \otimes \mathcal{A}_l:  \lambda_l \ne 0  \} \cup \{ \lambda_m \mathbf{u}_l^{(1)} \otimes \mathcal{A}_m:  \lambda_m \ne 0  \} , \\
C_2 &= \{
\mathcal{T}_{lm}:  \lambda_l \ne 0, \lambda_m \ne 0
\} = \{\lambda_m \mathbf{u}^{(1)}_l \otimes \mathcal{A}_m - \lambda_l \mathbf{u}^{(1)}_m \otimes \mathcal{A}_l: \lambda_l \ne 0, \lambda_m \ne 0\}.
\end{align*}
\begin{lemma}\label{lem:basis dphi s=1}
The union of $B_1,B_2,B_3,B_4,C_1$ and $C_2$ is a spanning set of $d_y \varphi_{\mathbf{n},r} (\Tt_{V_{\mathbf{n},r}} (y))$ satisfying the following properties:
\begin{enumerate}[label=(\roman*)]
\item elements in $B_1$ are orthogonal to those in $B_2 \cup B_3 \cup B_4 \cup C_1 \cup C_2$; \label{lem:basis dphi s=1:item1}
\item elements in $B_2 \cup B_3 \cup B_4$ are pairwise orthogonal; \label{lem:basis dphi s=1:item2}
\item elements in $C_1$ are orthogonal to those in $B_2 \cup B_3  \cup C_2$; \label{lem:basis dphi s=1:item3}
\item elements in $C_2$ are orthogonal to those in $B_4$; \label{lem:basis dphi s=1:item4}
\item The vector space $d_y \varphi_{\mathbf{n},r} (\Tt_{V_{\mathbf{n},r}} (y))$ can be written as
\begin{equation}\label{eq:decomposition}
d_y \varphi_{\mathbf{n},r} (\Tt_{V_{\mathbf{n},r}} (y)) = \operatorname{span} (B_1) \oplus \left( \operatorname{span} (C_1) + \operatorname{span} (B_4)  \right) \ \oplus  \left( \operatorname{span} (B_2 \sqcup B_3) + \operatorname{ span}(C_2) \right).
\end{equation}
\label{lem:basis dphi s=1:item5}
\end{enumerate}
\end{lemma}

\begin{proof}
The fact that the set $B_1\cup B_2 \cup B_3 \cup B_4 \cup C_1 \cup C_2$ spans $d_y \varphi_{\mathbf{n},r} (\Tt_{V_{\mathbf{n},r}} (y))$ follows from Lemmas~\ref{lem:differential phi} and \ref{lem:differential phi 1}. In the following, we fix positive integers $j,j_1\in \{t: \lambda_t \ne 0\}$, $2 \le i',i'_1  \le k$, $1\le q \le n_1 -r$, $1\le q' \le n_{i'} - 1$ and $1\le l,m \le r$.
\begin{enumerate}[label=(\roman*)]
\item We observe that
\begin{align*}
(\mathbf{x}_{q} \otimes  \mathbf{u}^{(2)}_j \otimes \cdots \otimes \mathbf{u}^{(k)}_j) &\perp (\mathbf{u}^{(1)}_{j_1} \otimes \cdots \otimes \mathbf{u}^{(i'-1)}_{j_1} \otimes \mathbf{x}^{(i')}_{j_1,q'} \otimes \mathbf{u}^{(i'+1)}_{j_1} \otimes \cdots \otimes \mathbf{u}^{(k)}_{j_1}), \\
(\mathbf{x}_{q} \otimes  \mathbf{u}^{(2)}_j \otimes \cdots \otimes \mathbf{u}^{(k)}_j)  &\perp (\mathbf{u}^{(1)}_{j_1} \otimes \cdots \otimes \mathbf{u}^{(k)}_{j_1}), \\
(\mathbf{x}_{q} \otimes  \mathbf{u}^{(2)}_j \otimes \cdots \otimes \mathbf{u}^{(k)}_j)  &\perp (\lambda_m \mathbf{u}^{(1)}_l \otimes \mathbf{u}^{(2)}_m \otimes  \cdots \otimes \mathbf{u}^{(k)}_m - \lambda_l \mathbf{u}^{(1)}_m \otimes \mathbf{u}^{(2)}_l  \otimes \cdots  \otimes \mathbf{u}^{(k)}_l),
\end{align*}
since $\mathbf{x}_{q}$ is orthogonal to $\mathbf{u}^{(1)}_{a}$ for any $1\le a  \le r$.
\item The relation
\footnotesize{
\[
( \mathbf{u}^{(1)}_j \otimes \cdots \otimes \mathbf{u}^{(i'-1)}_j \otimes \mathbf{x}^{(i')}_{j,q'} \otimes \mathbf{u}^{(i'+1)}_j \otimes \cdots \otimes \mathbf{u}^{(k)}_j ) \perp ( \mathbf{u}^{(1)}_{j_1} \otimes \cdots \otimes \mathbf{u}^{(i_1'-1)}_{j_1} \otimes \mathbf{x}^{(i_1')}_{j_1,q_1'} \otimes \mathbf{u}^{(i_1'+1)}_{j_1} \otimes \cdots \otimes \mathbf{u}^{(k)}_j )
\]}
\normalsize
follows from the facts
\begin{enumerate}[label=(\alph*)]
\item $j \ne j_1$: $\mathbf{u}^{(1)}_j \perp \mathbf{u}^{(1)}_{j_1}$;
\item $j = j_1, i' \ne i'_1$: $\mathbf{x}^{(i')}_{j,q'} \perp \mathbf{u}^{(i')}_{j}$; and
\item $j = j_1, i' = i'_1, q' \ne q'_1$: $\mathbf{x}^{(i')}_{j,q'} \perp \mathbf{x}^{(i')}_{j,q_1'}$.
\end{enumerate}
The orthogonality between $\mathbf{u}^{(1)}_j \otimes \cdots \otimes \mathbf{u}^{(k)}_j$ and $\mathbf{u}^{(1)}_{j_1} \otimes \cdots \otimes \mathbf{u}^{(k)}_{j_1}$ for $j \ne j_1$ can be obtained from the orthogonality between $\mathbf{u}^{(1)}_j$ and $\mathbf{u}^{(1)}_{j_1}$. {Finally, the relation
\[
( \mathbf{u}^{(1)}_j \otimes \cdots \otimes \mathbf{u}^{(i'-1)}_j \otimes \mathbf{x}^{(i')}_{j,q'} \otimes \mathbf{u}^{(i'+1)}_j \otimes \cdots \otimes \mathbf{u}^{(k)}_j ) \perp  (\mathbf{u}^{(1)}_{j_1} \otimes \cdots \otimes \mathbf{u}^{(k)}_{j_1})
\]
follows from the facts
\begin{enumerate}[label=(\alph*)]
\item $j = j_1$: $ \mathbf{x}^{(i')}_{j,q'} \perp \mathbf{u}^{(i')}_{j}$, and
\item $j \ne j _1$: $\mathbf{u}^{(1)}_j \perp \mathbf{u}^{(1)}_{j_1}$.
\end{enumerate}}

\item We have
\footnotesize{
\[
(\lambda_l \mathbf{u}_m^{(1)} \otimes \mathcal{A}_l) \perp ( \mathbf{u}^{(1)}_j \otimes \cdots \otimes \mathbf{u}^{(i'-1)}_j \otimes \mathbf{x}^{(i')}_{j,q'} \otimes \mathbf{u}^{(i'+1)}_j \otimes \cdots \otimes \mathbf{u}^{(k)}_j ),\quad (\lambda_l \mathbf{u}_m^{(1)} \otimes \mathcal{A}_l) \perp (\mathbf{u}^{(1)}_{j} \otimes \cdots \otimes \mathbf{u}^{(k)}_{j})
\]}
\normalsize
and
\[
(\lambda_l \mathbf{u}_m^{(1)} \otimes \mathcal{A}_l) \perp (\lambda_{m_1} \mathbf{u}_{l_1}^{(1)} \otimes \mathcal{A}_{m_1} - \lambda_{l_1} \mathbf{u}^{(1)}_{m_1} \otimes \mathcal{A}_{l_1})
\]
if $\lambda_m = 0$, $\lambda_j \ne 0$, $\lambda_{l_1}\ne 0$ and $\lambda_{m_1} \ne 0$ (in particular, $j,l_1,m_1 \ne m$), from which we can derive $\mathbf{u}_m^{(1)} \perp \mathbf{u}^{(1)}_j
$, $\mathbf{u}_m^{(1)} \perp \mathbf{u}^{(1)}_{l_1}
$ and $\mathbf{u}_m^{(1)} \perp \mathbf{u}^{(1)}_{m_1}$.
\item Since $\mathbf{u}^{(1)}_l \perp \mathbf{u}^{(1)}_{j'}$ whenever {$\lambda_l\ne 0$ and $\lambda_{j'}=0$}, the statement can be easily verified.
\item This is obviously obtained by orthogonal relations described in Items \eqref{lem:basis dphi s=1:item1}--\eqref{lem:basis dphi s=1:item4}.
\end{enumerate}
\end{proof}

According to \eqref{eq:decomposition}, we observe that the defectivity of $\dim (d_y \varphi_{\mathbf{n},r} (\Tt_{V_{\mathbf{n},r}} (y)))$ with respect to $\dim (V_{\mathbf{n},r} )$ consists of four parts: $\alpha$ from $\operatorname{span}(B_1)$; $\beta_1$ from $\operatorname{span}(C_1)$; $\beta_2$ from $\left( \operatorname{span}(C_1) \cap \operatorname{span}(B_4)  \right)$; and $\gamma$ from $\left( \operatorname{span} (B_2 \sqcup B_3) + \operatorname{ span}(C_2) \right)$. We assume that
\begin{equation}\label{eq:linear relation}
\dim \left(
\operatorname{span}(\mathcal{A}_{\lvert i_1 \rvert},\dots, \mathcal{A}_{\lvert i_p \rvert}) \right) = p - \delta_0,
\end{equation}
where $0 \le \delta_0 \le p-1$. It is clear that
\begin{equation}\label{eq:alpha+beta1}
\alpha + \beta_1 = \delta_0(n_1 - p).
\end{equation}
It is straightforward to verify that
\begin{equation}\label{eq:beta2}
\beta_2 = \# \lbrace  t: \lambda_{t} = 0, \mathcal{A}_t \in \operatorname{span}(\mathcal{A}_{\lvert i_1 \rvert},\dots, \mathcal{A}_{\lvert i_p \rvert}) \rbrace
\end{equation}
and $\gamma$ is equal to the dimension of the linear space of vectors $(c_{lm})\in \mathbb{R}^{\binom{p}{2}}$ satisfying the linear system
\begin{multline}\label{eq:gamma}
\sum_{j=i_1,\dots, i_p} \sum_{i'=2}^k  \mathbf{u}^{(1)}_j \otimes \mathbf{u}^{(2)}_j \otimes \cdots \otimes \mathbf{u}^{(i'-1)}_j \otimes \mathbf{y}_j^{(i')} \otimes \mathbf{u}^{(i'+1)}_j \otimes \cdots \otimes \mathbf{u}^{(k)}_j \\
= \sum_{\substack{l,m = i_1,\dots,i_p \\l < m}} c_{lm} (\lambda_m \mathbf{u}_l^{(1)} \otimes \mathcal{A}_m - \lambda_l \mathbf{u}_m^{(1)} \otimes \mathcal{A}_l)
\end{multline}
for some $\mathbf{y}^{(i')}_j \in \mathbb{R}^{n_{i'}}, i' = 2,\dots k, j = i_1,\dots, i_p$.

Let $\delta$ be a nonnegative integer
and let $V_{\mathbf{n},r}(i_1,\dots, i_p;\delta)$ be the subset of $V_{\mathbf{n},r}(i_1,\dots, i_p)$ consisting of points $y$ such that the defectivity of $\dim (d_y \varphi_{\mathbf{n},r} (T_{V_{\mathbf{n},r}} (y)))$ is $\delta$, i.e.,
\[
\delta = \alpha + \beta_1 + \beta_2 + \gamma= \delta_0(n_1-p) + \beta_2 + \gamma.
\]
We also denote
\[
\operatorname{codim}(V_{\mathbf{n},r}(i_1,\dots, i_p;\delta)) \coloneqq \dim (V_{\mathbf{n},r}(i_1,\dots, i_p)) - \dim((V_{\mathbf{n},r}(i_1,\dots, i_p;\delta))).
\]
The following technical lemma compares $\operatorname{codim}(V_{\mathbf{n},r}(i_1,\dots, i_p;\delta)) $ with $\delta$.
\begin{lemma}\label{lem:defectivity s=1}
If positive integers {$k\ge 3, n_1,\dots, n_k \ge 2$ satisfy} the following inequalities:
\begin{align}
\sum_{i'=2}^k n_{i'} &\geq n_1 +k-1, \label{lem:defectivity s=1:eq01}\\
\sum_{i'=2}^k n_{i'} &\ge n_{i''} + k,\quad 2 \le i'' \le k, \label{lem:defectivity s=1:eq02}
\end{align}
then for any $\delta >0$, we have
\[
\operatorname{codim}(V_{\mathbf{n},r}(i_1,\dots, i_p;\delta))  > \delta.
\]
\end{lemma}
\begin{proof}
To prove the claimed inequality, we estimate contributions of \eqref{eq:linear relation}, \eqref{eq:beta2} and \eqref{eq:gamma} to the codimension $\operatorname{codim}(V_{\mathbf{n},r}(i_1,\dots, i_p;\delta)) $ and divide the proof into three parts accordingly. 
\begin{enumerate}[label= (\roman*)]
\item Contribution of \eqref{eq:linear relation}: For simplicity, we may assume that $i_1 = 1,\dots, i_p = p$. In this case, we have
\[
\dim \left( \operatorname{span}(\{\mathcal{A}_1, \dots, \mathcal{A}_p \}) \right) = p - \delta_0.
\]
We may also assume that $\mathcal{A}_1,\dots, \mathcal{A}_{p - \delta_0}$ are linearly independent. Hence for each $p-\delta_0 +1 \le l \le p$, there exists some $a_{lj}\in \mathbb{R}$ with $j=1,\dots,p-\delta_0$ such that
\begin{equation}\label{lem:defectivity s=1:eq1}
\mathbf{u}^{(2)}_l \otimes \cdots \otimes \mathbf{u}^{(k)}_l = \sum_{j=1}^{p - \delta_0} a_{lj} \mathbf{u}^{(2)}_j \otimes \cdots \otimes \mathbf{u}^{(k)}_j.
\end{equation}
For each $i' = 2,\dots, k$, we denote
\[
r_{i'} = \rank \left( [ \mathbf{u}^{(i')}_1,\dots, \mathbf{u}^{(i')}_{p-\delta_0}] \right) \ge 1.
\]
We first address the extreme case where $r_{i_0'} = p-\delta_0$ for some $2 \le i_0' \le k$. The $(p-\delta_0)$ tensors
\begin{equation}\label{lem:defectivity s=1:eq3}
{a_{l1}} \mathbf{u}^{(2)}_1 \otimes \cdots \otimes  \widehat{\mathbf{u}^{(i_0')}_1} \otimes \cdots  \otimes \mathbf{u}^{(k)}_1,\dots, {a_{l,p-\delta_0}} \mathbf{u}^{(2)}_{p-\delta_0} \otimes \cdots \otimes  \widehat{\mathbf{u}^{(i_0')}_{p-\delta_0}} \otimes \cdots  \otimes \mathbf{u}^{(k)}_{p-\delta_0}
\end{equation}
are pairwise differed by a multiple scalar. We may infer from the combination of \eqref{lem:defectivity s=1:eq3} and \eqref{lem:defectivity s=1:eq1} that for each $p - \delta_0 + 1 \le l \le p$, there exists some $1\le j \le p - \delta_0$ such that $\mathbf{u}_l^{(i'_0)} \in {\operatorname{span} (\mathbf{u}_{1}^{(i'_0)},\dots,\dots, \mathbf{u}_{p-\delta_0}^{(i'_0)})}$ and 
\begin{equation}\label{lem:defectivity s=1:eq5}
\mathbf{u}_l^{(2)} = \pm \mathbf{u}_{j}^{(2)},\dots, \mathbf{u}_l^{(i'_0-1)} = \pm \mathbf{u}_{j}^{(i'_0-1)}, \mathbf{u}_l^{(i'_0+1)} = \pm \mathbf{u}_{j}^{(i'_0+1)},\dots, \mathbf{u}_{l}^{(k)} = \pm \mathbf{u}_{j}^{(k)}.
\end{equation}
\normalsize
Hence the contribution of \eqref{lem:defectivity s=1:eq1} to $\operatorname{codim} (V_{\mathbf{n},r}(i_1,\dots, i_p;\delta) )$ is at least
\begin{equation}\label{lem:defectivity s=1:eq6}
 (n_2 -1) + \cdots + (n_{i'_0-1} -1) + (n_{i'_0}- (p - \delta_0) ) + (n_{i'_0+1} -1) + \cdots + (n_k -1).
\end{equation}

Next we deal with the general case where $r_{i'} < p-\delta_0$ for all $2\le i' \le k$. It suffices to consider $i' = 2$. Without loss of generality, we suppose that $\mathbf{u}_1^{(2)}, \dots, \mathbf{u}_{r_{2}}^{(2)}$ are linearly independent, from which we obtain 
\[
\mathbf{u}^{(2)}_l \otimes \cdots \otimes \mathbf{u}^{(k)}_l = \sum_{t=1}^{r_2}  \mathbf{u}_t^{(2)} \otimes \mathcal{B}_t,
\]
where for each $1\le t \le r_2$, $\mathcal{B}_t$ is a scalar multiple of $\mathbf{u}^{(3)}_l \otimes \cdots \otimes \mathbf{u}^{(k)}_l$ and moreover it is a linear combination of $(p-\delta_0 - r_2+1)$ tensors: 
\[
a_{lt} \mathbf{u}^{(3)}_t \otimes \cdots \mathbf{u}^{(k)}_t, \quad a_{l,r_2+1} \mathbf{u}^{(3)}_{r_2+1} \otimes \cdots \mathbf{u}^{(k)}_{r_2+1},\quad \dots, \quad a_{l,p-\delta_0} \mathbf{u}^{(3)}_{p-\delta_0} \otimes \cdots \mathbf{u}^{(k)}_{p-\delta_0}.
\]
By an induction on $k$ and the fact that the contribution of the constraint 
\[
\rank \left( [ \mathbf{u}^{(2)}_1,\dots, \mathbf{u}^{(2)}_{p-\delta_0}] \right) = r_2
\] 
to $\operatorname{codim} (V_{\mathbf{n},r}(i_1,\dots, i_p;\delta) )$ is $(n_2 - r_2)(p-\delta_0 - r_2)$, we obtain that the contribution of \eqref{lem:defectivity s=1:eq1} to $\operatorname{codim} (V_{\mathbf{n},r}(i_1,\dots, i_p;\delta) )$ in this case is greater than the quantity displayed in the following \eqref{lem:defectivity s=1:eq6}. 

Lastly, we let $l$ run through the index set $\{p-\delta_0 + 1,\dots, p\}$ and conclude that the total contribution of \eqref{lem:defectivity s=1:eq1} to $\operatorname{codim} (V_{\mathbf{n},r}(i_1,\dots, i_p;\delta) )$ is at least 
\begin{equation}\label{lem:defectivity s=1:eq6}
\delta_0 \left(  (n_2 -1) + \cdots + (n_{i'_0-1} -1) + (n_{i'_0}- (p - \delta_0) ) + (n_{i'_0+1} -1) + \cdots + (n_k -1)  \right),
\end{equation}
which occurs in the case $r_{i'_0} = p - \delta_0$ for some $2 \le i'_0 \le k$.
\item Contribution of \eqref{eq:beta2}: We derive by a similar argument that the contribution of the set
\[
\lbrace
\mathcal{A}_t \in \operatorname{span}(\mathcal{A}_{1},\dots, \mathcal{A}_{p-\delta_0}): p+1 \le t \le r
\rbrace
\]
to $\operatorname{codim} (V_{\mathbf{n},r}(i_1,\dots, i_p;\delta) )$ is at least
\begin{equation}\label{lem:defectivity s=1:eq7}
\beta_2 \left(  (n_2 -1) + \cdots + (n_{i'_0-1} -1) + (n_{i'_0}- (p - \delta_0) ) + (n_{i'_0+1} -1) + \cdots + (n_k -1) \right).
\end{equation}
We notice that the quantities in \eqref{lem:defectivity s=1:eq6} and \eqref{lem:defectivity s=1:eq7} are independent of the index $i_0'$ since the factor in parentheses can be rewritten as 
\[
\left( \sum_{j=2}^k (n_j -1)  \right) + 1 - (p-\delta_0).
\]
\item Contribution of \eqref{eq:gamma}: Since $\mathbf{u}^{(1)}_1,\dots, \mathbf{u}^{(1)}_p$ are pairwise orthogonal, we may split \eqref{eq:gamma} as
\begin{equation}\label{lem:defectivity s=1:eq8}
\sum_{i'=2}^k \mathbf{u}^{(2)}_l \otimes \cdots \otimes \mathbf{u}^{(i'-1)}_l \otimes \mathbf{y}_l^{(i')} \otimes \mathbf{u}^{(i'+1)}_l \otimes \cdots \otimes \mathbf{u}^{(k)}_l = \sum_{m = 1}^p c_{lm} \lambda_m \mathcal{A}_m,   \quad l =1,\dots, p.
\end{equation}
Here we adopt the convention that $c_{ml}  = - c_{lm}$ for $1\le l < m \le p$ and $c_{11} = \cdots = c_{pp} = 0$. We recall from \eqref{lem:defectivity s=1:eq1} that $\mathcal{A}_l = \sum_{m=1}^{p - \delta_0} a_{lm} \mathcal{A}_m$ for each $p - \delta_0 +1 \le l \le p$. {Hence \eqref{lem:defectivity s=1:eq8}} can be written as
\begin{equation}\label{lem:defectivity s=1:eq9}
\sum_{i'=2}^k \mathbf{u}^{(2)}_l \otimes \cdots \otimes \mathbf{u}^{(i'-1)}_l \otimes \mathbf{y}_l^{(i')} \otimes \mathbf{u}^{(i'+1)}_l \otimes \cdots \otimes \mathbf{u}^{(k)}_l = \sum_{m = 1}^{p - \delta_0} \widetilde{c}_{lm} \mathcal{A}_m,\quad \quad l =1,\dots, p,
\end{equation}
\normalsize
where $\widetilde{c}_{lm}$ is defined for each $l =1,\dots, p $ and $m=1, \dots, p-\delta_0$ by
\[
\widetilde{c}_{lm} \coloneqq c_{lm} \lambda_m + \sum_{j=p -\delta_0 + 1}^p c_{lj}\lambda_j a_{jm}.
\]
We define a linear transformation
\[
\eta: C = \begin{bmatrix}
0 & c_{12} & \cdots & c_{1p}\\
-c_{12} & 0 & \cdots & c_{2p} \\
\vdots & \vdots & \ddots & \vdots \\
-c_{1p} & -c_{2p} & \cdots & 0
\end{bmatrix}   \mapsto \widetilde{C} = \begin{bmatrix}
0 & c_{12} & \cdots & c_{1p}\\
-c_{12} & 0 & \cdots & c_{2p} \\
\vdots & \vdots & \ddots & \vdots \\
-c_{1p} & -c_{2p} & \cdots & 0
\end{bmatrix}  \begin{bmatrix}
\lambda_1 & 0 & \cdots & 0\\
0 & \lambda_2 & \cdots & 0 \\
\vdots & \vdots & \ddots & \vdots \\
0 & 0 & \cdots & \lambda_p
\end{bmatrix}
\begin{bmatrix}
I_{p-\delta_0}   \\
A
\end{bmatrix},
\]
where $A$ is the $\delta_0 \times (p-\delta_0)$ matrix defined by
\[
A = \begin{bmatrix}
a_{p-\delta_0 + 1,1} &  \cdots & a_{p-\delta_0 + 1,p-\delta_0} \\
\vdots & \ddots & \vdots \\
a_{p,1} &  \cdots & a_{p,p-\delta_0}
\end{bmatrix}.
\]
It is clear that the $(l,m)$-th element of $\widetilde{C}$ is $\widetilde{c}_{lm}$ for $1 \le l \le p$ and $1\le m \le p-\delta_0$. We partition $C$ (resp. $\operatorname{diag}_2(\lambda_1,\dots, \lambda_p)$) as $\begin{bmatrix}
C_1 & C_2 \\
-C_2^\tp & C_3
\end{bmatrix}$ (resp. $\begin{bmatrix}
\Lambda_1 & 0 \\
0 & \Lambda_2
\end{bmatrix}$) so that
\[
\eta(C) = \begin{bmatrix}
C_1\Lambda_1 + C_2\Lambda_2 A \\
-C_2^\tp\Lambda_1 + C_3\Lambda_2A
\end{bmatrix}.
\]
This implies that the kernel of $\eta$ is given by
\begin{equation}\label{eq:defectivity-kernel}
\ker(\eta) = \left\lbrace \begin{bmatrix}{\Lambda_1^{-1} A^\tp \Lambda_2 C_3^\tp \Lambda_2 A \Lambda_1^{-1}} & {-\Lambda_1^{-1} A^\tp \Lambda_2 C_3^\tp} \\
-C_3 \Lambda_2 A \Lambda_1^{-1} & C_3
\end{bmatrix}: C_3 \in \mathbb{R}^{\delta_0 \times \delta_0}, C_3^\tp = - C_3
\right\rbrace.
\end{equation}
In particular, $\gamma \ge \dim (\ker (\eta)) =\binom{\delta_0}{2}$. We may assume that $\gamma \ge \binom{\delta_0}{2}+ \delta_0$ since otherwise the conclusion follows immediately. In fact, if we denote $\beta \coloneqq \sum_{i'=2}^k (n_{i'} - 1)$, then we have
\footnotesize
\begin{align*}
\operatorname{codim} (V_{\mathbf{n},r}(i_1,\dots, i_p;\delta)) - \delta
&\geq
(\delta_0 + \beta_2) (\beta - (n_{i_0'} - 1) + (n_{i_0'} - (p - \delta_0))) - \left( \delta_0(n_1 - p) + \beta_2 + \gamma \right) \\
&\ge \delta_0 (\beta - (p-\delta_0) + 1) - \delta_0(n_1 - p) - \gamma \\
 & = \delta_0(\beta - n_1 + \delta_0 + 1)-\gamma\\
 &\geq \delta_0(\delta_0+1)-\gamma.
\end{align*}
\normalsize
Thus the inequality $\gamma<\binom{\delta_0}{2} + \delta_0$ implies that $\operatorname{codim} (V_{\mathbf{n},r}(i_1,\dots, i_p;\delta)) > \delta$.

{In the following, we assume that $\gamma \ge \binom{\delta_0}{2} + \delta_0$.} {We first claim that} the contribution of the linear subspace $\left( \operatorname{span} (B_2 \sqcup B_3) + \operatorname{ span}(C_2) \right)$ to $\operatorname{codim} (V_{\mathbf{n},r}(i_1,\dots, i_p;\delta))$ is bounded below by
\begin{equation}\label{lem:defectivity s=1:eq10}
\left( \gamma - \binom{\delta_0}{2} - \delta_0 \right) \left( \beta+1-\max_{2\le i' \le k} \{ n_{i'} \} \right).
\end{equation}
To this end, for each $C\notin\ker(\eta)$ such that $\widetilde{C} = \eta(C)$ satisfies \eqref{lem:defectivity s=1:eq9}, we count the contribution to $\operatorname{codim} (V_{\mathbf{n},r}(i_1,\dots, i_p;\delta))$ of each nonzero entry of $\widetilde{C}$. 

By the same argument as before, the least contribution occurs when 
\[
\rank([\mathbf{u}^{(i')}_1,\dots, \mathbf{u}^{(i')}_{p - \delta_0}, \mathbf{u}^{(i')}_l])= p-\delta_0 + 1
\]
for some $2\le i' \le k$. We rewrite \eqref{lem:defectivity s=1:eq9} as
\begin{align*}
\mathbf{y}^{(2)}_l \otimes \mathbf{u}^{(3)}_l \otimes \cdots \otimes \mathbf{u}^{(k)}_l &= \sum_{m = 1}^{p-\delta_0} \mathbf{u}^{(2)}_m \otimes \left( \widetilde{c}_{lm} \mathbf{u}_m^{(3)} \otimes \cdots \otimes \mathbf{u}_m^{(k)} \right) \\
&- \mathbf{u}^{(2)}_l \otimes \left(
\sum_{i'=3}^k \mathbf{u}^{(3)}_l \otimes \cdots \otimes \mathbf{u}^{(i'-1)}_l \otimes \mathbf{y}^{(i')}_l \otimes \mathbf{u}^{(i'+1)}_l \otimes \cdots \otimes \mathbf{u}^{(k)}_l
\right),
\end{align*}
from which we obtain that tensors
\footnotesize
\[
\widetilde{c}_{l1}  \mathbf{u}_1^{(3)} \otimes \cdots \otimes \mathbf{u}_1^{(k)},  \dots, \widetilde{c}_{l,p-\delta_0}  \mathbf{u}_{{p-\delta_0}}^{(3)} \otimes \cdots \otimes \mathbf{u}_{{p-\delta_0}}^{(k)},
\sum_{i'=3}^k \mathbf{u}^{(3)}_l \otimes \cdots \otimes \mathbf{u}^{(i'-1)}_l \otimes \mathbf{y}^{(i')}_l \otimes \mathbf{u}^{(i'+1)}_l \otimes \cdots \otimes \mathbf{u}^{(k)}_l
\]
\normalsize
are only differed by a scalar multiple. Next we discuss with respect to different cases:
\begin{enumerate}[label= (\roman*)]
\item There exist $1\le m_1 < \cdots < m_q \le p-\delta_0, q \ge 2$ such that $\widetilde{c}_{lm_1},\dots,  \widetilde{c}_{lm_q} \ne 0$. In this case, we have
\[
\mathbf{u}^{(3)}_{m_1} =  \cdots = \pm \mathbf{u}^{(3)}_{m_q},\dots, \mathbf{u}^{(k)}_{m_1} = \cdots  = \pm \mathbf{u}^{(k)}_{m_q}.
\]
\item There exist $1\le m_0 \le p-\delta_0$ such that $\widetilde{c}_{lm_0} \ne 0$ but all other $\widetilde{c}_{lm} = 0$. In this case, we have
\begin{enumerate}[label=(\alph*)]
\item If $\sum_{i'=3}^k \mathbf{u}^{(3)}_l \otimes \cdots \otimes \mathbf{u}^{(i'-1)}_l \otimes \mathbf{y}^{(i')}_l \otimes \mathbf{u}^{(i'+1)}_l \otimes \cdots \otimes \mathbf{u}^{(k)}_l = 0$, then we have
\[
\mathbf{y}^{(2)}_l \otimes \mathbf{u}^{(3)}_l \otimes \cdots \otimes \mathbf{u}^{(k)}_l = \widetilde{c}_{lm_0} \mathbf{u}^{(2)}_{m_0} \otimes  \mathbf{u}_{m_0}^{(3)} \otimes \cdots \otimes \mathbf{u}_{m_0}^{(k)},
\]
from which we may derive
\[
\mathbf{u}^{(3)}_l = \pm \mathbf{u}^{(3)}_{m_0},\dots, \mathbf{u}^{(k)}_l = \pm \mathbf{u}^{(k)}_{m_0}.
\]
\item If $\sum_{i'=3}^k \mathbf{u}^{(3)}_l \otimes \cdots \otimes \mathbf{u}^{(i'-1)}_l \otimes \mathbf{y}^{(i')}_l \otimes \mathbf{u}^{(i'+1)}_l \otimes \cdots \otimes \mathbf{u}^{(k)}_l \ne 0$, then either $ \mathbf{u}^{(2)}_l = \pm \mathbf{u}^{(2)}_{m_0}$ or
\[
\mathbf{u}^{(3)}_l = \pm \mathbf{u}^{(3)}_{m_0},\dots, \mathbf{u}^{(k)}_l = \pm \mathbf{u}^{(k)}_{m_0}.
\]
If $\mathbf{u}^{(2)}_l = \pm \mathbf{u}^{(2)}_{m_0}$ holds, then we have
\footnotesize
\begin{equation}\label{eq:induction}
\sum_{i'=3}^k \mathbf{u}^{(3)}_l \otimes \cdots \otimes \mathbf{u}^{(i'-1)}_l \otimes \widetilde{\mathbf{y}}^{(i')}_l \otimes \mathbf{u}^{(i'+1)}_l \otimes \cdots \otimes \mathbf{u}^{(k)}_l = {\widetilde{c}'_{lm_0}}  \mathbf{u}_{m_0}^{(3)} \otimes \cdots \otimes \mathbf{u}_{m_0}^{(k)}
\end{equation}
\normalsize
for some nonzero $\widetilde{c}'_{lm_0}$ and $\widetilde{\mathbf{y}}^{(i')}_l$'s, which can be further discussed by an induction on the number of factors of tensors {in \eqref{eq:induction}}.
\end{enumerate}
{In summary, we obtain that the contribution of each nonzero entry {$\widetilde{c}_{lm}$ ($l \neq m$) of $\widetilde{C}$} to $\operatorname{codim} (V_{\mathbf{n},r}(i_1,\dots, i_p;\delta))$ is at least
\[
\beta+1-\max_{i'} \{ n_{i'} \}.
\]
Moreover, we recall that there are already $\delta_0$ pairs of $(l,i'_0)$ satisfying \eqref{lem:defectivity s=1:eq5} and ${\operatorname{dim}(\im(\eta))} \ge \gamma - \binom{\delta_0}{2}$, from which we obtain the desired lower bound  \eqref{lem:defectivity s=1:eq10} for the overall contribution of $\left( \operatorname{span} (B_2 \sqcup B_3) + \operatorname{ span}(C_2) \right)$ to $\operatorname{codim} (V_{\mathbf{n},r}(i_1,\dots, i_p;\delta))$.}
\end{enumerate}
\end{enumerate}

\if 
We either have
\begin{equation}\label{lem:defectivity s=1:eq2}
\rank \left( [\mathbf{u}^{(i')}_1,\dots, \mathbf{u}^{(i')}_{p-\delta_0}] \right) = 1,\quad i' = 2,\dots, k,
\end{equation}
or there exists some $2 \le i'_0 \le k$ such that the $(p-\delta_0)$ tensors
\begin{equation}\label{lem:defectivity s=1:eq3}
{a_{l1}} \mathbf{u}^{(2)}_1 \otimes \cdots \otimes  \widehat{\mathbf{u}^{(i_0')}_1} \otimes \cdots  \otimes \mathbf{u}^{(k)}_1,\dots, {a_{l,p-\delta_0}} \mathbf{u}^{(2)}_{p-\delta_0} \otimes \cdots \otimes  \widehat{\mathbf{u}^{(i_0')}_{p-\delta_0}} \otimes \cdots  \otimes \mathbf{u}^{(k)}_{p-\delta_0}
\end{equation}
are pairwise differed by a multiple scalar. Indeed, if $\rank \left( [\mathbf{u}^{(i_0')}_1,\dots, \mathbf{u}^{(i_0')}_{p-\delta_0}] \right) \ge 2$ for some $2 \le i_0' \le k$, then by comparing ranks of the $i'_0$-th flattening of both sides of \eqref{lem:defectivity s=1:eq1}, we may obtain the desired conclusion.

The former case \eqref{lem:defectivity s=1:eq2} together with \eqref{lem:defectivity s=1:eq1} can further imply that
\begin{equation}\label{lem:defectivity s=1:eq4}
\rank \left( [\mathbf{u}^{(i')}_1,\dots, \mathbf{u}^{(i')}_{p}] \right) = 1,\quad i' = 2,\dots, k.
\end{equation}
For the latter case, we may infer from the combination of \eqref{lem:defectivity s=1:eq3} and \eqref{lem:defectivity s=1:eq1} that for each $p - \delta_0 + 1 \le l \le p$, there exists some $1\le j \le p - \delta_0$ such that
\footnotesize
\begin{equation}\label{lem:defectivity s=1:eq5}
\mathbf{u}_l^{(i'_0)} \in {\operatorname{span} (\mathbf{u}_{1}^{(i'_0)},\dots,\dots, \mathbf{u}_{p-\delta_0}^{(i'_0)})}, \mathbf{u}_l^{(2)} = \pm \mathbf{u}_{j}^{(2)},\dots, \mathbf{u}_l^{(i'_0-1)} = \pm \mathbf{u}_{j}^{(i'_0-1)}, \mathbf{u}_l^{(i'_0+1)} = \pm \mathbf{u}_{j}^{(i'_0+1)},\dots, \mathbf{u}_{l}^{(k)} = \pm \mathbf{u}_{j}^{(k)}.
\end{equation}
\normalsize
Hence the contribution of \eqref{lem:defectivity s=1:eq1} to $\operatorname{codim} (V_{\mathbf{n},r}(i_1,\dots, i_p;\delta) )$ is at least
\begin{equation}\label{lem:defectivity s=1:eq6}
\delta_0 \left(  (n_2 -1) + \cdots + (n_{i'_0-1} -1) + (n_{i'_0}- (p - \delta_0) ) + (n_{i'_0+1} -1) + \cdots + (n_k -1) \right).
\end{equation}
\fi
We notice that contributions \eqref{lem:defectivity s=1:eq6}, \eqref{lem:defectivity s=1:eq7} and \eqref{lem:defectivity s=1:eq10} are all disjoint. Hence $\left( \operatorname{codim} (V_{\mathbf{n},r}(i_1,\dots, i_p;\delta)) - \delta \right)$ is at least
\small
\begin{align*}
&\left( \gamma - \binom{\delta_0}{2} - \delta_0 \right) {(\beta +1-\max_{i'} \{ n_{i'} \} )}  + (\delta_0 + \beta_2) (\beta {- (p-\delta_0) + 1}) - \left( \delta_0(n_1 - p) + \beta_2 + \gamma \right) \\
&\ge \left( \gamma - \binom{\delta_0}{2} - \delta_0 \right) {(\beta +1-\max_{i'} \{n_{i'} \} )} + \delta_0 (\beta - (p-\delta_0) + 1) - \delta_0(n_1 - p) - \gamma \\
&{\ge} {\left( \gamma- \binom{\delta_0}{2}-\delta_0 \right)} (\beta {-\max_{i'} \{n_{i'} \} }) - {\left( \binom{\delta_0}{2}+\delta_0 \right) + \delta_0  (\beta+1 - n_1+\delta_0) }\\
 &{\ge \delta_0(\beta - n_1 + \frac{\delta_0 + 1}{2})},
\end{align*}
\normalsize
since $p\leq r\leq\max\{n_{i'}\} < \beta$ according to \eqref{lem:defectivity s=1:eq02}. If $\delta_0 > 0$, then \eqref{lem:defectivity s=1:eq01} guarantees that
\[
\delta_0(\beta - n_1 + \frac{\delta_0 + 1}{2})  > 0.
\]
If $\delta_0 = 0$, then $\delta = \beta_ 2 + \gamma >0
$ implies either $\beta_2 > 0$ or $\gamma > 0$. Hence we have
\begin{align*}
\operatorname{codim} (V_{\mathbf{n},r}(i_1,\dots, i_p;\delta)) - \delta &\ge \gamma {(\beta +1-\max_{i'} \{n_{i'} \} )} + \beta_2 (\beta - p +1) - (\beta_2 + \gamma) \\
&= \gamma(\beta - {\max_{i'} \{n_{i'}\} }) + \beta_2 (\beta - p) \\
&> 0,
\end{align*}
which follows from \eqref{lem:defectivity s=1:eq02}.
\end{proof}

The rest part of the discussion is similar to the one in Subsection~\ref{sec:generic-s2}.

\begin{lemma}\label{lem:fiber pi1 2}
Let $s = 1$ and $V_{\mathbf{n},r} (i_1,\dots, i_p;\delta)$ be defined as before. For any $y \in V_{\mathbf{n},r} (i_1,\dots, i_p;\delta)$, the dimension of the fiber $\pi_2^{-1}(y)$ is a constant. Thus for an irreducible subvariety $Z$ of  $V_{\mathbf{n},r} (i_1,\dots, i_p;\delta)$, $\pi_2^{-1}(Z)$ is irreducible of dimension
\[
\dim (Z) + \prod_{i=1}^k n_i -  \dim (d_y \varphi_{\mathbf{n},r} (\Tt_{V_{\mathbf{n},r}} (y))).
\]
\end{lemma}

\begin{lemma}\label{lem:s=1 positive delta}
Let $s = 1$ and let $V_{\mathbf{n},r} (i_1,\dots, i_p;\delta)$ be defined as before. Assume that $\delta>0$, and \eqref{lem:defectivity s=1:eq01} and \eqref{lem:defectivity s=1:eq02} hold, then we have
\[
\dim (\pi_1 (\pi_2^{-1}(V_{\mathbf{n},r} (i_1,\dots, i_p;\delta)) )) < {\prod_{i=1}^k n_i}.
\]
\end{lemma}

\begin{proof}
By Lemma~\ref{lem:defectivity s=1}, we have
\[
\dim (V_{\mathbf{n},r} ) -  \dim (d_y \varphi_{\mathbf{n},r}(\Tt_{V(\mathbf{n},r)}(y))) = \delta  < \dim (V_{\mathbf{n},r} )  - \dim (V_{\mathbf{n},r} (i_1,\dots, i_p;\delta)).
\]
The rest of the proof is similar to that of Lemma~\ref{lem:s=2 nonempty E}.
\end{proof}

Equipped with Lemmas~\ref{lem:fiber pi1 2} and \ref{lem:s=1 positive delta}, one can prove the following result by the same argument employed in the proof of Proposition~\ref{prop:KKT location s=2}.
\begin{proposition}\label{prop:KKT location s=1}
Assume that $s = 1$. If \eqref{lem:defectivity s=1:eq01} and \eqref{lem:defectivity s=1:eq02} hold, then for a generic $\mathcal{A}$, all KKT points of \eqref{lem:KKT parameter:optproblem} are contained in $W_{\mathbf{n},r}$.
\end{proposition}

We notice that in most scenarios, we may assume $n_1 = \cdots = n_k = n$. In this case, conditions \eqref{lem:defectivity s=1:eq01} and \eqref{lem:defectivity s=1:eq02} reduce to $n \ge 1 +\frac{2}{k-2}$, which can be further rewritten as $n \ge 2 + \delta_{k,3}$, where $\delta_{k,3}$ is the Kronecker delta function. Hence we obtain the following
\begin{corollary}\label{cor:KKT location s=1}
Assume that $s = 1$, $k\ge 3$ and $n_1 = \cdots = n_k \ge 2$. For a generic $\mathcal{A}$, all KKT points of \eqref{lem:KKT parameter:optproblem} are contained in $W_{\mathbf{n},r}$.
\end{corollary}
\begin{proof}
Let $n$ be the common value of $n_1,\dots, n_k$. If $n =2, k \ge 4$ or $n\ge 3, k \ge 3$, then the conclusion directly follows from Proposition~\ref{prop:KKT location s=1}. Therefore, it is sufficient to deal with the remaining case where $n = 2, k =3$, which can be easily addressed by the same analysis we have done in Lemma~\ref{lem:defectivity s=1}.
\end{proof}

\subsection{Local differemorphism}\label{sec:local}
We recall that Propositions~\ref{prop:KKT location s=3}, \ref{prop:KKT location s=2} and \ref{prop:KKT location s=1} locate all KKT points of our original problem \eqref{eq:sota} in a smaller parameter space $W_{\mathbf{n},r}$, for a generic tensor $\mathcal{A}$. Since the tensor space parametrized by $W_{\mathbf{n},r}$ is $Q_{s} (\mathbf{n},r)$, we may shrink accordingly the feasible domain of \eqref{eq:sota} and \eqref{eq:original} to $W_{\mathbf{n},r}$ and $Q_{s} (\mathbf{n},r)$ respectively. Moreover, according to Proposition~\ref{prop:smooth}-\eqref{prop:smooth:item6} and Figure~\ref{fig:relations}, $Q_s(\mathbf{n},r) = \bigsqcup_{t = 0}^r R_s(\mathbf{n},t)$ and each $R_s(\mathbf{n},t)$ is parametrized by $W_{\mathbf{n},t,\ast}$. Therefore, the analysis of convergent behaviour of Algorithm~\ref{algo} relies heavily on the relation between $R_s(\mathbf{n},t)$ and $W_{\mathbf{n},t,\ast}$, which is the focus of this subsection.

To begin with, we prove in the following lemma that $R_s(\mathbf n,r)$ and $W_{\mathbf n,r,\ast}$ are locally the same.
\begin{lemma}[Local Diffeomorphism]\label{lem:localdiff}
For any positive integers $n_1,\dots,n_k$ and $r\leq \min\{n_1,\dots,n_k\}$, the set $W_{\mathbf n,r,\ast}$ is a smooth manifold and is locally diffeomorphic to the manifold $R_s(\mathbf n,r)$.
\end{lemma}

\begin{proof}
We recall from Proposition~\ref{prop:smooth} that $W_{\mathbf n,r,++}$ is a principle {$D_{r,k}\rtimes_{\eta} S_r$-bundle on $R_s(\mathbf{n},r)$ and each connected components \footnote{There are $2^{r + c}$ connected components where $c = \# \{i: n_i = r, 1\le i \le s\}$ since $\mathbb{R}_{\ast}$ has two connected components and $V(r,n_i) = \O(n_i)$ also has two connected components if $r = n_i$.} of $W_{\mathbf n,r,\ast}$ is diffeomorphic to $W_{\mathbf n,r,++}$}. In particular, since $D_{r,k}\rtimes_{\eta} S_r$ is a finite group, for any $\mathcal{T}\in R_s(\mathbf{n},r)$, the fiber $\varphi_{\mathbf{n},r}^{-1}(\mathcal{T})$ of the map
\[
\varphi_{\mathbf{n},r}: W_{\mathbf n,r,\ast} \to R_s(\mathbf{n},r)
\]
consists of $e \coloneqq 2^{rk} r!$ points. Therefore, for a small enough neighbourhood $V \subseteq R_s(\mathbf{n},r)$ of $\mathcal{T}$, the inverse image $\varphi_{\mathbf{n},r}^{-1}(V)$ is the disjoint union of $e$ open subsets ${W_1,\dots, W_{e}}\subseteq W_{\mathbf n,r,\ast}$ and for each $j=1,\dots, e$, we have that
\[
\varphi_{\mathbf{n},r}|_{W_j}: W_j \to {V}
\]
is a diffeomorphism.
\end{proof}

We recall that the objective function of \eqref{eq:original} is the half of the squared distance of a given tensor $\mathcal{A}$ to $P_s(\mathbf{n},r)$. The restriction of this function to $R_s(\mathbf n,r)$ can be parametrized by $W_{\mathbf n,r,\ast}$ as
\begin{equation}\label{eq:objective-p}
g( U,\mathbf x)=\frac{1}{2}\|\mathcal A-(U^{(1)},\dots,U^{(k)})\cdot\operatorname{diag}_k(\mathbf x)\|^2,\quad (U,\mathbf x) \in W_{\mathbf n,r,\ast}.
\end{equation}
By applying Proposition~\ref{prop:critical}, Lemmas~\ref{lem:generic-morse} and \ref{lem:localdiff}, we obtain the following
\begin{proposition}[Generic Nondegeneracy]\label{prop:nondegnerate}
For a generic tensor $\mathcal A$, each critical point of the function $g$ in $W_{\mathbf n,r,\ast}$ is nondegenerate.
\end{proposition}

We remind the reader that a critical point $( U,\mathbf x)$ of $g$ is a point at which the Riemannian gradient $\operatorname{grad} (g)( U,\mathbf x)$ of $g$ at {$( U,\mathbf{x})$} is zero. If we embed $W_{\mathbf n,r,\ast}$ into $\mathbb R^{n_1\times r}\times\dots\times\mathbb R^{n_k\times r}\times\mathbb R^r$ in a tutorial way such that $W_{\mathbf n,r,\ast}$ becomes an embedded submanifold, then the vanishing of $\operatorname{grad} (g)( U,\mathbf x)$ is equivalent to {the fact that} the projection of the Euclidean gradient $\nabla g( U,\mathbf x)$ onto the tangent space of $W_{\mathbf n,r,\ast}$ at $( U,\mathbf x)$ is zero. Thus to characterize critical points of $g$, the tangent space of $W_{\mathbf n,r,\ast}$ is a necessary ingredient. Since $W_{\mathbf n,r,\ast}$ is the product of Stiefel manifolds, Oblique manifolds and Euclidean lines, the tangent space of $W_{\mathbf n,r,\ast}$ can be easily computed.
\begin{proposition}[Tangent Space of $W_{\mathbf n,r,\ast}$]\cite{AMS-08,EAT-98}\label{prop:tangent}
The tangent space of $W_{\mathbf n,r,\ast}$ at $( U,\mathbf x) \in W_{\mathbf n,r,\ast}$ is
\begin{multline}\label{eq:tangent-para}
\operatorname{T}_{W_{\mathbf n,r,\ast}}( U,\mathbf x)
=\operatorname{T}_{\V(r,n_1)}(U^{(1)}) \times\dots\times\operatorname{T}_{\V(r,n_s)}(U^{(s)})\times\\
\operatorname{T}_{\OB(r,n_{s+1})}(U^{(s+1)})
\times\dots\times\operatorname{T}_{\OB(r,n_{k})}(U^{(k)}) \times\mathbb R^r,
\end{multline}
where
\begin{equation}\label{eq:stiefel-tangent}
\operatorname{T}_{\V(r,n_i)}(U^{(i)})=\{Z\in\mathbb R^{n_i\times r}\colon (U^{(i)})^\tp Z+Z^\tp U^{(i)}=0\},\quad i=1,\dots,s,
\end{equation}
and
\begin{equation}\label{eq:oblique-tangent}
\operatorname{T}_{\OB(r,n_i)}(U^{(i)})=\{Z\in\mathbb R^{n_i\times r}\colon \big((U^{(i)})^\tp Z\big)_{11} = \cdots = \big((U^{(i)})^\tp Z\big)_{rr} =  0\}, \quad i=s+1,\dots,k.
\end{equation}
\end{proposition}

{Based on \eqref{eq:tangent-para} and the previous discussion, the vanishing condition of $\operatorname{grad} (g)( U,\mathbf x)$ can be explicitly characterized by the following two general facts, which we record for easy reference.}
\begin{lemma}\cite{HY-19}\label{lem:grand}
Let $A\in \V(r,n)$ and let $f : \V(r,n)\subseteq \mathbb R^{n\times r}\rightarrow \mathbb R$ be a smooth function. Then $\operatorname{grad}(f)(A)=0$ if and only if
\begin{equation}\label{eq:symmetry}
\nabla f(A) = A (\nabla f(A))^\tp A,
\end{equation}
which is also equivalent to $\nabla f(A)=AP$ for some $r\times r$ symmetric matrix $P$. In particular, $A^\tp \nabla f(A)$ is a symmetric matrix.
\end{lemma}

{We notice that $\OB(r,n)$ is an open submanifold of $\B(r,n)$, which is a product of $r$ copies of $\mathbb S^{n-1}=\V(1,n)$. Thus, an application of Lemma~\ref{lem:grand} gives the following}
\begin{lemma}\label{lem:grand-oblique}
Let $A\in \OB(r,n)$ and let $f : \OB(r,n)\subseteq \mathbb R^{n\times r}\rightarrow \mathbb R$ be a smooth function. Then $\operatorname{grad}(f)(A)=0$ if and only if
\begin{equation}\label{eq:symmetry}
\nabla f(A) = A \diag_{{2}}\big((\nabla f(A))^\tp A\big).
\end{equation}
\end{lemma}

\section{KKT Points}\label{sec:kkt}
In this section, we characterize KKT points of problem \eqref{eq:sota} and study the connection between these points and critical points of the squared distance function of a given tensor to $R_s(\mathbf n,r)$. By Proposition~\ref{prop:sota-max} and Lemma~\ref{lem:kkt-equiv}, \eqref{eq:sota} is equivalent to the maximization problem \eqref{eq:sota-max} and their respective {KKT} points are in a one-to-one correspondence. Therefore our goal can be achieved by understanding KKT points of \eqref{eq:sota-max}, which can be written more explicitly as
\begin{flalign}\label{eq:sota-max-matrix}
&\text{(mLRPOTA{(r)})}\
\begin{array}{rl}
\max& \sum_{j=1}^r\Big(\big(\big(U^{(1)}\big)^\tp ,\dots,\big(U^{(k)}\big)^\tp \big)\cdot\mathcal A\Big)_{{j,\dots,j}}^2\\
 \text{s.t.} &\big(U^{(i)}\big)^\tp U^{(i)}=I_r \  \text{for all }i\in\{1,\dots,s\},\\
& \big((U^{(i)})^\tp U^{(i)}\big)_{11} = \cdots = \big((U^{(i)})^\tp U^{(i)}\big)_{rr} =1 \ \text{for all }i\in\{s+1,\dots,k\}.
\end{array}&
\end{flalign}
As we have remarked after the proof of Proposition~\ref{prop:sota-max}, the discussion in this section involves the problem \eqref{eq:sota-max} or equivalently \eqref{eq:sota-max-matrix} for different values of $r$, hence we name the problem by mLRPOTA($r$) to indicate the specific problem we are working with.

Before we proceed, we fix some notations. Let $ U := (U^{(1)},\dots,U^{(k)})$ be the $k$-tuple of  variable matrices in \eqref{eq:sota-max-matrix}. For each $1 \le i \le k$ and $1 \le j \le r$, let $\mathbf u^{(i)}_j$ be the $j$-th column of the matrix $U^{(i)}$ and define
\begin{align}
\mathbf v^{(i)}_j &\coloneqq \mathcal{A}\tau_i(\mathbf u^{(1)}_j,\dots,\mathbf u^{(k)}_j)\in \mathbb{R}^{n_i}, \\
V^{(i)} &\coloneqq \begin{bmatrix}\mathbf v^{(i)}_1&\dots&\mathbf v^{(i)}_r\end{bmatrix} \in \mathbb{R}^{n_i \times r}, \label{eq:critial}\\
\lambda_j( U) &\coloneqq \Big(\big(\big(U^{(1)}\big)^\tp ,\dots,\big(U^{(k)}\big)^\tp \big)\cdot\mathcal A\Big)_{j\dots j} =\langle\mathcal A,\mathbf u^{(1)}_j\otimes\dots\otimes\mathbf u^{(k)}_j\rangle \in \mathbb{R}, \label{eq:lambda} \\
\Lambda &\coloneqq \operatorname{diag}_2 \left(\lambda_1( U),\dots, \lambda_r( U)  \right) \in \mathbb{R}^{r\times r}. \label{eq:lambda-tensor}
\end{align}
%

The objective function of \eqref{eq:sota-max-matrix} is thus written as
\begin{equation}\label{eq:objective}
f( U) \coloneqq
\sum_{j=1}^r\Big(\big(\big(U^{(1)}\big)^\tp ,\dots,\big(U^{(k)}\big)^\tp \big)\cdot\mathcal A\Big)_{j\dots j}^2 = \sum_{j=1}^r\lambda_j( U)^2.
\end{equation}
\subsection{Basic properties}\label{sec:kkt-basics}

Since \eqref{eq:sota-max-matrix} is a nonlinear optimization problem with equality constraints of continuously differentiable functions, the general theory discussed in Subsection~\ref{subsec:KKT} applies to \eqref{eq:sota-max-matrix} and the KKT condition can be explicitly obtained. 
The following characterization of a KKT point of \eqref{eq:sota-max-matrix} can be obtained by a direct computation.
\begin{proposition}[KKT Point]\label{def:kkt}
A feasible point $U = (U^{(1)},\dots,U^{(k)})$ of \eqref{eq:sota-max-matrix} is a KKT point with a Lagrange multiplier $ P=  (P_1,\dots,P_s,{\mathbf p})$ where $P_i$ is an $r\times r$ symmetric matrix for $i=1,\dots,s$ and ${\mathbf p} \in\mathbb R^r$, if and only if $U$ and $P$ satisfy the following system
\begin{equation}\label{eq:kkt}
\begin{cases}V^{(i)}\Lambda-U^{(i)}P_i=0,& \text{if }i=1,\dots,s,\\
V^{(i)}\Lambda-U^{(i)} \operatorname{diag}_2({\mathbf p})=0,& \text{if } i=s+1,\dots,k.
\end{cases}
\end{equation}
\end{proposition}

\begin{proof}
With notations as above, we have
\[
\nabla_{U^{(i)}} f(U)=2V^{(i)}\Lambda\ \text{for all }i=1,\dots,k.
\]
By the KKT condition \eqref{eq:abstract-kkt} and the separability of constraints of $U^{(i)}$'s in \eqref{eq:sota-max-matrix}, we have that at a KKT point $U$, there exist matrices $P_i\in\mathbb R^{r\times r}$ for $i=1,\dots,s$ and vectors $\mathbf p_j\in\mathbb R^r$ for $j=s+1,\dots,k$ such that
\[
\nabla_{U^{(i)}} f(U)-U^{(i)}(P_i+P_i^\tp)=0\ \text{if }i=1,\dots,s,\ \text{and }\nabla_{U^{(j)}} f(U)-2U^{(j)}\operatorname{diag}_2({\mathbf p_j})=0\ \text{if }j=s+1,\dots,k.
\]
Therefore, we can assume without loss of generality that $P_i$ is symmetric for all $i=1,\dots,s$. By the fact that $U^{(j)}\in\B(r,n_j)$ and \eqref{eq:lambda-tensor}, we also have that for all $j=s+1,\dots,k$
\[
\mathbf p_j=\frac{1}{2}\Diag_2((U^{(j)})^\tp \nabla_{U^{(j)}} f(U))=\Diag_2((U^{(j)})^\tp V^{(j)}\Lambda)=\Diag_2(\Lambda^2).
\]
Consequently, $\mathbf p_{s+1}=\dots=\mathbf p_k$, which can be denoted as $\mathbf p$.
In summary, the KKT condition for problem \eqref{eq:sota-max-matrix} can be explicitly written as \eqref{eq:kkt}.
\end{proof}

It follows immediately from the system \eqref{eq:kkt} that for all $1\le i \le s$,
\begin{equation}\label{eq:kkt-symmetry}
(U^{(i)})^\tp V^{(i)}\Lambda=
\big(V^{(i)}\Lambda\big)^\tp U^{(i)}.
\end{equation}
In fact, the equations \eqref{eq:kkt-symmetry} and those in the first half of \eqref{eq:kkt} are equivalent \cite{AMS-08,EAT-98}. We denote by $\operatorname{M}( U)$ the set of all multipliers associated to a feasible point $U$ of \eqref{eq:sota-max-matrix}. The theory of {LICQ} mentioned in Section~\ref{subsec:KKT} can be applied to determine the size of $\operatorname{M}(U)$ when $U$ is a local maximizer of \eqref{eq:sota-max-matrix}.
\begin{proposition}[LICQ]\label{prop:licq}
At any feasible point of \eqref{eq:sota-max-matrix}, the {LICQ} is satisfied. In particular, at any local maximizer of \eqref{eq:sota-max-matrix}, the system of KKT condition holds and $\operatorname{M}( U)$ is a singleton.
\end{proposition}

\begin{proof}
{Once the LICQ is satisfied, the uniqueness of multipliers follows immediately from the discussion in Section~\ref{subsec:KKT}.}
Let {$P=  (P_1,\dots,P_s,\mathbf p) \in (\SS^r)^{\times s} \times \mathbb R^r$} be a  multiplier for the equality constraints in \eqref{eq:sota-max-matrix}. We assume that they satisfy the relation
\begin{equation}\label{prop:licq:eq1}
\nabla_{ U} \bigg(\sum_{i=1}^s\langle \big(U^{(i)}\big)^\tp U^{(i)}-I,P_i\rangle + \sum_{j=s+1}^k\langle\operatorname{Diag}_2\big(\big(U^{(j)}\big)^\tp U^{(j)}\big)-{\mathbf e},{\mathbf p}\rangle\bigg)=0.
\end{equation}
Here $\SS^r$ denotes the space of $r\times r$ symmetric matrices and $\mathbf e$ denotes the $r$-dimensional vector whose elements are all equal to one. To verify the LICQ, we prove that {$P_i=0$  for all $1\le i \le s$ and $\mathbf{p} = 0$}. Since $U$ is partitioned as $U=(U^{(1)},\dots, U^{(k)})$, we may separate \eqref{prop:licq:eq1} accordingly as
\begin{equation}\label{eq:kkt-2}
\langle U^{(i)}P_i,M^{(i)}\rangle=0\ \text{for all }M^{(i)}\in\mathbb R^{n_i\times r},\ \text{for all }i=1,\dots,s,
\end{equation}
and
\begin{equation}\label{eq:kkt-1}
\langle U^{(j)}\diag_2({\mathbf p}),M^{(j)}\rangle =0\ \text{for all }M^{(j)}\in\mathbb R^{n_j\times r},\ \text{for all }j=s+1,\dots,k.
\end{equation}
We see immediately from \eqref{eq:kkt-2} that $U^{(i)}P_i=0$ and hence $P_i=0$ by the orthogonality of column vectors of $U^{(i)}$ for $i=1,\dots,s$. Smilarly from \eqref{eq:kkt-1} that $\operatorname{diag}_2({\mathbf p})=0$  since each column of $U^{(j)}$ is a unit vector for all $j=s+1,\dots,k$.
\end{proof}
\subsection{Primitive KKT points and essential KKT points}\label{sec:degenerate-kkt}
It is possible that for some $1 \le j \le r$, $\lambda_j$ approaches to zero along iterations of an algorithm solving the problem \eqref{eq:sota}. In this case, the limiting partially orthogonal tensor has a parametrization with factor matrices of fewer columns. We discuss such a reduced case in this subsection. To do so, we need to compare KKT points of mLRPOTA($r$) and those of mLRPOTA($r-1$), where for each positive integer $t$, mLRPOTA($t$) denotes the problem \eqref{eq:sota-max-matrix} with $r = t$.
\begin{proposition}[KKT Reduction]\label{prop:degenerate}
Let
\[
 U=(U^{(1)},\dots,U^{(k)})\in \V(r,n_1)\times\dots\times \V(r,n_s)\times\B(r,n_{s+1})\times\dots\times\B(r,n_k)
\]
be a KKT point of the problem mLRPOTA($r$) and let $1 \le j \le r$ be a fixed integer. Set
\[
\hat{ U} \coloneqq (\hat U^{(1)},\dots,\hat U^{(k)})\in \V(r-1,n_1)\times\dots\times \V(r-1,n_s)\times\B(r-1,n_{s+1})\times\dots\times\B(r-1,n_k),
\]
where for each $1\le i \le k$, $\hat U^{(i)}$ is the matrix obtained by deleting the $j$-th column of $U^{(i)}$. If $\lambda_j(U) = 0$,
then $\hat{U}$ is a KKT point of the problem mLRPOTA($r-1$).
\end{proposition}

\begin{proof}
By Proposition~\ref{def:kkt}, the KKT system of problem \eqref{eq:sota-max-matrix} is
\begin{equation*}
\begin{cases}V^{(i)}\Lambda-U^{(i)}P_i=0,& \text{if }i=1,\dots,s,\\
V^{(i)}\Lambda-U^{(i)} \operatorname{diag}_2({\mathbf p})=0,& \text{if } i=s+1,\dots,k,
\end{cases}
\end{equation*}
where $ P = (P_1,\dots,P_s,{\mathbf p})$ 
is the Lagrange multiplier associated to $U$. Without loss of generality, we may assume that $j=r$, which implies that the last diagonal element of $\Lambda$ is zero. Thus,
\[
(U^{(i)})^\tp \begin{bmatrix}\mathbf v^{(i)}_1&\dots&\mathbf v^{(i)}_{r-1}&\mathbf v^{(i)}_r\end{bmatrix}\begin{bmatrix}\hat \Lambda& 0\\  0& 0\end{bmatrix}=P_i,\quad 1\le i \le s,
\]
and
\[
\begin{bmatrix}\mathbf v^{(i)}_1&\dots&\mathbf v^{(i)}_{r-1}&\mathbf v^{(i)}_r\end{bmatrix}\begin{bmatrix}\hat \Lambda& 0\\  0& 0\end{bmatrix}=U^{(i)}\diag_2({\mathbf p}),\quad s+1\leq i\leq k,
\]
where $\hat\Lambda$ is the leading $(r-1)\times (r-1)$ principal submatrix of $\Lambda$.
This implies that the last column of $P_i$ is zero. By the symmetry of $P_i$, we conclude that
\[
P_i=\begin{bmatrix}\hat P_i& 0\\  0& 0\end{bmatrix}, \quad 1\le i \le s,
\]
{where $\hat P_i$ is the leading $(r-1)\times (r-1)$ principal submatrix of $P_i$.}
Likewise, we have
\[
{\mathbf p}= \begin{bmatrix}
\hat{\mathbf p} \\
0
\end{bmatrix}.
\]
Therefore we have
\[
\begin{bmatrix}\mathbf v^{(i)}_1&\dots&\mathbf v^{(i)}_{r-1}&   \mathbf{v}^{(i)}_r\end{bmatrix}\begin{bmatrix}\hat \Lambda& 0\\  0&0\end{bmatrix}=\begin{bmatrix}\hat U^{(i)}&\mathbf u^{(i)}_r\end{bmatrix}\begin{bmatrix}\hat P_i& 0\\  0& 0\end{bmatrix},\quad 1\le i \le s,
\]
and
\[
\begin{bmatrix}\mathbf v^{(i)}_1&\dots&\mathbf v^{(i)}_{r-1}& \mathbf{v}^{(i)}_r \end{bmatrix}\begin{bmatrix}\hat \Lambda& 0\\  0& 0\end{bmatrix}=\begin{bmatrix}\hat U^{(i)}&\mathbf u^{(i)}_r\end{bmatrix}\diag_2(\begin{bmatrix}
\hat{\mathbf p} \\
0
\end{bmatrix} ),\quad s+1\leq i\leq k,
\]
which simplifies to
\begin{equation*}
\begin{cases}
\begin{bmatrix}
\mathbf v^{(i)}_1&\dots&\mathbf v^{(i)}_{r-1}
\end{bmatrix}
\hat\Lambda- \hat U^{(i)}\hat P_i =0,& \text{if }i=1,\dots,s,\\
\begin{bmatrix}\mathbf v^{(i)}_1&\dots&\mathbf v^{(i)}_{r-1}\end{bmatrix}\hat\Lambda- \hat U^{(i)}\diag_2(\hat{\mathbf p})= 0 ,& \text{if } i=s+1,\dots,k.
\end{cases}
\end{equation*}
%
Consequently, we may conclude that $\hat{ U}$ is a KKT point of mLRPOTA($r-1$) {by Proposition~\ref{def:kkt}}.
\end{proof}

A KKT point $ U=(U^{(1)},\dots,U^{(k)})$ of {mLRPOTA($r$)} 
with $\lambda_j(U) \ne 0$
for all $1\le j \le r$ is called a \textit{primitive KKT point}. Iteratively applying Proposition~\ref{prop:degenerate}, we obtain the following
\begin{corollary}\label{cor:essential-primitive}
Let $T$ be a subset of $\{1,\dots,r\}$ of cardinality $t=|T|<r$ and let $ U=(U^{(1)},\dots,U^{(k)})$ be a KKT point of {mLRPOTA($r$)}. Set
\[
\hat{ U} \coloneqq (\hat U^{(1)},\dots,\hat U^{(k)})\in \V(r-t,n_1)\times\dots\times \V(r-t,n_s)\times\B(r-t,n_{s+1})\times\dots\times\B(r-t,n_k),
\]
where for each $1 \le i \le k$, $\hat U^{(i)}$ is obtained by deleting columns {of $U^{(i)}$} indexed by $T$. If $\lambda_j(U) = 0$ exactly for $j\in T$, then $\hat{U}$ is a primitive KKT point of {mLRPOTA($r-t$)}.
\end{corollary}

It would happen that several KKT points of mLRPOTA($r$) reduce in this way to the same primitive KKT point of {mLRPOTA($r-t$)}. We call the set of such KKT points an \emph{essential KKT point}. Therefore, except for the essential KKT point with all $\lambda_j(U)=0$ (i.e., the trivial case), there is a one to one correspondence between essential KKT points of {mLRPOTA($r$)} and primitive KKT points of {mLRPOTA($t$)} for $1 \le t \le r$.
\subsection{Critical points are KKT points}\label{sec:critical-kkt}
In this subsection, we establish the relation between KKT points of problem {\eqref{eq:sota}} and critical points of {$g$ on the manifold $W_{\mathbf n,r,\ast}$, where $g$ is the function defined in \eqref{eq:objective-p}.} To achieve this, we recall from \eqref{eq:tangent-form} and Lemmas~\ref{lem:grand} and \ref{lem:grand-oblique} that components of the gradient of $g$ at a point $( U,\mathbf x)\in W_{\mathbf n,r,\ast}$ are respectively given by
\begin{align}
\operatorname{grad}_{U^{(i)}}g( U,\mathbf x)&=- (I_{n_i}-\frac{1}{2}U^{(i)}(U^{(i)})^\tp )\big(V^{(i)} \Gamma-U^{(i)}(V^{(i)}\Gamma)^\tp  U^{(i)}\big),\ i=1,\dots,s,\label{eq:gradient-ui}\\
\operatorname{grad}_{U^{(i)}}g( U,\mathbf x)&=-\Big( V^{(i)}\Gamma-U^{(i)} {\diag_2\big((U^{(i)})^\tp V^{(i)}\Gamma\big)} \Big),\ i=s+1,\dots,k,\label{eq:gradient-ui-ob} \\
\operatorname{grad}_{\mathbf x}g( U,\mathbf x)&= \nabla_{\mathbf x} g( U,\mathbf x) = \mathbf x-\operatorname{Diag}_k\big(((U^{(1)})^\tp ,\dots,(U^{(k)})^\tp )\cdot\mathcal A\big),\label{eq:gradient-x}
\end{align}
where $\Gamma=\operatorname{diag}_2(\mathbf x)$ is the diagonal matrix formed by the vector $\mathbf x$.

\begin{proposition}\label{prop:critical-equivalence}
A point $( U,\mathbf x)\in W_{\mathbf n,r,\ast}$ is a critical point of $g$ in $W_{\mathbf n,r,\ast}$ if and only if $( U,\operatorname{diag}_k(\mathbf x))$ is a KKT point of problem \eqref{eq:sota}.
\end{proposition}

\begin{proof}
We recall that a critical point $(U,\mathbf x)\in W_{\mathbf n,r,\ast}$ of $g$ is defined by $\operatorname{grad}(g)(U,\mathbf x)=0$. It follows from {\eqref{eq:gradient-ui}, \eqref{eq:gradient-ui-ob} and \eqref{eq:gradient-x},} Lemmas~\ref{lem:grand} and \ref{lem:grand-oblique}
that these critical points are characterized by
\[
\begin{cases}
\nabla_{U^{(i)}} g( U,\mathbf x)=U^{(i)}P_i, &\text{if }1\le i \le s, \\
\nabla_{U^{(i)}} g( U,\mathbf x)=U^{(i)} {\diag_2(\mathbf p)},&\text{if }s + 1 \le i \le k, \\
\nabla_{\mathbf x} g( U,\mathbf x) =0,
\end{cases}
\]
where $P_i$ is some $r\times r$ symmetric matrix and ${\mathbf p}\in\mathbb R^r$. By \eqref{eq:gradient-x}, we have
\[
\mathbf x = \operatorname{Diag}_k\big(((U^{(1)})^\tp ,\dots,(U^{(k)})^\tp )\cdot\mathcal A\big),
\]
and according to \eqref{eq:objective-p}, we derive
\[
\nabla_{U^{(i)}} g( U,\mathbf x)=-V^{(i)} {\Gamma}+U^{(i)}{\Gamma^2}.
\]
Therefore, {Proposition~\ref{def:kkt}} implies that a critical point of $g$ on $W_{\mathbf n,r,\ast}$ must come from a KKT point of problem \eqref{eq:sota}. The converse is obvious and this completes the proof.
\end{proof}

\begin{definition}[Nondegenerate KKT Point]\label{def:nondegenerate}
A KKT point $( U,\mathcal{D})$ of problem \eqref{eq:sota} is nondegenerate  if $( U,\mathbf x)\in W_{\mathbf n,r,\ast}$ is a nondegenerate critical point of $g$ with $\operatorname{diag}_k(\mathbf x)=\mathcal{D}$.
\end{definition}

Let $U$ be a primitive KKT point of problem \eqref{eq:sota-max} which uniquely determines the KKT point $(U,\mathcal \diag_k(\x))$ of problem \eqref{eq:sota}. If $( U,\mathbf x) \in W_{\mathbf n,r}$ then $U$ is called a \emph{smooth KKT point} and the corresponding essential KKT point is called a \emph{smooth essential KKT point}. {Note that by definition, a smooth primitive KKT point actually lies in {$W_{\mathbf n,r,*}$.}
\begin{theorem}[Finite Smooth Essential Critical Points]\label{thm:kkt}
For a generic tensor, there are only finitely many smooth essential KKT points for problem~\eqref{eq:sota}, and for any positive integers $r \ge s > 0$,
a smooth primitive KKT point of  {mLRPOTA($s$)} corresponding to a smooth essential KKT point of {mLRPOTA($r$)} is nondegenerate.
\end{theorem}

\begin{proof}
To prove the first statement, it is sufficient to show that there are only finitely many smooth primitive KKT points in $W_{\mathbf n,r,\ast}$, and the finiteness of smooth essential KKT points follows from Proposition~\ref{prop:degenerate} and the layer structure of the set $W_{\mathbf n,r}$ (induced from that of $Q_s(\mathbf n,r)$ given by Proposition~\ref{prop:smooth}--\eqref{prop:smooth:item6}). We recall that KKT points on $W_{\mathbf n,r,\ast}$ are defined by \eqref{eq:kkt}, which is a system of polynomial equations, which implies that the set $K_{\mathbf{n},r}$ of KKT points of problem \eqref{eq:sota-max} on $W_{\mathbf n,r,\ast}$ is a closed subvariety of the quasi-variety $W_{\mathbf n,r,\ast}$. We also note that there are finitely many irreducible components of $K_{\mathbf n,r}$ \cite{H-77} and hence it suffices to prove that each irreducible component of $K_{\mathbf{n},r}$ is a singleton. Now let $Z\subseteq K_{\mathbf{n},r}$ be an irreducible component of $K_{\mathbf{n},r}$. If $Z$ contains infinitely many points, then $\dim (Z) \ge 1$ according to Corollary~\ref{cor:0-dim}. However, according to Proposition~\ref{prop:critical-equivalence}, each point in $Z$ determines a critical point of the function $g$ on the manifold $W_{\mathbf n,r,\ast}$. This implies that the set of critical points of $g$ on $W_{\mathbf n,r,\ast}$ has a positive dimension and hence by Proposition~\ref{prop:smooth locus}, its smooth locus is a manifold of positive dimension, which contradicts Lemma~\ref{lem:isolated critical points} and Proposition~\ref{prop:nondegnerate}.

Next, by Corollary~\ref{cor:essential-primitive}, given a non-primitive smooth KKT point $ U$ of {mLRPOTA($r$)}, we obtain a primitive smooth KKT point $\hat{ U}$ of {mLRPOTA($s$)} with some $s < r$.
Obviously we have $(\hat{ U},\mathbf x)\in W_{\mathbf n,s,\ast}$, where $\mathbf x$ is determined by $\lambda_j(\hat{ U})$'s. Since for a generic tensor the function $g$ has only nondegenerate critical points on $W_{\mathbf n,s,\ast}$ by Proposition~\ref{prop:nondegnerate}, the second assertion follows from Proposition~\ref{prop:critical-equivalence} and Corollary~\ref{cor:essential-primitive}.
\end{proof}
\section{Convergence Analysis of the iAPD-ALS Algorithm}\label{sec:convergence}
With all the preparations in Sections~\ref{sec:podt} and \ref{sec:kkt}, we are ready to carry out a convergence analysis for Algorithm~\ref{algo}. To that end, we denote the collection of the $k$ factor matrices in the $p$-th iteration of Algorithm~\ref{algo} by
\begin{equation}\label{eq:whole-u}
U_{[p]} \coloneqq (U_{[p]}^{(1)},\dots, U_{[p]}^{(k)}),
\end{equation}
and we recall that Algorithm~\ref{algo} consists of two types of updates: ALS updates from $U^{(s)}_{[p-1]}$ to $U^{(k)}_{[p-1]}$ and iAPD updates from $U^{(k)}_{[p-1]}$ to $U^{(s)}_{[p]}$, which are exhibited in Figure~\ref{fig:flow} for illustration. Hence the analysis of the global dynamics of Algorithm~\ref{algo} consists of two parts accordingly, which is done in Subsection~\ref{sec:properties}. The goal of Subsection~\ref{sec:global} is to prove the global convergence of Algorithm~\ref{algo} and the rest two subsections are respectively devoted to prove the sublinear convergence rate in general and the generic linear convergence rate.

\begin{figure}[htbp]
\centering
\begin{tikzpicture}[node distance=1.8cm,font=\sf]

\node (start) [startstop] {$(p-1)$-th iteration};
\node (pro1) [process, right of=start, xshift=1.8cm] {$U^{(1)}_{[p-1]}$};
\node (dot1) [dotsnd,right of=pro1,xshift=0.2cm]{$\dots$};
\node (dash1) [processd, right of=dot1,xshift=3.6cm]{};
\node (pro2) [process, right of=dot1, xshift= 0.4cm] {$U^{(s)}_{[p-1]}$};
\node (pro3) [process, right of=pro2, xshift= 0.4cm] {$U^{(s+1)}_{[p-1]}$};
\node (dot2) [dotsnd,right of=pro3,xshift=0.2cm]{$\dots$};
\node (pro4) [process, right of=dot2, xshift= 0.4cm] {$U^{(k)}_{[p-1]}$};

\node (startp) [startstop, below of=start, yshift=-0.6cm,xshift=-0.45cm] {$p$-th iteration};
\node (prop1) [process, below of=pro1, yshift=-0.6cm] {$U^{(1)}_{[p]}$};
\node (ppro4) [processb, left of=prop1, xshift= 0.4cm] {$U^{(k)}_{[p-1]}$};
\node (dotp1) [dotsnd, right of=prop1,xshift=0.2cm]{$\dots$};
\node (dotp4) [dotsnd,above of=dotp1,yshift=-3.1cm]{\text{iAPD updates}};
\node (prop2) [process, right of=dotp1, xshift= 0.4cm] {$U^{(s)}_{[p]}$};
\node (dash2) [processdd, left of=prop2,xshift=-1.1cm]{};
\node (prop3) [process, right of=prop2, xshift= 0.4cm] {$U^{(s+1)}_{[p]}$};
\node (dotp2) [dotsnd,right of=prop3,xshift=0.2cm]{$\dots$};
\node (dotp3) [dotsnd,above of=dotp2,yshift=1.8cm]{\text{ALS updates}};
\node (prop4) [process, right of=dotp2, xshift= 0.4cm] {$U^{(k)}_{[p]}$};

\draw [arrow](pro1) -- (dot1);
\draw [arrow](dot1) -- (pro2);
\draw [arrow](pro2) -- (pro3);
\draw [arrow](pro3) -- (dot2);
\draw [arrow](dot2) -- (pro4);


\draw[red,->] let \p1=(pro4), \p2=(prop1) in (pro4) -- (\x1, -0.33*\x2) -| (prop1);


\draw[arrow] (ppro4)--(prop1);
\draw [arrow](prop1) -- (dotp1);
\draw [arrow](dotp1) -- (prop2);
\draw [arrow](prop2) -- (prop3);
\draw [arrow](prop3) -- (dotp2);
\draw [arrow](dotp2) -- (prop4);
\end{tikzpicture}
\caption{\label{fig:flow}Flow chart of Algorithm~\ref{algo}}
\end{figure}
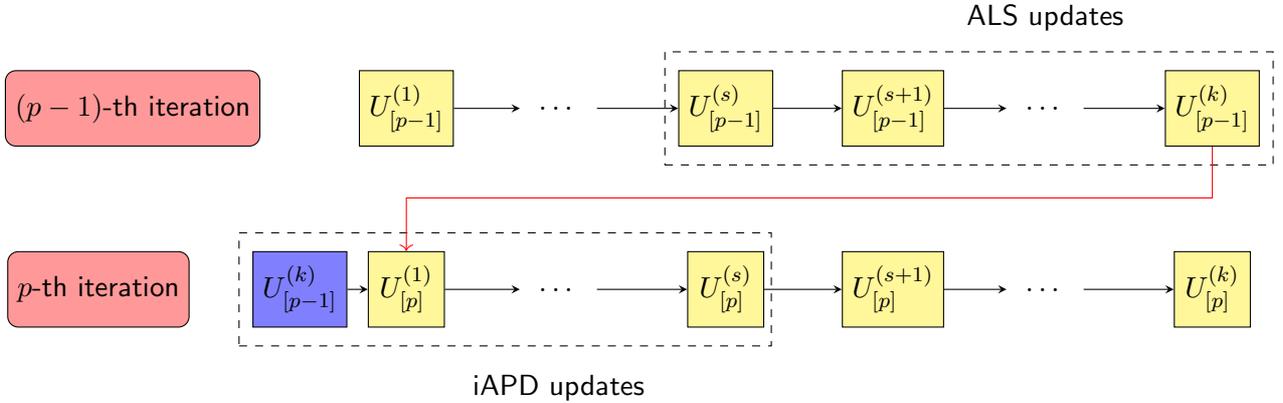

For the reader's convenience, we recall some notations from Subsection~\ref{subsec:notation} and Algorithm~\ref{algo}. Let $f$ be the function defined in \eqref{eq:objective}, which is the objective function of mLRPOTA($r$). Given $\mathcal{A} \in \mathbb{R}^{n_1} \otimes \cdots \otimes \mathbb{R}^{n_k}$ and $1\le i \le k$, maps $\mathcal{A}\tau$ and $\mathcal{A}\tau_i$ on $\mathbb{R}^{n_1} \times \cdots \times \mathbb{R}^{n_k}$ are respectively defined by
\begin{align*}
\mathcal{A}\tau (\mathbf{u}_1,\dots, \mathbf{u}_k) &= \langle\mathcal A,\mathbf u_1\otimes\dots\otimes\mathbf u_k\rangle \in \mathbb{R}, \\
\mathcal{A}\tau_i (\mathbf{u}_1,\dots, \mathbf{u}_k) &= \langle\mathcal A,\mathbf u_1\otimes\dots\otimes\mathbf u_{i-1}\otimes\mathbf u_{i+1}\otimes\dots\otimes\mathbf u_k\rangle \in \mathbb{R}^{n_i}.
\end{align*}
We write $\mathbf{u}^{(l)}_{t,[p]}$ for the $t$-th column vector of $U^{(l)}_{[p]}$, from which we construct a block vector
\begin{equation}\label{eq:block-xp}
\mathbf{x}^i_{j,[p]} := (\mathbf u^{(1)}_{j,[p]},\dots,\mathbf u^{(i-1)}_{j,[p]}, {\mathbf u}^{(i)}_{j,[p-1]}, {\mathbf u}^{(i+1)}_{j,[p-1]},\dots,{\mathbf u}^{(k)}_{j,[p-1]}).
\end{equation}
We remind the reader that the superscript $i$ of $\mathbf x^{i}_{j,[p]}$ is not in parentheses so that we can easily distinguish the vector $\mathbf u^{(i)}_{j,[p]}$ and the block vector $\mathbf{x}^i_{j,[p]} $. The rule for the superscript to be in parentheses or not is whether this superscript indicates a component. Following this rule, we also denote
\begin{equation}\label{eq:lambda+vp}
\lambda^{i-1}_{j,[p]} := \mathcal{A} \tau (\mathbf{x}^i_{j,[p]} ),\quad \mathbf{v}^{(i)}_{j,[p]} := \mathcal{A} \tau_i (\mathbf{x}^i_{j,[p]} )
\end{equation}
and
\begin{equation}\label{eq:Lambda + Vp}
\Lambda^{(i)}_{[p]} := \operatorname{diag}_2 (\lambda^{i-1}_{1,[p]},\dots, \lambda^{i-1}_{r,[p]} ),\quad V^{(i)}_{[p]} := \begin{bmatrix}
\mathbf{v}^{(i)}_{1,[p]} & \cdots \mathbf{v}^{(i)}_{r,[p]}
\end{bmatrix}.
\end{equation}
Since $\mathbf{x}^1_{j,[p]} = {\mathbf{x}}^{k+1}_{j,[p-1]}$ by \eqref{eq:block-xp}, we obtain from \eqref{eq:lambda+vp} the relation
\begin{equation}\label{eq:lambda-notationp}
\lambda^0_{j,[p]} ={\lambda}^k_{j,[p-1]}.
\end{equation}
Moreover, {for each $s+1 \le i \le k$,} the vector $\mathbf{u}^{(i)}_{j,[p]}$ is updated by the formula:
\begin{equation}\label{eq:updatep}
\mathbf u^{(i)}_{j,[p]} =
\sgn(\lambda^{i-1}_{j,[p]})\frac{\mathcal A\tau_i(\mathbf x^{i}_{j,[p]})}{\|\mathcal A\tau_i(\mathbf x^{i}_{j,[p]})\|}.
\end{equation}
We also define for each tuple $({W}, \mathcal{T}) = (W^{(1)},\dots, W^{(k)},\mathcal{T}) \in \mathbb{R}^{n_1\times r} \times \cdots \times \mathbb{R}^{n_k \times r} \times {(\mathbb{R}^r)^{\otimes k}}$ the norms
\[
\| W \|_F^2 \coloneqq \sum_{i=1}^k\| W^{(i)}\|_F^2,\quad \|( W,\mathcal{T}) \|_F^2 \coloneqq \sum_{i=1}^k\| W^{(i)}\|_F^2 + \lVert \mathcal{T} \rVert^2.
\]

In Algorithm~\ref{algo}, if one column of the iteration matrix $U^{(i)}_{[p]}$ is removed in the truncation step, we say that one truncation occurs. It is possible that several truncations may appear in one iteration. We denote by $J_p$ the set of indices of columns which are removed by these truncations in the $p$-th iteration and thus the number of truncations occurred in the $p$-th iteration of Algorithm~\ref{algo} is the cardinality of $J_p$. We have the following lemma regarding the decrease of the objective function value caused by these truncations.
\begin{lemma}\label{lem:truncation-obj}
There are at most $r$ truncations in Algorithm~\ref{algo} and the decrease of the objective function value caused by each truncation is at most $\kappa^2$. Thus, the total decrease of the objective function value caused by all truncations is at most $r\kappa^2$.
\end{lemma}

\begin{proof} 
Note that before the truncation step in the $p$-th iteration, we have 
\[
f(U^{(1)}_{[p]},\dots,U^{(s)}_{[p]},U^{(s+1)}_{[p-1]},\dots,U^{(k)}_{[p-1]})=\sum_{j} (\lambda^{(s+1)}_{j,[p]})^2,
\]
where $j$ runs through $1$ to the number of columns of $U^{(i)}_{[p]}$'s. The $j$-th column of these matrices are truncated if and only if $j \in J_p$, i.e., $|\lambda^{(s+1)}_{j,[p]}|<\kappa$. Therefore, the new objective function value after the truncation step is
\[
f(\tilde U^{(1)}_{[p]},\dots,\tilde U^{(s)}_{[p]},\tilde U^{(s+1)}_{[p-1]},\dots,\tilde U^{(k)}_{[p-1]})=\sum_{j\not\in J_p}(\lambda^{(s+1)}_{j,[p]})^2>f(U^{(1)}_{[p]},\dots,U^{(s)}_{[p]},U^{(s+1)}_{[p-1]},\dots,U^{(k)}_{[p-1]})- \lvert J_p \rvert \kappa^2.
\]
Since iteration matrices only have $r$ columns at the initialization, truncations may occur at most $r$ times in Algorithm~\ref{algo}.
\end{proof}

\subsection{Local properties}\label{sec:properties}
As is discussed at the beginning of this section, the analysis of ALS and iAPD updates is an essential ingredient of the global convergence analysis of Algorithm~\ref{algo}. In this subsection, we estimate how the value of the objective function varies with respect to each update in Algorithm~\ref{algo}. To achieve this, it is sufficient to compare the function value between updates in $(p-1)$-th and $p$-th iterations of Algorithm~\ref{algo} for each fixed positive integer $p$. Hence we drop the subscript ``$_{[p]}$" to simplify the notation.

We denote the data of the current iteration  as
\[
\overline{U}:=(\overline{U}^{(1)},\dots,\overline{U}^{(k)})
\]
and the data of the next iteration  as 
\[
U:=(U^{(1)},\dots,U^{(k)}).
\]
Similarly, the subscript ``$_{[p]}$" in \eqref{eq:block-xp}--\eqref{eq:updatep} can also be dropped accordingly. To be more precise, we denote by $\mathbf{u}^{(l)}_{t}$ {(resp. $\overline{\mathbf{u}}^{(l)}_{t}$)} the $t$-th column vector of $U^{(l)}$ {(resp. $\overline{U}^{(l)}$).} Hence \eqref{eq:block-xp}--\eqref{eq:updatep} become
\begin{align}
\mathbf{x}^i_{j} &= (\mathbf u^{(1)}_{j},\dots,\mathbf u^{(i-1)}_{j}, \overline{\mathbf u}^{(i)}_{j}, \overline{\mathbf u}^{(i+1)}_{j},\dots,\overline{\mathbf u}^{(k)}_{j}), \label{eq:block-x} \\
\lambda^{i-1}_{j} &= \mathcal{A} \tau (\mathbf{x}^i_{j} ),\quad \mathbf{v}^{(i)}_{j} = \mathcal{A} \tau_i (\mathbf{x}^i_{j} ), \label{eq:lambda+v}\\
\Lambda^{(i)} &= \operatorname{diag}_2 (\lambda^{i-1}_{1},\dots, \lambda^{i-1}_{r} ),\quad V^{(i)} = \begin{bmatrix}
\mathbf{v}^{(i)}_{1} & \cdots \mathbf{v}^{(i)}_{r}
\end{bmatrix}, \label{eq:Lambda + V} \\
\lambda^0_{j} &=\overline{\lambda}^k_{j}, \label{eq:lambda-notation} \\
\mathbf u^{(i)}_{j} &=
\sgn(\lambda^{i-1}_{j})\frac{\mathcal A\tau_i(\mathbf x^{i}_{j})}{\|\mathcal A\tau_i(\mathbf x^{i}_{j})\|}. \label{eq:update}
\end{align}
We also define
\begin{equation}\label{eq:variable-u}
 U_{i} \coloneqq (U^{(1)},\dots,U^{(i)},\overline{U}^{(i+1)},\dots,\overline{U}^{(k)}),\quad  i = 0,\dots, k.
\end{equation}
With this notation, we in particular have
\[
U_{0} = \overline{U} =  \overline{U}_{k}.
\]
We first deal with the ALS updates.
\begin{lemma}[Monotonicity of ALS updates]\label{lem:lambda-k}
Let $r$ be the number of columns in the matrix $U^{(k)}$. For each $s \le i \le k-1$ and $1\le j \le r$, we have
\begin{equation}\label{eq:lambda-k}
\lambda^{i+1}_{j}=\lambda^i_{j}+\sgn(\lambda^i_{j})\frac{{|\lambda^{i+1}_{j}|}}{2}\|\mathbf u^{(i+1)}_{j}-\overline{\mathbf u}^{(i+1)}_{j}\|^2
\end{equation}
and
\begin{align}\label{eq:obj-k}
f( U_{i+1})-f( U_{i})&=\frac{1}{2}\sum_{j=1}^{r}|\lambda^{i+1}_{j}(\lambda^{i+1}_{j}+\lambda^{i}_{j})|\|\mathbf u^{(i+1)}_{j}-\overline{\mathbf u}^{(i+1)}_{j}\|^2\nonumber\\
&\geq \kappa^2\|U^{(i+1)}-\overline{U}^{(i+1)}\|_F^2,
\end{align}
where $\kappa > 0$ is the truncation parameter chosen in Algorithm~\ref{algo}.
\end{lemma}

\begin{proof}
We first recall from \eqref{eq:lambda+v} and \eqref{eq:segre and partial segre} that
\[
\lambda^i_{j}=\mathcal A\tau(\mathbf x^{i+1}_{j}) = \langle \mathcal A\tau_{i}(\mathbf x^{i}_{j}),\mathbf u^{(i)}_{j}\rangle = \langle \mathcal{A} \tau_{i+1}(\mathbf x^{i+1}_{j}),\overline{\mathbf u}^{(i+1)}_{j}\rangle,\quad s \le i \le k-1,\ 1\le j \le r.
\]
This implies
\begin{align*}
\lambda^{i+1}_{j}-\lambda^i_{j}&=\langle \mathcal A\tau_{i+1}(\mathbf x^{i+1}_{j}),\mathbf u^{(i+1)}_{j}\rangle- \langle \mathcal A\tau_{i}(\mathbf x^{i}_{j}),\mathbf u^{(i)}_{j}\rangle\\
&=\langle \mathcal A\tau_{i+1}(\mathbf x^{i+1}_{j}),\mathbf u^{(i+1)}_{j}\rangle- \langle \mathcal{A} \tau_{i+1}(\mathbf x^{i+1}_{j}),\overline{\mathbf u}^{(i+1)}_{j}\rangle\\
&=\sgn(\lambda^i_{j})|\lambda^{i+1}_{j}|\langle \mathbf u^{(i+1)}_{j},\mathbf u^{(i+1)}_{j}-{\overline{\mathbf u}^{(i+1)}_{j}}\rangle\\
&=\sgn(\lambda^i_{j})\frac{|\lambda^{i+1}_{j}|}{2}\|\mathbf u^{(i+1)}_{j}-\overline{\mathbf u}^{(i+1)}_{j}\|^2,
\end{align*}
which completes the proof of \eqref{eq:lambda-k}. Here the third equality follows from the relation
\[
|\lambda^{i+1}_{j}|=|\langle \mathcal A\tau_{i+1}(\mathbf x^{i+1}_{j}),\mathbf u^{(i+1)}_{j}\rangle|=\|\mathcal A\tau_{i+1}(\mathbf x^{i+1}_{j})\|
\]
derived from \eqref{eq:lambda+v} and \eqref{eq:update} and the last equality is obtained by observing that both $\mathbf u^{(i+1)}_{j}$ and $\overline{\mathbf u}^{(i+1)}_{j}$ are unit vectors. 

According to the truncation step of Algorithm~\ref{algo}, it is clear that 
\[
\lambda^s_{j}\neq 0,\quad  j=1,\dots,r.
\]
Together with \eqref{eq:lambda-k}, this implies that
\begin{equation}\label{eq:sign}
\sgn(\lambda^{i+1}_{j})=\sgn(\lambda^i_{j}),\quad s\le i \le k-1,\quad 1\le j \le r,
\end{equation}
and therefore either
\begin{equation}\label{eq:monotone-k1}
\kappa\leq \lambda^{s+1}_{j}\leq \lambda^{s+2}_{j}\leq\dots\leq \lambda^{k}_{j},
\end{equation}
or 
\begin{equation}\label{eq:monotone-k2}
-\kappa\geq \lambda^{s+1}_{j}\geq \lambda^{s+2}_{j}\geq\dots\geq \lambda^{k}_{j}.
\end{equation}

Since the sequence $\{\lambda^i_j\}_{i\geq s}$ is monotone inside the iteration for all $1\le j \le r$, we have for all $s \le i \le k - 1 $ that
\begin{align*}
f( U_{i+1})-f( U_{i})&=\sum_{j=1}^r(\lambda^{i+1}_{j})^2-\sum_{j=1}^r(\lambda^{i}_{j})^2\\
&{=\frac{1}{2}\sum_{j=1}^r(\lambda^{i+1}_{j}+\lambda^{i}_{j}) \sgn(\lambda^i_{j}) |\lambda^{i+1}_{j}|  \|\mathbf u^{(i+1)}_{j}-\overline{\mathbf u}^{(i+1)}_{j}\|^2}\\
&{=\frac{1}{2}\sum_{j=1}^r (\lambda^{i+1}_{j}+\lambda^{i}_{j})  \sgn(\lambda^{i+1}_{j}+\lambda^i_{j})|\lambda^{i+1}_{j}| \|\mathbf u^{(i+1)}_{j}-\overline{\mathbf u}^{(i+1)}_{j}\|^2}\\
&=\frac{1}{2}\sum_{j=1}^r|\lambda^{i+1}_{j}(\lambda^{i+1}_{j}+\lambda^{i}_{j})|\|\mathbf u^{(i+1)}_{j}-\overline{\mathbf u}^{(i+1)}_{j}\|^2\\
&\geq \kappa^2\|U^{(i+1)}-\overline{U}^{(i+1)}\|_F^2,
\end{align*}
{where the penultimate equality follows from \eqref{eq:sign} and the last inequality is obtained by \eqref{eq:monotone-k1} and \eqref{eq:monotone-k2}.}
\end{proof}

Next we analyse the iAPD updates.
\begin{lemma}[Monotonicity of {iAPD} updates]\label{lem:lambda-s}
If the current iteration in Algorithm~\ref{algo} is not a truncation iteration, then we have
\begin{equation}\label{eq:obj-s}
f( U_{i+1})-f( U_{i})\geq \frac{\epsilon}{2}\|U^{(i+1)}-\overline{U}^{(i+1)}\|_F^2,\quad 0 \le i \le s-1,
\end{equation}
where $\epsilon > 0$ is the proximal parameter chosen in Algorithm~\ref{algo}. Moreover, the conclusion holds true for $U^{(i)}$'s before the truncation step of Algorithm~\ref{algo} even if the current iteration is a truncation iteration.
\end{lemma}
\begin{proof}
We notice that the only difference between a truncation iteration and an iteration without truncation is the occurrence of the truncation step. This difference is vacuous before the truncation step. Thus it suffices to prove the first half of the statement.

Let $r$ be the number of columns of $U^{(i)}$'s.
For $i=0,\dots,s-1$, we have
\begin{align}
f( U_{i+1})-f( U_{i})&=\sum_{j=1}^r(\lambda^{i+1}_{j})^2-\sum_{j=1}^r(\lambda^{i}_{j})^2\nonumber\\
&=\sum_{j=1}^r(\lambda^{i+1}_{j}+\lambda^{i}_{j})(\lambda^{i+1}_{j}-\lambda^{i}_{j})\nonumber\\
&=\sum_{j=1}^r\lambda^{i}_{j}(\lambda^{i+1}_{j}-\lambda^{i}_{j}) + \sum_{j=1}^r\lambda^{i+1}_{j}(\lambda^{i+1}_{j}-\lambda^{i}_{j}).\label{eq:obj-k-exp}
\end{align}
By \eqref{eq:block-x}--\eqref{eq:lambda-notation}, we observe that
\begin{align*}
((U^{(i+1)})^\tp V^{(i+1)}\Lambda^{(i+1)})_{jj}& =(\mathbf u^{(i+1)}_{j})^\tp\mathbf v^{(i+1)}_{j}\lambda^{i}_{j}\\
&=((\mathbf u^{(i+1)}_{j})^\tp\mathcal A\tau_{i+1}\mathbf x^{i+1}_{j})\lambda^{i}_{j}\\
&=\mathcal A\tau(\mathbf x^{i+2}_{j})\lambda^{i}_{j}\\
&=\lambda^{i+1}_{j}\lambda^{i}_{j},
\end{align*}
and similarly $((\overline{U}^{(i+1)})^\tp V^{(i+1)}\Lambda^{(i+1)})_{jj} = \lambda^{i}_{j}\lambda^{i}_{j}$.

{We analyze the first summand in the last line of \eqref{eq:obj-k-exp} by considering the following two cases:}
\begin{enumerate}[label=(\roman*)]
\item\label{item:proof of 4.2} If $\sigma_{\min}(S^{(i+1)}) \geq \epsilon$, then there is no proximal step in Algorithm~\ref{algo}. Here $S^{(i+1)}$ is the positive semidefinite matrix in the polar decomposition $V^{(i+1)}\Lambda^{(i+1)} = U^{(i+1)} S^{(i+1)}$.
By Theorem~\ref{thm:error}, we obtain
\begin{align}
\sum_{j=1}^r\lambda^{i}_{j}(\lambda^{i+1}_{j}-\lambda^{i}_{j})
&=
\langle U^{(i+1)} ,  V^{(i+1)}\Lambda^{(i+1)} \rangle
- \langle \overline{U}^{(i+1)} , V^{(i+1)}\Lambda^{(i+1)} \rangle \nonumber \\
&=\frac{1}{2}\big\|(U^{(i+1)}-\overline{U}^{(i+1)})\sqrt{S^{(i+1)}}\big\|_F^2\nonumber\\
&\geq \frac{\epsilon}{2}\|U^{(i+1)}-\overline{U}^{(i+1)}\|_F^2.\label{eq:bound-ineq}
\end{align}

\item If $\sigma_{\min}(S^{(i+1)}) <\epsilon$, we consider the following matrix optimization problem
\begin{equation}\label{eq:proximal}
\begin{array}{rl}
\max&\langle V^{(i+1)}\Lambda^{(i+1)},H\rangle-\frac{\epsilon}{2}\|H-\overline{U}^{(i+1)}\|_F^2\\
\text{s.t.}& H\in V(r,n_{i+1}).
\end{array}
\end{equation}
{Since $H,\overline{U}^{(i+1)} \in {V(r,n_{i+1})}$, we must have}
\[
\frac{\epsilon}{2}\|H- \overline{U}^{(i+1)}\|_F^2=\epsilon r-\epsilon\langle \overline{U}^{(i+1)},H\rangle.
\]
Thus, by Lemma~\ref{lem:polar}, a global maximizer of \eqref{eq:proximal} is given by an orthonormal factor in the polar decomposition of $V^{(i+1)}\Lambda^{(i+1)}+\epsilon \overline{U}^{(i+1)}$. According to the proximal step in Algorithm~\ref{algo}, $U^{(i+1)}$ is such a polar orthonormal factor hence it is a global maximizer of \eqref{eq:proximal}. Therefore we have
\[
\langle V^{(i+1)}\Lambda^{(i+1)},U^{(i+1)}\rangle-\frac{\epsilon}{2}\|U^{(i+1)}
-\overline{U}^{(i+1)}\|_F^2\geq \langle V^{(i+1)}\Lambda^{(i+1)},\overline{U}^{(i+1)}\rangle,
\]
and the inequality \eqref{eq:bound-ineq} also holds in this case.
\end{enumerate}

In the following, we verify the nonnegativity of the second summand in the last line of \eqref{eq:obj-k-exp}, by which the proof is complete. According to \eqref{eq:bound-ineq}, we have 
\begin{equation}\label{eq:byproduct}
0\leq \sum_{j=1}^r\lambda^{i}_{j}(\lambda^{i+1}_{j}-\lambda^{i}_{j})
=\sum_{j=1}^r\lambda^{i+1}_{j}\lambda^{i}_{j}-\sum_{j=1}^r(\lambda^{i}_{j})^2,
\end{equation}
which together with the Cauchy-Schwartz inequality implies that
\begin{equation}\label{eq:caychy}
\big(\sum_{j=1}^r(\lambda^{i}_{j})^2\big)^2\leq \big(\sum_{j=1}^r\lambda^{i}_{j}\lambda^{i+1}_{j} \big)^2\leq\sum_{j=1}^r(\lambda^{i}_{j})^2\sum_{j=1}^r(\lambda^{i+1}_{j})^2.
\end{equation}

Now suppose that the iteration $\overline{U}=U_{[0]}$ is the initialization of Algorithm~\ref{algo}.
Since $f(U_{[0]})>0$, we conclude that
$f(U_{0})=\sum_{j=1}^r(\lambda^{0}_{j})^2>0$ and hence $\sum_{j=1}^r\lambda^{1}_{j}\lambda^{0}_{j}>0$ by \eqref{eq:byproduct}. Thus, we conclude that
\begin{equation}\label{eqn:proof of 4.2-1}
\sum_{j=1}^r\lambda^{1}_{j}\lambda^{0}_{j} \leq \sum_{j=1}^r(\lambda^{1}_{j})^2\frac{\sum_{j=1}^r(\lambda^{0}_{j})^2}
{\sum_{j=1}^r\lambda^{1}_{j}\lambda^{0}_{j}}\leq \sum_{j=1}^r(\lambda^{1}_{j})^2,
\end{equation}
where the first inequality follows from \eqref{eq:caychy} and the second from \eqref{eq:byproduct}. {Combining \eqref{eqn:proof of 4.2-1} with \eqref{eq:obj-k-exp} and \eqref{eq:bound-ineq}, we may obtain \eqref{eq:obj-s} for $i=0$ at the first iteration $p=1$.}

{On the one hand, from \eqref{eq:caychy} we obtain
\[
0 \le \sum_{j=1}^r (\lambda_{j}^i)^2 (\sum_{j=1}^r (\lambda_{j}^{i+1})^2 - \sum_{j=1}^r (\lambda_{j}^i)^2).
\]
Since $\sum_{j=1}^r (\lambda_{j}^i)^2 > 0$ if there is no truncation, } {we must have
\[
0 \le \sum_{j=1}^r (\lambda_{j}^{i+1})^2 - \sum_{j=1}^r (\lambda_{j}^i)^2 = f(U_{i+1})-f(U_{i}),
\]
i.e., the objective function $f$ is monotonically increasing} during the iAPD-ALS iteration as long as there is no truncation. On the other hand, {by Lemma~\ref{lem:truncation-obj},} there are at most $r$ truncations and the total loss of $f$ caused by truncations is at most $r\kappa^2<f(U_{[0]})$.
Therefore, $f$ is always positive and  according to \eqref{eq:byproduct}, we may conclude that $\sum_{j=1}^r\lambda^{i+1}_{j}\lambda^{i}_{j}>0$ along iterations. By induction on $p$, we obtain
\[
\sum_{j=1}^r\lambda^{i+1}_{j}\lambda^{i}_{j} \leq \sum_{j=1}^r(\lambda^{i+1}_{j})^2\frac{\sum_{j=1}^r(\lambda^{i}_{j})^2}
{\sum_{j=1}^r\lambda^{i+1}_{j}\lambda^{i}_{j}}\leq \sum_{j=1}^r(\lambda^{i+1}_{j})^2,
\]
which together with \eqref{eq:obj-k-exp} and \eqref{eq:bound-ineq}, implies \eqref{eq:obj-s} holds for any iteration.
\end{proof}

Combining Lemmas~\ref{lem:lambda-k} and \ref{lem:lambda-s}, we are able to derive the following estimate of the increment of the objective function during each update in Algorithm~\ref{algo}.
\begin{proposition}[Sufficient Increase]\label{prop:subfficient}
If the current iteration in Algorithm~\ref{algo} is not a truncation iteration, then we have
\begin{equation}\label{eq:obj-suf}
f( U)-f( \overline{U})\geq \frac{\min\{\epsilon,2\kappa^2\}}{2}\| U- \overline{U}\|_F^2.
\end{equation}
\end{proposition}

\begin{proof}
We have
\begin{align*}
f( U)-f( \overline{U})&=\sum_{i=0}^{k-1}\big(f( U_{i+1})-f( U_{i})\big)\\
&=\sum_{i=0}^{s-1}\big(f( U_{i+1})-f( U_{i})\big)+\sum_{i=s}^{k-1}\big(f( U_{i+1})-f( U_{i})\big)\\
&\geq \frac{\epsilon}{2}\sum_{i=0}^{s-1}\|U^{(i+1)}-\overline{U}^{(i+1)}\|_F^2
+\kappa^2\sum_{i=s}^{k-1}\|U^{(i+1)}-\overline{U}^{(i+1)}\|_F^2\\
&\geq\frac{\min\{\epsilon,2\kappa^2\}}{2}\| U- \overline{U}\|_F^2,
\end{align*}
where the penultimate inequality follows from Lemmas~\ref{lem:lambda-s} and \ref{lem:lambda-k}.
\end{proof}

\subsection{Global convergence}\label{sec:global}
In contrast to the preceding subsection which only concentrates on local properties of Algorithm~\ref{algo}, we study the global convergence of Algorithm~\ref{algo} in this subsection. This requires us to work with all iterations at the same time and therefore we adopt notations with subscript ``$_{[p]}$" introduced in \eqref{eq:whole-u}--\eqref{eq:updatep} to distinguish different iterations.

At each truncation iteration, the number of columns of the matrices in $ U_{[p]}$ is decreased strictly. The first issue we have to address is that the iteration $ U_{[p]}$ is not vacuous, i.e., the numbers of the columns of the matrices in $ U_{[p]}$ are positive and stable. We have the following proposition, which is recorded for later reference.
\begin{proposition}\label{prop:positive}
The number of columns of $U^{(i)}_{[p]}$'s will be stable at a positive integer $s\leq r$
and there exists some $N_0$ such that $f( U_{[p]})$ is nondecreasing for all $p\geq N_0$.
\end{proposition}

\begin{proof}
We first recall from Lemma~\ref{lem:truncation-obj} that truncations occur at most $r$ times and hence the total number of removed columns of matrices in $ U_{[p]}$ is bounded above by $r$.
It follows from Lemmas~\ref{lem:truncation-obj}, \ref{lem:lambda-k} and \ref{lem:lambda-s} that if the $p$-th iteration is a truncation iteration and {the number of columns of the matrices in $ U_{[p]}$ is decreased from $r_1$ to some $r_2<r_1$ during this iteration\footnote{Here {$r_1$} is the number of columns of the matrices in $U_{[p-1]}$.}}, then we have
\[
f( U_{[p]})\geq f( U_{[p-1]})-(r_1-r_2)\kappa^2.
\]
After all such truncation iterations, the value of the objective function decreases at most $r\kappa^2$. Moreover, at each iteration without truncation, the value of the objective function is nondecreasing by Proposition~\ref{prop:subfficient}. We  also recall that $r\kappa^2<f( U_{[0]})$ by the choice of $\kappa$ in Algorithm~\ref{algo}. Hence, 
\[
f( U_{[p]})=\sum_{j=1}^{r_2} \lambda_j(U_{[p]})^2\geq f(U_{[0]})-(r-r_2)\kappa^2>r_2\kappa^2.
\]
In particular, this implies that there exists some $1\le j \le r_2$ such that $\lvert \lambda_j(U_{[p]}) \rvert > \kappa$ and thus the $j$-th column of $U_{[p]}$ cannot be removed by the truncation step. Therefore, the number of columns of the factor matrices in $ U_{[p]}$ is not zero for all $p$. Since there are only finitely many truncation iterations, one can find some $N_0$ such that the $p$-th iteration is not a truncation iteration for all $p\geq N_0$.
\end{proof}

Next let us consider the following optimization problem
\begin{equation}\label{eq:sota-unconstrained}
\min_{ U} h( U):=-f( U)+\sum_{i=1}^s\delta_{{\V(r,n_i)}}(U^{(i)})+\sum_{i=s+1}^k\delta_{\B(r,n_i)}(U^{(i)}),
\end{equation}
where $\delta_{\V(r,n_i)}(U^{(i)})$ and $\delta_{\B(r,n_i)}(U^{(i)})$ are respectively the indicator functions of $\V(r,n_i)$ and $\B(r,n_i)$ defined by \eqref{eq:indicator}. By \cite[Theorem~8.2 and Proposition~10.5]{RW-98}, we have
\begin{multline*}
\partial \bigg(\sum_{i=1}^s\delta_{\V(r,n_i)}(U^{(i)})+\sum_{i=s+1}^k\delta_{\B(r,n_i)}(U^{(i)})\bigg)\\
=\partial \delta_{\V(r,n_1)}(U^{(1)})\times\dots\times \partial \delta_{\V(r,n_s)}(U^{(s)})\times\partial \delta_{\B(r,n_{s+1})}(U^{(s+1)})\times\dots\times\partial \delta_{\B(r,n_{k})}(U^{(k)}).
\end{multline*}
Moreover, according to \cite[Exercise~8.8]{RW-98} and \eqref{eq:normal-sub} we also obtain that
\begin{equation}\label{eq:sub-sum}
\partial h( U)=-\nabla f( U)+ \partial \bigg(\sum_{i=1}^s\delta_{\V(r,n_i)}(U^{(i)})+\sum_{i=s+1}^k\delta_{\B(r,n_i)}(U^{(i)})\bigg).
\end{equation}
{For simplicity, we abbreviate $\nabla_{U^{(i)}} f( U)$ as $\nabla_i f( U)$}. As in the proof of Proposition~\ref{def:kkt}, we have
\begin{equation}\label{eq:partial gradient}
\nabla_i f( U)=2V^{(i)}\Lambda,\quad 1 \le i \le k.
\end{equation}
Thus, it follows from {Proposition}~\ref{def:kkt}, \eqref{eq:normal-stf}, \eqref{eq:normal-sub} and \eqref{eq:sub-sum} that critical points of $h$ are exactly KKT points of problem \eqref{eq:sota-max}.

It is straightforward to verify that $h$ is a KL function according to Lemma~\ref{lem:KL functions}, and \eqref{eq:sota-unconstrained} is an unconstrained reformulation of problem \eqref{eq:sota-max} in the sense that
\[
U_\ast \in \operatorname{argmax}_{U\in V_{\mathbf{n},r}} f(U) \iff U_\ast \in \operatorname{argmin}_{U} h(U).
\]
Moreover, we also have
\[
\max_{U\in V_{\mathbf{n},r}} f(U) = -\min_{U} h(U).
\]

\begin{lemma}[Subdifferential Bound]\label{lem:subdiff}
If the $(p+1)$-th iteration is not a truncation iteration, then there exists a subgradient $ W_{[p+1]}\in\partial h( U_{[p+1]})$ such that
\begin{equation}\label{eq:subdiff}
\| W_{[p+1]}\|_F\leq {2\sqrt{k}}(2r{\sqrt{k}}{\|\mathcal A\|^2} +\epsilon)\| U_{[p+1]}- U_{[p]}\|_F.
\end{equation}
\end{lemma}

\begin{proof}
The subdifferential set of subgradients of $h$ can be decomposed as:
\begin{multline}\label{eq:sub-block}
\partial h( U)=(-\nabla_1 f( U)+\partial\delta_{\V(r,n_1)}(U^{(1)}))\times\dots\times (-\nabla_s f( U)+\partial\delta_{\V(r,n_s)}(U^{(s)}))\\
\times (-\nabla_{s+1} f( U)+\partial\delta_{\B(r,n_{s+1})}(U^{(s+1)}))\times\dots\times (-\nabla_k f( U)+\partial\delta_{\B(r,n_s)}(U^{(s)})).
\end{multline}
Following notations in Algorithm~\ref{algo}, for each $1 \le j \le r$, we set
\begin{align}
\mathbf x_{j} &\coloneqq (\mathbf u^{(1)}_{j,[p+1]},\dots,\mathbf u^{(k)}_{j,[p+1]}), \\
\mathbf v^{(i)}_j &\coloneqq \mathcal A\tau_i(\mathbf x_{j}), \\ \lambda_j &\coloneqq \mathcal A\tau(\mathbf x_{j}), \\
V^{(i)} &\coloneqq \begin{bmatrix}\mathbf v^{(i)}_1&\dots&\mathbf v^{(i)}_r\end{bmatrix}, \label{eq:matrix-vi} \\
\Lambda &\coloneqq {\operatorname{diag}_2}(\lambda_1,\dots,\lambda_r). \label{eq:matrix-lambda}
\end{align}
where $\mathbf u^{(i)}_{j,[p+1]}$ is the {$j$-th} column of the matrix $U^{(i)}_{[p+1]}, 1 \le i \le k, 1 \le j \le r$.
In the following, we divide the proof into two parts.
\begin{enumerate}[label=(\roman*)]
\item\label{proof:sub-1} \if The situation for $i=1,\dots,s$.\fi For $i=1,\dots,s$, we have by \eqref{eq:algorithm1-polar} and \eqref{eq:proximal} that
\begin{equation}\label{eqn:proof of lemma 4.5-1}
V^{(i)}_{[p+1]}\Lambda^{(i)}_{[p+1]}+\alpha_{i,[p+1]} U^{(i)}_{[p]}=U^{(i)}_{[p+1]}S^{(i)}_{[p+1]},
\end{equation}
where $\alpha_{i,[p+1]} = \epsilon$ or $0$ depending on {whether or not} there is a proximal correction. {According to \eqref{eq:normal-stf}, \eqref{eq:normal-sub} and \eqref{eqn:proof of lemma 4.5-1}, we have
\[
-U^{(i)}_{[p+1]} \in \partial\delta_{\V(r,n_i)}(U^{(i)}_{[p+1]}), \quad
V^{(i)}_{[p+1]}\Lambda^{(i)}_{[p+1]}+\alpha_{i,[p+1]} U^{(i)}_{[p]} \in \partial\delta_{\V(r,n_i)}(U^{(i)}_{[p+1]}),
\]
which implies that $V^{(i)}_{[p+1]}\Lambda^{(i)}_{[p+1]}+\alpha_{i,[p+1]}\big(U^{(i)}_{[p]}-U^{(i)}_{[p+1]}\big)\in \partial\delta_{\V(r,n_i)}(U^{(i)}_{[p+1]})$. If we take
\begin{equation}\label{eq:w}
W^{(i)}_{[p+1]}:=-V^{(i)}\Lambda +V^{(i)}_{[p+1]}\Lambda^{(i)}_{[p+1]}+\alpha_{i,[p+1]}\big(U^{(i)}_{[p]}-U^{(i)}_{[p+1]}\big),
\end{equation}
then we have}
\[
W^{(i)}_{[p+1]}\in -V^{(i)}\Lambda +\partial\delta_{\V(r,n_i)}(U^{(i)}_{[p+1]}).
\]
On the other hand,
\small
\begin{align}
\|W^{(i)}_{[p+1]}\|_F
&=\|V^{(i)}\Lambda -V^{(i)}_{[p+1]}\Lambda^{(i)}_{[p+1]}-\alpha_{i,[p+1]}\big(U^{(i)}_{[p]}-U^{(i)}_{[p+1]}\big)\|_F\nonumber\\
&\leq\|V^{(i)}\Lambda -V^{(i)}_{[p+1]}\Lambda \|_F+\|V^{(i)}_{[p+1]}\Lambda -V^{(i)}_{[p+1]}\Lambda^{(i)}_{[p+1]}\|_F+\alpha_{i,[p+1]}\|U^{(i)}_{[p]}-U^{(i)}_{[p+1]}\|_F\nonumber\\
&\leq\|V^{(i)}-V^{(i)}_{[p+1]}\|_F\|\Lambda \|_F+\|V^{(i)}_{[p+1]}\|_F\|\Lambda -\Lambda^{(i)}_{[p+1]}\|_F+\alpha_{i,[p+1]}\|U^{(i)}_{[p]}-U^{(i)}_{[p+1]}\|_F\\
&\leq{\|\mathcal A\|}\|\Lambda \|_F\big(\sum_{j=1}^r\|\tau_i(\mathbf x_j)-\tau_i(\mathbf x^i_{j,[p+1]})\|\big)\nonumber\\
&\quad +\|V^{(i)}_{[p+1]}\|_F\|\mathcal A\|\big(\sum_{j=1}^r\|\tau(\mathbf x_j)-\tau(\mathbf x^i_{j,[p+1]})\|\big)+\alpha_{i,[p+1]}\|U^{(i)}_{[p]}-U^{(i)}_{[p+1]}\|_F\nonumber\\
&\leq \sqrt{r}\|\mathcal A\|^2\big(\sum_{j=1}^r\sum_{t=i+1}^{k}\|\mathbf u^{(t)}_{j,[p+1]}-\mathbf u^{(t)}_{j,[p]}\| \big)\nonumber\\
&\quad+\sqrt{r}\|\mathcal A\|^2\big(\sum_{j=1}^r\sum_{t=i}^{k}\|\mathbf u^{(t)}_{j,[p+1]}-\mathbf u^{(t)}_{j,[p]}\|\big)+\alpha_{i,[p+1]}\|U^{(i)}_{[p]}-U^{(i)}_{[p+1]}\|_F\nonumber\\
&{\leq 2\sqrt{r}\|\A\|^2\sum_{t=1}^k \left( \sum_{j=1}^r\big( \|\mathbf u^{(t)}_{j,[p+1]}-\mathbf u^{(t)}_{j,[p]}\|\big) \right)+\alpha_{i,[p+1]}\|U^{(i)}_{[p]}-U^{(i)}_{[p+1]}\|_F}\nonumber\\
&\leq {2r{\sqrt{k}}\|\mathcal A\|^2 \| U_{[p+1]}- U_{[p]}\|_F +\epsilon \|U^{(i)}_{[p]}-U^{(i)}_{[p+1]}\|_F},\label{eq:subdif-est}
\end{align}
\normalsize
where the third inequality follows from
\[
V^{(i)}-V^{(i)}_{[p+1]}=\begin{bmatrix}\mathcal A(\tau_i(\mathbf x_1)-\tau_i(\mathbf x^i_{1,[p+1]}))&\dots&\mathcal A(\tau_i(\mathbf x_r)-\tau_i(\mathbf x^i_{r,[p+1]}))\end{bmatrix},
\]
and a similar formula for $(\Lambda -\Lambda^{(i)}_{[p+1]})$, the fourth inequality is derived from
\[
|\mathcal A\tau(\mathbf x)|\leq \|\mathcal A\|
\]
for any vector $\mathbf x:=(\mathbf x_1,\dots,\mathbf x_k)$ with $\|\mathbf x_i\|=1$ for all $i=1,\dots,k$ and the last one follows from $\alpha \leq\epsilon$ and $\sum_{j=1}^r\|\mathbf u^{(t)}_{j,[p+1]}-\mathbf u^{(t)}_{j,[p]}\| \le \sqrt{r}\| U^{(t)}_{[p+1]}- U^{(t)}_{[p]}\|_F$ for all $t=1,\dots, k$.

\item\label{proof:sub-2} For $i=s,\dots,k$, recall from \eqref{eq:lambda-orth} that 
\begin{equation}\label{eq:lambda-als}
{\Lambda_{[p+1]}^{(i+1)} \coloneqq \diag_2(\lambda^i_{1,[p+1]},\dots,\lambda^i_{r,[p+1]})}.
\end{equation}
It follows from the alternating least square step in Algorithm~\ref{algo}, \eqref{eq:updatep} and \eqref{eq:lambda-k} that
\[
\mathbf{v}^{(i)}_{j,[p+1]}\lambda^i_{j,[p+1]}=\mathbf u^{(i)}_{j,[p+1]}(\lambda^i_{j,[p+1]})^2, \quad j=1,\dots,r,
\]
which can be written more compactly as 
\begin{equation}\label{eq:opt-k}
V^{(i)}_{[p+1]}\Lambda_{[p+1]}^{(i+1)} =U^{(i)}_{[p+1]}\big(\Lambda_{[p+1]}^{(i+1)} \big)^2.
\end{equation}
If we take
\begin{equation}\label{eq:w-k}
W^{(i)}_{[p+1]}:=-V^{(i)}\Lambda +V^{(i)}_{[p+1]}\Lambda_{[p+1]}^{(i+1)},
\end{equation}
then we have
\[
W^{(i)}_{[p+1]}\in -V^{(i)}\Lambda +\partial\delta_{\B(r,n_i)}(U^{(i)}_{[p+1]}).
\]
{By a similar argument as in \ref{proof:sub-1}, we have}
\begin{equation}\label{eq:subdif-est-als}
\|W^{(i)}_{[p+1]}\|_F=\|V^{(i)}\Lambda -V^{(i)}_{[p+1]}\Lambda_{[p+1]}^{(i+1)}\|_F\leq {2r\sqrt{k}}\|\mathcal A\|^2\| U_{[p+1]}- U_{[p]}\|_F.
\end{equation}
\end{enumerate}

Noticing that $2W_{[p+1]}\in \partial h( U_{[p+1]})$, we obtain \eqref{eq:subdiff} from \ref{proof:sub-1} and \ref{proof:sub-2}.
\end{proof}

The following is a classical result.
\begin{lemma}[Abstract Convergence]\cite{ABS-13}\label{lem:abs-conv}
Let $h : \mathbb R^n\rightarrow\mathbb R\cup\{\pm\infty\}$ be a proper lower semicontinuous function and {let $\{\x^{(k)}\}\subseteq \mathbb R^n$ be} a sequence satisfying the following conditions:
\begin{enumerate}[label=(\roman*)]
\item\label{lem:abs-conv:item1} there is a constant $\alpha>0$ such that
\[
h( \x^{(k)})-h( \x^{(k+1)})\ge \alpha\| \x^{(k+1)}- \x^{(k)}\|^2,
\]
\item\label{lem:abs-conv:item2} there is a constant $\beta>0$ and a {subgradient} $ \w^{(k+1)}\in \partial h( \x^{(k+1)})$ such that
\[
\| \w^{(k+1)}\|\leq \beta\| \x^{(k+1)}- \x^{(k)}\|,
\]
\item\label{lem:abs-conv:item3} there is a subsequence $\{ \x^{(k_i)}\}$ of $\{ \x^{(k)}\}$ and $ \x^*\in\mathbb R^n$ such that
\[
 \x^{(k_i)}\rightarrow \x^*\ \text{and }h( \x^{(k_i)})\rightarrow h( \x^*)\ \text{as }i\rightarrow\infty.
\]
\end{enumerate}
If $h$ has the Kurdyka-\L ojasiewicz property at the point $ \x^*$, then the whole sequence $\{ \x^{(k)}\}$ converges to $\x^*$, and $ \x^*$ is a critical point of $h$.
\end{lemma}

We remark that since the feasible domain of the optimization problem \eqref{eq:sota-max} is the product of Stiefel and {Oblique} manifolds which is compact, the iteration sequence $\{ U_{[p]}\}$ generated by Algorithm~\ref{algo} must be bounded. In particular condition \eqref{lem:abs-conv:item3} in Lemma~\ref{lem:abs-conv} is satisfied automatically.
\begin{proposition}\label{prop:monotone}
Given a sequence $\{ U_{[p]}\}$ generated by Algorithm~\ref{algo}, the sequence $\{f( U_{[p]})\}$ increases monotonically {for all $p\geq N_0$ with some positive $N_0$} and hence converges.
\end{proposition}

\begin{proof}
Since the sequence $\{ U_{[p]}\}$ is bounded, $\{f( U_{[p]})\}$ is also bounded by continuity. The convergence then follows from {Propositions~\ref{prop:subfficient} and \ref{prop:positive}}.
\end{proof}
\begin{theorem}[Global Convergence]\label{thm:global}
{Any sequence $\{ U_{[p]}\}$ generated by Algorithm~\ref{algo} is bounded and converges to a KKT point of the problem~\eqref{eq:sota}}.
\end{theorem}

\begin{proof}
The conclusion is obtained by applying Lemma~\ref{lem:abs-conv} to the unconstrained minimization problem \eqref{eq:sota-unconstrained}: Condition~\ref{lem:abs-conv:item1} in Lemma~\ref{lem:abs-conv} follows from Proposition~\ref{prop:subfficient}, condition~\ref{lem:abs-conv:item2} is verified by Lemma~\ref{lem:subdiff} and condition~\ref{lem:abs-conv:item3} is also ensured by noticing that the sequence $\{ U_{[p]}\}$ is bounded and the function $h$ defined in \eqref{eq:sota-unconstrained} is continuous on the product of Stiefel and Oblique manifolds. Lastly, the Kurdyka-\L ojasiewicz property of the function $h$ follows from Lemma~\ref{lem:KL functions}.
\end{proof}

\subsection{{Sublinear convergence rate}}\label{sec:sublinear}
In this subsection, we prove the sublinear convergence rate of Algorithm~\ref{algo}.
Since the \L ojasiewicz exponent of the function $h$ in \eqref{eq:sota-unconstrained} is unknown, the convergence rate analysis such as the one in \cite{BST-14} or \cite{ABS-13} cannot be applied directly. To resolve this issue, we start with the following function
\begin{equation}\label{eq:function-aml}
q( U, P) \coloneqq f( U)-\sum_{i=1}^s\langle P^{(i)},(U^{(i)})^\tp U^{(i)}-I\rangle-\sum_{i=s+1}^k\langle {\diag_2}(\mathbf p),(U^{(i)})^\tp U^{(i)}-I\rangle,
\end{equation}
{which is a polynomial of degree $2k$ in $N$ variables:
\[
( U, P) \coloneqq ((U^{(1)},\dots,U^{(k)}),(P^{(1)},\dots, P^{(s)},\mathbf p))
\in  \left( \prod_{i=1}^k \mathbb{R}^{n_i\times r} \right) \times (\SS^r)^{\times s} \times \mathbb{R}^r,
\]
where $N \coloneqq  {(1 + \sum_{i=1}^s  n_i)r + s\binom{r+1}{2} }$. Here we remind the reader that $f$ is the objective function of \eqref{eq:sota-max} and $q$ is the Lagrange function associated to $f$ and constraints of its variables.

Denote
\begin{equation}\label{eq:tau}
{\zeta} \coloneqq 1-\frac{1}{2k(6k-3)^{N-1}},
\end{equation}
which is the \textit{\L ojasiewicz exponent} of the polynomial $q$ obtained by Lemma~\ref{lemma:loja1}. We suppose that $ U^*$ is a KKT point of \eqref{eq:sota-max} with the unique\footnote{LICQ guarantees the uniqueness of the Lagrange multiplier.} multiplier $ P^*$.

By Proposition~\ref{def:kkt} we have
\[
\nabla q( U^*, P^*)=0.
\]
According to Lemma~\ref{lemma:loja1}, there exist some $\gamma,\mu>0$ such that
\[
\|\nabla q(U,P)\|_F\geq \mu|q( U, P)-q( U^*, P^*)|^{{\zeta}}\ \text{whenever }\|(U, P)-( U^*,P^*)\|_F\leq \gamma.
\]
Therefore, we obtain
\begin{equation}\label{eq:loj-ineq}
\sum_{i=1}^s\|\nabla_i f( U)-2U^{(i)}P^{(i)}\|_F^2+\sum_{i=s+1}^k\|\nabla_i f( U)-2U^{(i)} {\diag_2}({\mathbf p})\|_F^2\geq \mu^2(f( U)-f( U^*))^{2 {\zeta}}
\end{equation}
for any feasible point $ U$ of \eqref{eq:sota-max} {and a point $P\in (\SS^r)^{\times s} \times \mathbb{R}^r$} such that $\|( U, P)-( U^*, P^*)\|_F\leq \gamma$.
\begin{theorem}[Sublinear Convergence Rate]\label{thm:sublinear}
Let $\{ U_{[p]}\}$ be a sequence generated by Algorithm~\ref{algo} for a given nonzero tensor $\mathcal A\in\mathbb R^{n_1}\otimes\dots\otimes\mathbb R^{n_k}$ and let ${\zeta}$ be defined by \eqref{eq:tau}.  The following statements hold:
\begin{enumerate}[label=(\alph*)]
\item \label{thm:sublinear-item1} {the sequence} $\{f( U_{[p]})\}$ converges to $f^*$, with {a} sublinear convergence rate at least
$O(p^{\frac{1}{1-2{\zeta}}})$, that is, there exist $M_1>0$ and ${p_1} \in \mathbb{N}$ such that for all $p \geq {p_1}$,
\begin{equation}\label{eq:sub-linear}
f^*-f( U_{[p]}) \leq M_1 \, p^{\frac{1}{1-2{\zeta}}}.
\end{equation}
\item \label{thm:sublinear-item2} $\{ U_{[p]}\}$ converges to $ U^*$ globally with {a} sublinear convergence rate at least
{$O(p^{\frac{{\zeta}-1}{2{\zeta}-1}})$}, that is, there exist $M_2>0$ and ${p_1} \in \mathbb{N}$ such that for all $p \ge {p_1}$,
\[
\| U_{[p]} -  U^*\|_F \leq M_2 \, p^{\frac{{\zeta}-1}{2{\zeta}-1}}.
\]
\end{enumerate}
\end{theorem}

\begin{proof}
{By Theorem~\ref{thm:global}, $\{U_{[p]}\}$ converges globally to a KKT point $U^*=(U^{(*,1)},\dots,U^{(*,k)})$ with a unique Lagrange multiplier $P^*=(P^{(*,1)},\dots,P^{(*,s)},\mathbf p^{*})$.} Again, we discuss with respect to different values of $1 \le i \le k$.
\begin{enumerate}[label=(\roman*)]
\item For $1\le i \le s$, we let
$P^{(i)}_{[p]} := S^{(i)}_{[p]}-\alpha_{i,[p]} {I_r}$, where
\begin{align*}
\alpha_{i,[p]} &:=
\begin{cases}
\epsilon, &\text{if proximal correction is executed},\\
0, &\text{otherwise}
\end{cases} \\
S^{(i)}_{[p]} &:=\begin{cases}
(U^{(i)}_{[p]})^\tp (V^{(i)}_{[p]}\Lambda^{(i)}_{[p]}+\epsilon U^{(i)}_{[p-1]}),&\text{if proximal correction is executed},\\
(U^{(i)}_{[p]})^\tp V^{(i)}_{[p]}\Lambda^{(i)}_{[p]},&\text{otherwise}. \end{cases}
\end{align*}
Note that $\{ U_{[p]}\}$ converges by Theorem~\ref{thm:global} and hence $\{ V^{(i)}_{[p]}\Lambda^{(i)}_{[p]} \}$ converges {by \eqref{eq:block-xp}, \eqref{eq:lambda+vp} and \eqref{eq:Lambda + Vp}}. Recall that in Algorithm~\ref{algo}, the proximal correction step is determined by the minimal singular value of $V^{(i)}_{[p]}\Lambda^{(i)}_{[p]}$. Thus $\alpha_{i,[p]}$ will be stable at some $\alpha_i$ for sufficiently large $p$ (say $p\geq p_0$) and $1 \le i \le s$.
By \eqref{eq:partial gradient}, \eqref{eq:algorithm1-polar}, \eqref{eq:w} and Lemma~\ref{lem:subdiff}, we have
\begin{align}\label{eq:sub-sublinear}
{\|\nabla_i f(U_{[p+1]})-2U_{[p+1]}^{(i)} P_{[p+1]}^{(i)}\|_F}&=\|\nabla_i f( U_{[p+1]})-2U^{(i)}_{[p+1]}S^{(i)}_{[p+1]}+2\alpha_i U^{(i)}_{[p+1]}\|_F\nonumber\\
&=\|-2W^{(i)}_{[p+1]}\|_F\nonumber\\
&\leq {C_0}\| U_{[p+1]}- U_{[p]}\|_F
\end{align}
where $C_0:={2(2r\sqrt{k}{\|\mathcal A\|^2} +\epsilon)}>0$ is a constant, {and the inequality follows from \eqref{eq:subdif-est}}.

Next we denote
\begin{equation}\label{eq:multiPi}
\hat P_{[p+1]}^{(i)}:=\frac{1}{4}\big((U^{(i)}_{[p+1]})^\tp\nabla_i f( U_{[p+1]})+(\nabla_i f(U_{[p+1]}))^\tp U^{(i)}_{[p+1]}\big).
\end{equation}
By construction, $\hat P_{[p+1]}^{(i)}$ is a symmetric matrix. 
By \eqref{eq:partial gradient}, we have $\nabla_i f(U_{[p+1]})=2V^{(i)}\Lambda$. Thus
\begin{align}\label{eq:sub-linear-2}
&\|U_{[p+1]}^{(i)} P_{[p+1]}^{(i)}-U_{[p+1]}^{(i)}\hat P_{[p+1]}^{(i)}\|_F\nonumber\\
=&\| \left( S^{(i)}_{[p+1]}-\alpha_{i,[p+1]} I_r \right)-\frac{1}{2}\Big((U^{(i)}_{[p+1]})^\tp V^{(i)}\Lambda+(V^{(i)}\Lambda)^\tp U^{(i)}_{[p+1]}\Big)\|_F\nonumber\\
=&\frac{1}{2}\|(U^{(i)}_{[p+1]})^\tp\big(V^{(i)}_{[p+1]}\Lambda^{(i)}_{[p+1]}
+\alpha_{i,[p+1]}U^{(i)}_{[p]}\big)-\alpha_{i,[p+1]} I_r-(U^{(i)}_{[p+1]})^\tp V^{(i)}\Lambda\nonumber\\
&\quad +\big(V^{(i)}_{[p+1]}\Lambda^{(i)}_{[p+1]}+\alpha_{i,[p+1]}U^{(i)}_{[p]}\big)^\tp U^{(i)}_{[p+1]}-\alpha_{i,[p+1]} I_r-(V^{(i)}\Lambda)^\tp U^{(i)}_{[p+1]}\|_F\nonumber\\
\leq&\|(U^{(i)}_{[p+1]})^\tp\big(V^{(i)}_{[p+1]}\Lambda^{(i)}_{[p+1]}
+\alpha_{i,[p+1]}U^{(i)}_{[p]}\big)-\alpha_{i,[p+1]} I_r-(U^{(i)}_{[p+1]})^\tp V^{(i)}\Lambda\|_F\nonumber\\
\leq& \|(U^{(i)}_{[p+1]})^\tp\big(V^{(i)}_{[p+1]}\Lambda^{(i)}_{[p+1]}-V^{(i)}\Lambda\big)\|_F
+\epsilon\|U^{(i)}_{[p+1]}-U^{(i)}_{[p]}\|_F\nonumber\\
\leq& C_1\|V^{(i)}_{[p+1]}\Lambda^{(i)}_{[p+1]}-V^{(i)}\Lambda\|_F+\epsilon\|U^{(i)}_{[p+1]}-U^{(i)}_{[p]}\|_F\nonumber\\
\leq& C_2\|U^{(i)}_{[p+1]}-U^{(i)}_{[p]}\|_F,
\end{align}
where the first equality holds because $U^{(i)}_{[p+1]}\in\V(r,n_i)$, the second equality follows from the fact that the matrix $S^{(i)}_{[p+1]}=(U^{(i)}_{[p+1]})^\tp\big(V^{(i)}_{[p+1]}\Lambda^{(i)}_{[p+1]}+\alpha_{i,[p+1]}U^{(i)}_{[p]}\big)$ is symmetric, the second inequality follows from Lemma~\ref{lem:distance} and the last inequality can be obtained by the proof of Lemma~\ref{lem:subdiff}, from which one can also easily see that $C_1,C_2>0$ are some constants depending only on the tensor $\mathcal A$.

Lastly, we combine \eqref{eq:sub-sublinear} and \eqref{eq:sub-linear-2} to obtain
\begin{align}\label{eq:sub-linear-1}
&\|\nabla_i f( U_{[p+1]})-2U_{[p+1]}^{(i)}\hat P_{[p+1]}^{(i)}\|_F\nonumber\\
=&\|\nabla_i f( U_{[p+1]})-2U_{[p+1]}^{(i)} P_{[p+1]}^{(i)}+2U_{[p+1]}^{(i)} P_{[p+1]}^{(i)}-2U_{[p+1]}^{(i)}\hat P_{[p+1]}^{(i)}\|_F\nonumber\\
\leq& \|\nabla_i f(U_{[p+1]})-2U_{[p+1]}^{(i)} P_{[p+1]}^{(i)}\|_F+2\|U_{[p+1]}^{(i)} P_{[p+1]}^{(i)}-U_{[p+1]}^{(i)}\hat P_{[p+1]}^{(i)}\|_F\nonumber\\
\leq& C\|U^{(i)}_{[p+1]}-U^{(i)}_{[p]}\|_F,
\end{align}
where $C\coloneqq C_0+2C_2>0$ is a constant. Moreover, since $\{U_{[p]}\}$ is convergent by Theorem~\ref{thm:global}, it is valid to take limits on both sides of \eqref{eq:multiPi} to obtain
\begin{equation}\label{eq:limit-apd}
\lim_{p\rightarrow\infty}\hat P^{(i)}_{[p]}=(U^{(*,i)})^\tp V^{(*,i)}\Lambda^*=P^{(*,i)}.
\end{equation}

\item For $s+1 \le i \le k$, {we denote}
\[
\mathbf p^{(i)}_{[p]}:= \left( (\lambda^i_{1,[p]})^2,\dots, (\lambda^i_{r,[p]})^2 \right)^\tp \in \mathbb{R}^r.
\]
It follows from \eqref{eq:opt-k}, \eqref{eq:w-k} and \eqref{eq:subdif-est-als} that
\[
\|\nabla_i f( U_{[p+1]})-2U^{(i)}_{[p+1]}{\diag_2}(\mathbf p^{(i)}_{[p+1]})\|_F\leq {D_0}\| U_{[p+1]}- U_{[p]}\|_F,
\]
where $D_0:=4r{\sqrt{k}}\|\mathcal A\|^2>0$ is a constant.

Next we let $\hat\lambda_{j,[p]} \coloneqq \A\tau( \mathbf u^{(1)}_{j,[p]},\dots,\mathbf u^{(k)}_{j,[p]})$ and let
\[
\hat{\mathbf p}_{[p]}\coloneqq((\hat\lambda_{1,[p]})^2,\dots,(\hat\lambda_{r,[p]})^2)^\tp.
\]
We remark that $\hat\lambda_{j,[p]}$ only depends  on $U_{[p]}$.
Now we have
\begin{align}\label{eq:sub-linear-4}
\|\hat{\mathbf p}_{[p+1]}-\mathbf p^{(i)}_{[p+1]}\|^2&=\sum_{j=1}^r((\hat\lambda_{j,[p+1]})^2-(\lambda^i_{j,[p+1]})^2)^2\nonumber\\
&=\sum_{j=1}^r(\hat\lambda_{j,[p+1]}+\lambda^i_{j,[p+1]})^2
(\hat\lambda_{j,[p+1]}-\lambda^i_{j,[p+1]})^2\nonumber\\
&\leq 4\|\A\|^2\sum_{j=1}^r(\hat\lambda_{j,[p+1]}-\lambda^i_{j,[p+1]})^2\nonumber   \\
&\leq D_1\|U_{[p+1]}-U_{[p]}\|_F^2,
\end{align}
where $D_1 \coloneqq 4rk\|\A\|^4>0$, {the first inequality follows from \eqref{eq:unit}} and the last one follows from a similar argument as in the proof of Lemma~\ref{lem:subdiff}.

Finally we arrive at
\begin{align}
&\|\nabla_i f( U_{[p+1]})-2U^{(i)}_{[p+1]}{\diag_2}(\hat{\mathbf p}_{[p+1]})\|_F\nonumber\\
=&
\|\nabla_i f( U_{[p+1]})-2U^{(i)}_{[p+1]}{\diag_2}(\mathbf p^{(i)}_{[p+1]})+2U^{(i)}_{[p+1]}{\diag_2}(\mathbf p^{(i)}_{[p+1]})-2U^{(i)}_{[p+1]}{\diag_2}(\hat{\mathbf p}_{[p+1]})\|_F\nonumber\\
\leq& \|\nabla_i f( U_{[p+1]})-2U^{(i)}_{[p+1]}{\diag_2}(\mathbf p^{(i)}_{[p+1]})\|_F +\|2U^{(i)}_{[p+1]}{\diag_2}(\mathbf p^{(i)}_{[p+1]})-2U^{(i)}_{[p+1]}{\diag_2}(\hat{\mathbf p}_{[p+1]})\|_F\nonumber\\
\leq&D_0\|U_{[p+1]}-U_{[p]}\|_F+2\sqrt{D_1}\|U_{[p+1]}-U_{[p]}\|_F\nonumber\\
=&D\|U_{[p+1]}-U_{[p]}\|_F,\label{eq:sub-linear-3}
\end{align}
where $D=D_0+2\sqrt{D_1}>0$ is a constant. Moreover, we also have
\begin{equation}\label{eq:limit-als}
\lim_{p\rightarrow\infty}\hat{\mathbf p}_{[p]}=\Diag_2\big((\Lambda^*)^2\big)=\mathbf p^{*},
\end{equation}
{by the convergence of the sequence $\{U_{[p]}\}$.}
\end{enumerate}

Now we are ready to prove \ref{thm:sublinear-item1}. To do that, we denote $\hat P_{[p]}:=(\hat P^{(1)}_{[p]},\dots,\hat P^{(s)}_{[p]},\hat{\mathbf p}_{[p]})$ and observe {from \eqref{eq:limit-apd} and \eqref{eq:limit-als}} that for any sufficiently large $p$, we have
\[
\|( U_{[p]}, {\hat P_{[p]}})-( U^*, P^*)\|_F\leq \gamma,
\]
from which we may obtain
\scriptsize
\begin{align}
&\ \ \quad \mu^2(f( U_{[p]})-f( U^*))^{2{\zeta}}\nonumber \\
&\leq \sum_{i=1}^s\|\nabla_i f( U_{[p]})-2U^{(i)}_{[p]}{\hat P^{(i)}_{[p]}}\|_F^2+\sum_{i=s+1}^k\|\nabla_i f( U_{[p]})-2U^{(i)}_{[p]}{\diag_2(\hat{\mathbf p}_{[p]})}\|_F^2\nonumber\\
& {\leq \sum_{i=1}^s \left( \|\nabla_i f( U_{[p+1]})-2U^{(i)}_{[p+1]} \hat P^{(i)}_{[p+1]}\|_F+\|\nabla_i f( U_{[p+1]})-2U^{(i)}_{[p+1]} \hat P^{(i)}_{[p+1]}-\nabla_i f( U_{[p]})+2U^{(i)}_{[p]}\hat P^{(i)}_{[p]}\|_F \right)^2}\nonumber\\
&+\sum_{i=s+1}^k \left( \|\nabla_i f( U_{[p+1]})-2U^{(i)}_{[p+1]}\diag_2(\hat{\mathbf p}_{[p+1]})\|_F + \|\nabla_i f( U_{[p+1]})-2U^{(i)}_{[p+1]}\diag_2(\hat{\mathbf p}_{[p+1]})-\nabla_i f( U_{[p]})+2U^{(i)}_{[p]}\diag_2(\hat{\mathbf p}_{[p]})\|_F \right)^2 \nonumber \\
&\leq {(s(C+L_1)^2+(k-s)(D+L_2)^2)\| U_{[p+1]}- U_{[p]}\|_F^2}\label{eq:fun-iteration}\\
&\leq {L(f( U_{[p+1]})-f( U_{[p]}))},\nonumber
\end{align}
\normalsize
where $L = \frac{2(s(C+L_1)^2+(k-s)(D+L_2)^2)}{\min\{2\kappa^2,\epsilon\}}>0$ is a constant. Here the first inequality follows from \eqref{eq:loj-ineq}, the third inequality from \eqref{eq:sub-linear-1}, \eqref{eq:sub-linear-3} and the Lipschitz continuity of the function $\nabla_i f( U_{[p+1]})-2U^{(i)}_{[p+1]} \hat P^{(i)}_{[p+1]}$ (resp. $\nabla_i f( U_{[p+1]})-2U^{(i)}_{[p+1]}\diag_2(\hat{\mathbf p}_{[p+1]})$) with a Lipschitz constant $L_1$ (resp. $L_2$) and  the last inequality is a consequence of Proposition~\ref{prop:subfficient}. According to \eqref{eq:fun-iteration}, if we set $\beta_p \coloneqq f( U^*)-f( U_{[p]})$ then we must have
\[
{\beta_{p}-\beta_{p+1}}\geq M\beta_p^{2{\zeta}}
\]
for some constant $M>0$. The rest of the proof is similar to that in \cite[Theorem~3.2--(1)]{HL-18}. We define a function $h(x)\coloneqq x^{-2\zeta}$, then it follows that
\[
\beta_{p}-\beta_{p+1}\geq M\beta_{p}^{2\zeta}=Mh(\beta_{p})^{-1}.
\]
We notice that $h(x)=t'(x)$ where $ t(x) \coloneqq \frac{x^{1-2\zeta}}{1-2\zeta} $ and  that $h$ is decreasing on $\mathbb{R}_{++}$. Hence we have
\begin{equation*}
M\leq h(\beta_{p})(\beta_{p}-\beta_{p+1}) \leq \int_{\beta_{p+1}}^{\beta_{p}}h(x)\,dx	=t(\beta_{p})-t(\beta_{p+1})
=\frac{1}{2\zeta-1}(\beta_{p+1}^{1-2\zeta}-\beta_{p}^{1-2\zeta}).
\end{equation*}
Let $p_0$ be a positive integer such that the above analysis is guaranteed for all $p\geq p_0$ and let ${p_1 \coloneqq 2p_0}$. For each $p \ge{p_1 }$, we have by induction that
\[
\beta_{p}^{1-2\zeta}\geq M(2\zeta-1)+\beta_{p-1}^{1-2\zeta}\geq \cdots\geq M(2\zeta-1)(p-{p_0})+\beta_{{p_0}}^{1-2\zeta}\geq M(2\zeta-1)(p-{p_0}).
\]
By \eqref{eq:tau}, we have $\zeta > 1/2$ which implies that
\[
\beta_{p}\leq [M(2\zeta-1)(p-{p_0})]^{\frac{1}{1-2\zeta}}=
\left[ M(2\zeta-1)\frac{(p-{p_0})}{p}\right]^{\frac{1}{1-2\zeta}}p^{\frac{1}{1-2\zeta}} \le \left[ M(\zeta-\frac{1}{2})\right]^{\frac{1}{1-2\zeta}} p^{\frac{1}{1-2\zeta}}.
\]
Therefore, if we take $M_{1} \coloneqq \left[ M(\zeta-\frac{1}{2})\right]^{\frac{1}{1-2\zeta}}$ then for each $p\geq {p_1}$, it must hold that
\[
0\leq \beta_{p}\leq M_{1}p^{\frac{1}{1-2\zeta}},
\]
which completes the proof of \ref{thm:sublinear-item1}.

To prove \ref{thm:sublinear-item2}, we observe that by \eqref{eq:fun-iteration} the inequality
\[
\beta_p^{\zeta}\leq K\| U_{[p+1]}- U_{[p]}\|_F
\]
holds for some constant $K>0$. The rest of the proof follows from the same method as in \cite[Theorem~3.2--(2)]{HL-18}.
\end{proof}

We conclude this subsection by remarking that Algorithm~\ref{algo} belongs to the category of first order methods, for which only $O(1/p)$ convergence rate can be obtained for a general non-convex problem {\cite{Beck-book}}.  Surprisingly, however, for problem \eqref{eq:sota-max}, Theorem~\ref{thm:sublinear} proves that the convergence rate can be faster than the classical rate $O(1/p)$. We also notice that the convergence rate in Theorem~\ref{thm:sublinear} is already best possible one can obtain without any further hypothesis, in the sense that there exist examples \cite{EK-15,EHK-15} exhibiting a sublinear convergence rate even for $r=1$.
\subsection{{R-Linear convergence}}\label{sec:linear}
In this subsection, we establish the R-linear convergence of Algorithm~\ref{algo} for a generic tensor $\mathcal{A} \in \mathbb{R}^{n_1}\otimes \cdots \otimes \mathbb{R}^{n_k}$.
\begin{lemma}[Relative Error from APD]\label{lem:gradient-diff}
There exists a constant $\gamma_1>0$ such that
\[
\|\nabla_i f( U_{[p+1]})-U^{(i)}_{[p+1]}(\nabla_i f( U_{[p+1]}))^\tp U^{(i)}_{[p+1]}\|_F \leq \gamma_1 \| U_{[p]}- U_{[p+1]}\|_F
\]
for all $1 \le i \le s$ and $p\in \mathbb{N}$.
\end{lemma}

\begin{proof}
By Algorithm~\ref{algo}, we have
\begin{equation}\label{eq:linear-polar}
V^{(i)}_{[p+1]}\Lambda^{(i)}_{[p+1]}+{\alpha_{i,[p+1]}} U^{(i)}_{[p]}=U^{(i)}_{[p+1]}S^{(i)}_{[p+1]}
\end{equation}
where $S^{(i)}_{[p+1]}$ is a symmetric positive semidefinite matrix and $\alpha_{i,[p+1]} = 0$ unless there is a proximal correction, in which case $\alpha_{i,[p+1]} = \epsilon$. Since $U^{(i)}_{[p+1]}\in \V(r,n_i)$ is an orthonormal matrix, we have
\begin{equation}\label{eq:linear-symmetry}
S^{(i)}_{[p+1]}=(U^{(i)}_{[p+1]})^\tp \big(V^{(i)}_{[p+1]}\Lambda^{(i)}_{[p+1]}+{\alpha_{i,[p+1]}} U^{(i)}_{[p]}\big)=\big(V^{(i)}_{[p+1]}\Lambda^{(i)}_{[p+1]}+{\alpha_{i,[p+1]}} U^{(i)}_{[p]}\big)^\tp U^{(i)}_{[p+1]},
\end{equation}
where the second equality follows from the symmetry of the matrix $S^{(i)}_{[p+1]}$.

If we let
\[
W^{(i)}_{[p+1]} \coloneqq V^{(i)}\Lambda -V^{(i)}_{[p+1]}\Lambda^{(i)}_{[p+1]}-{\alpha_{i,[p+1]}}\big(U^{(i)}_{[p]}-U^{(i)}_{[p+1]}\big),
\]
then it follows from Lemma~\ref{lem:subdiff} that there exists some constant $\gamma_0 > 0$ such that
\begin{equation}\label{eq:relative-gamma0}
\|W^{(i)}_{[p+1]}\|_F\leq \gamma_0 \| U_{[p]}- U_{[p+1]}\|_F.
\end{equation}
Moreover, we observe that there exists some constant $\eta_1>0$ such that
\begin{align}
&\quad\ \|V^{(i)}_{[p+1]}\Lambda^{(i)}_{[p+1]}-U^{(i)}_{[p+1]}(V^{(i)}\Lambda)^\tp U^{(i)}_{[p+1]}\|_F\nonumber\\
&=\|U^{(i)}_{[p+1]}S^{(i)}_{[p+1]}-{\alpha_{i,[p+1]}} U^{(i)}_{[p]}-U^{(i)}_{[p+1]}(V^{(i)}\Lambda)^\tp U^{(i)}_{[p+1]}\|_F\nonumber\\
&=\|U^{(i)}_{[p+1]}\big(V^{(i)}_{[p+1]}\Lambda^{(i)}_{[p+1]}+{\alpha_{i,[p+1]}} U^{(i)}_{[p]}\big)^\tp U^{(i)}_{[p+1]}-{\alpha_{i,[p+1]}} U^{(i)}_{[p]}-U^{(i)}_{[p+1]}(V^{(i)}\Lambda)^\tp U^{(i)}_{[p+1]}\|_F\nonumber\\
&\leq\|U^{(i)}_{[p+1]}\big((V^{(i)}_{[p+1]}\Lambda^{(i)}_{[p+1]})^\tp -(V^{(i)}\Lambda)^\tp \big)U^{(i)}_{[p+1]}\|_F+{\alpha_{i,[p+1]}}\| U^{(i)}_{[p+1]}(U^{(i)}_{[p]})^\tp U^{(i)}_{[p+1]}-U^{(i)}_{[p]}\|_F\nonumber\\
&\leq   \|V^{(i)}_{[p+1]}\Lambda^{(i)}_{[p+1]}-V^{(i)}\Lambda\|_F+{\alpha_{i,[p+1]}}\| U^{(i)}_{[p+1]}(U^{(i)}_{[p]})^\tp U^{(i)}_{[p+1]}-U^{(i)}_{[p]}\|_F\nonumber\\
&\leq \eta_1\| U_{[p+1]}- U_{[p]}\|_F.\label{eq:relative-2}
\end{align}
Here the two equalities {follow} from \eqref{eq:linear-polar} and \eqref{eq:linear-symmetry} respectively, the penultimate inequality\footnote{Since $U$ is orthonormal, we must have
\[
\|UAU\|_F^2=\|AU\|_F^2=\langle AUU^\tp,A\rangle\leq \|A\|_F^2.
\]
}
follows from the fact that $U^{(i)}_{[p+1]}\in V(r,n_i)$ and the last inequality is obtained by combining \eqref{eq:relative-gamma0} and the relation
\[
\|(U^{(i)}_{[p]})^\tp U^{(i)}_{[p+1]}-I_r\|_F\leq \|U^{(i)}_{[p]}-U^{(i)}_{[p+1]}\|_F,
\]
which is obtained by a direct application of Lemma~\ref{lem:distance}.

Combining \eqref{eq:relative-gamma0}, \eqref{eq:relative-2} and the fact that $\nabla_i f( U_{[p+1]})=2V^{(i)}\Lambda$ obtained in $\eqref{eq:partial gradient}$, we arrive at the following estimate:
\begin{align}
&\quad\ \frac{1}{2}\|\nabla_i f( U_{[p+1]})-U^{(i)}_{[p+1]}(\nabla_i f( U_{[p+1]}))^\tp U^{(i)}_{[p+1]}\|_F\nonumber\\
&=\|V^{(i)}\Lambda -U^{(i)}_{[p+1]}(V^{(i)}\Lambda)^\tp U^{(i)}_{[p+1]}\|_F\nonumber\\
&=\|W^{(i)}_{[p+1]}+V^{(i)}_{[p+1]}\Lambda^{(i)}_{[p+1]}+{\alpha_{i,[p+1]}}\big(U^{(i)}_{[p]}-U^{(i)}_{[p+1]}\big)-U^{(i)}_{[p+1]}(V^{(i)}\Lambda)^\tp U^{(i)}_{[p+1]}\|_F\nonumber\\
&\leq \|W^{(i)}_{[p+1]}\|_F+{\alpha_{i,[p+1]}}{\|U^{(i)}_{[p]}-U^{(i)}_{[p+1]}\|_F}+\|V^{(i)}_{[p+1]}\Lambda^{(i)}_{[p+1]}-U^{(i)}_{[p+1]}(V^{(i)}\Lambda)^\tp U^{(i)}_{[p+1]}\|_F\nonumber\\
&\leq \eta_2\| U_{[p]}- U_{[p+1]}\|_F+\|V^{(i)}_{[p+1]}\Lambda^{(i)}_{[p+1]}-U^{(i)}_{[p+1]}(V^{(i)}\Lambda)^\tp U^{(i)}_{[p+1]}\|_F\nonumber\\
&\leq \eta_3 \| U_{[p+1]}- U_{[p]}\|_F,\label{eq:relative-1}
\end{align}
where $\eta_2= \gamma_0+\epsilon$ and $\eta_3\coloneqq \eta_1+\eta_2>0$. The desired inequality easily follows from \eqref{eq:relative-1} {with $\gamma_1:=2\eta_3$}.
\end{proof}
\begin{lemma}[Relative Error from ALS]\label{lem:gradient-diff-ob}
There exists a constant $\gamma_2>0$ such that
\[
\|\nabla_i f( U_{[p+1]})-U^{(i)}_{[p+1]}\diag_2\big((\nabla_i f( U_{[p+1]}))^\tp U^{(i)}_{[p+1]}\big)\|_F \leq \gamma_2 \| U_{[p]}- U_{[p+1]}\|_F
\]
for all $s+1 \le i \le k$ and $p\in \mathbb{N}$.
\end{lemma}

\begin{proof}
For each $s+1 \leq i\leq k$ and $p\in \mathbb{N}$, we recall from \eqref{eq:lambda-orth}, \eqref{eq:matrix-vi} and \eqref{eq:matrix-lambda} respectively that 
\begin{align*}
\Lambda^{(i+1)}_{[p+1]} & =\diag_2(\lambda^i_{1,[p+1]},\dots,\lambda^i_{r,[p+1]}), \\
V^{(i)} &= \begin{bmatrix}\mathbf v^{(i)}_1&\dots&\mathbf v^{(i)}_r\end{bmatrix}, \\ 
\Lambda &= {\operatorname{diag}_2}(\lambda_1,\dots,\lambda_r).
\end{align*}
We observe that $V^{(i)} = [\mathbf{v}^{(i)}_1,\dots, \mathbf{v}^{(i)}_r]$ is obtained from $U_{[p+1]}$ as $\mathbf v^{(i)}_j=\A\tau_i(\x_j)$
and hence
\[
\langle \mathbf v^{(i)}_j, \mathbf u^{(i)}_{j,[p+1]}\rangle=\A\tau(\x_j)=\lambda^k_{j,[p+1]}.
\]
This implies 
\begin{align}
&\quad\ \| V^{(i)}_{[p+1]}\Lambda^{(i+1)}_{[p+1]}-U^{(i)}_{[p+1]}\diag_2\big((V^{(i)}\Lambda)^\tp U^{(i)}_{[p+1]}\big)\|^2_F\nonumber\\
&=\|U^{(i)}_{[p+1]}(\Lambda^{(i+1)}_{[p+1]})^2-U^{(i)}_{[p+1]}\diag_2\big((V^{(i)}\Lambda)^\tp U^{(i)}_{[p+1]}\big)\|^2_F\nonumber \\
&=\|(\Lambda^{(i+1)}_{[p+1]})^2-\diag_2\big((V^{(i)}\Lambda)^\tp U^{(i)}_{[p+1]}\big)\|^2_F\nonumber\\
&=\sum_{j=1}^r((\lambda^i_{j,[p+1]})^2-(\lambda^k_{j,[p+1]})^2)^2\nonumber  \\
&\leq \xi \| U_{[p+1]}- U_{[p]}\|_F^2,\label{eq:relative-2-ob}
\end{align}
where $\xi>0$ is some constant, the first equality follows from \eqref{eq:opt-k} and the last
inequality follows from \if Lemma~\ref{lem:lambda-k} or \fi
{a similar argument as that for \eqref{eq:sub-linear-4}}. Moreover, Lemma~\ref{lem:subdiff} implies  that there exists some constant $\gamma_0>0$ such that
\[
{\|W^{(i)}_{[p+1]}\|_F}\leq \gamma_0 \| U_{[p]}- U_{[p+1]}\|_F,
\]
where
\[
W^{(i)}_{[p+1]}:=V^{(i)}\Lambda -V^{(i)}_{[p+1]}\Lambda^{(i+1)}_{[p+1]}.
\]

As a consequence, we may derive from \eqref{eq:relative-2-ob} that
\begin{align}
&\quad\ \frac{1}{2}\|\nabla_i f(U_{[p+1]})-U^{(i)}_{[p+1]}\diag_2\big((\nabla_i f(U_{[p+1]}))^\tp U^{(i)}_{[p+1]}\big)\|_F\nonumber\\
&=\|V^{(i)}\Lambda -U^{(i)}_{[p+1]}\diag_2\big((V^{(i)}\Lambda)^\tp U^{(i)}_{[p+1]}\big)\|_F\nonumber\\
&=\|W^{(i)}_{[p+1]}+V^{(i)}_{[p+1]}\Lambda^{(i+1)}_{[p+1]}-U^{(i)}_{[p+1]}\diag_2\big((V^{(i)}\Lambda)^\tp U^{(i)}_{[p+1]}\big)\|_F\nonumber\\
&\leq \|W^{(i)}_{[p+1]}\|_F+\|V^{(i)}_{[p+1]}\Lambda^{(i+1)}_{[p+1]}-U^{(i)}_{[p+1]}\diag_2\big((V^{(i)}\Lambda)^\tp U^{(i)}_{[p+1]}\big)\|_F\nonumber\\
&\leq (\gamma_0+\xi)\| U_{[p]}- U_{[p+1]}\|_F,\label{eq:relative-1-ob}
\end{align}
which concludes the proof {with $\gamma_2:=2(\gamma_0+\xi)$}.
\end{proof}

\begin{lemma}[\L ojasiwicz's Inequality]\label{lem:gradient}
If $( U^*,\mathcal{D}^*)$ is a nondegenerate KKT point of problem \eqref{eq:sota}, then there exist $\mu,\eta>0$ such that
\begin{equation}\label{eq:gradient-ineq}
\sum_{i=1}^s\|\nabla_i f( U)-U^{(i)}\nabla_i (f( U))^\tp U^{(i)}\|_F^2+
\sum_{i=s+1}^k\|\nabla_i f( U)-U^{(i)}{\diag_2}(\nabla_i (f( U))^\tp U^{(i)})\|_F^2\geq \mu |f( U)-f( U^*)|
\end{equation}
for any {$ U\in\V(r,n_1)\times\dots\times\V(r,n_s)\times\B(r,n_{s+1})\times\dots\times\B(r,n_k)$ satisfying $\| U- U^*\|_F\leq\eta$}. \if Here $f$ is the objective function of problem \eqref{eq:sota-max} defined by \eqref{eq:objective}. \fi
\end{lemma}

\begin{proof}
Let $\delta>0$ be the radius of the neighborhood given by Proposition~\ref{prop:lojasiewicz}. For a given $ U\in\V(r,n_1)\times\dots\times\V(r,n_s)\times\B(r,n_{s+1})\times\dots\times\B(r,n_k)$, we let $\mathcal{D}$ be the diagonal tensor uniquely determined by the relation:
\[
\operatorname{Diag}_k(\mathcal{D}) = \operatorname{Diag}_k\big(((U^{(1)})^\tp ,\dots,(U^{(k)})^\tp )\cdot\mathcal A\big).
\]
For a KKT point $( U^*,\mathcal{D}^*)$ of \eqref{eq:sota}, by Lemma~\ref{lem:kkt-equiv}, we have  
\[
\operatorname{Diag}_k(\mathcal{D}^*)=\operatorname{Diag}_k\big(((U^{(*,1)})^\tp ,\dots,(U^{(*,k)})^\tp )\cdot\mathcal A\big),
\]
where $ U^*=(U^{(*,1)},\dots,U^{(*,k)})$. It is clear that there exists some $\eta>0$ such that
\begin{equation}\label{eq:upsilon}
\|( U,\mathcal{D})-( U^*,\mathcal{D}^*)\|_F\leq \delta
\end{equation}
whenever $\| U- U^*\|_F\leq \eta$.

We denote $\mathbf x^* \coloneqq \operatorname{Diag}_k(\mathcal{D}^*) $ and $\mathbf x \coloneqq \operatorname{Diag}_k(\mathcal{D})$. Since $(U^\ast,\mathcal{D}^\ast)$ is nondegenerate, Propositions~\ref{prop:critical-equivalence} and \ref{prop:lojasiewicz} imply
the existence of some $\mu_0>0$ such that
\[
\|\operatorname{grad}(g)( U,\mathbf x)\|^2\geq \mu_0 |g( U,\mathbf x)-g( U^*,\mathbf x^*)|,
\]
if $\| U- U^*\|_F\leq \eta$. We observe that
\[
|g( U,\mathbf x)-g( U^*,\mathbf x^*)|= {\frac{1}{2}} |f( U)-f( U^*)|.
\]
In fact, we have $f( U)=\|\mathbf x\|^2$ and
\begin{align*}
g( U,\mathbf x)&=\frac{1}{2}\|\mathcal A-(U^{(1)},\dots,U^{(k)})\cdot\operatorname{diag}_k(\mathbf x)\|^2\\
&=\frac{1}{2}\|\mathcal A\|^2-\langle \mathcal A,(U^{(1)},\dots,U^{(k)})\cdot\operatorname{diag}_k(\mathbf x)\rangle+\frac{1}{2}\|\mathbf x\|^2\\
&=\frac{1}{2}\|\mathcal A\|^2-\langle\operatorname{Diag}_k{\big(((U^{(1)})^\tp ,\dots,(U^{(k)})^\tp )}\cdot\mathcal A\big),\mathbf x\rangle+\frac{1}{2}\|\mathbf x\|^2\\
&=\frac{1}{2}(\|\mathcal A\|^2-\|\mathbf x\|^2).
\end{align*}
By {\eqref{eq:gradient-x}} and the definition of $\mathbf x$, we also have
\[
\operatorname{grad}_{\mathbf x}g( U,\mathbf x)=\mathbf x-\operatorname{Diag}_k\big(((U^{(1)})^\tp ,\dots,(U^{(k)})^\tp )\cdot\mathcal A\big)=0.
\]
Since $\nabla_i f( U) =2 V^{(i)}\Gamma, 1 \le i \le k$ and
\begin{align*}
\operatorname{grad}_{U^{(i)}}g( U,\mathbf x) &=-(I-\frac{1}{2}U^{(i)}(U^{(i)})^\tp )(V^{(i)}\Gamma-U^{(i)}(V^{(i)}\Gamma)^\tp  U^{(i)}),\quad 1\le i \le s, \\
\operatorname{grad}_{U^{(i)}}g( U,\mathbf x) &=-\Big(V^{(i)}\Gamma-U^{(i)}{\diag_2}\big((U^{(i)})^\tp V^{(i)}\Gamma\big)\Big),\quad  s+1 \le i \le k,
\end{align*}
where $\Gamma={\operatorname{diag}_2(\mathbf x)}$, {the assertion will follow} if we can show that
\[
\|I-\frac{1}{2}U^{(i)}(U^{(i)})^\tp \|_F\leq \mu_1
\]
is uniformly bounded by some $\mu_1>0$ for $U$ such that $\| U- U^*\|_F\leq \eta$. But this is obviously true by continuity and the proof is complete.
\end{proof}

\begin{lemma}[R-Linear Convergence Rate]\label{lem:linear}
Let $\{ U_{[p]}\}$ be {a sequence} generated by Algorithm~\ref{algo} for a given nonzero tensor $\mathcal A\in\mathbb R^{n_1}\otimes\dots\otimes\mathbb R^{n_k}$.
If $\{ U_{[p]}\}$ converges to a nondegenerate KKT point $ U^*$ of \eqref{eq:sota-max}, then it converges $R$-linearly.
\end{lemma}
\begin{proof}
%
For a sufficiently large $p$, Lemma~\ref{lem:gradient} implies that
\small
\[
\sum_{i=1}^s\|\nabla_i f( U_{[p]})-U^{(i)}_{[p]}\nabla_i (f( U_{[p]}))^\tp U^{(i)}_{[p]}\|_F^2+
\sum_{i=s+1}^k\|\nabla_i f( U_{[p]})-U^{(i)}_{[p]}\diag_2(\nabla_i (f( U_{[p]}))^\tp U^{(i)}_{[p]})\|_F^2\geq \mu |f( U_{[p]})-f( U^*)|
\]
\normalsize
On the other hand, by Lemma~\ref{lem:gradient-diff}, we have
\[
\sum_{i=1}^s\|\nabla_i f( U_{[p]})-U^{(i)}_{[p]}\nabla_i (f( U_{[p]}))^\tp U^{(i)}_{[p]}\|_F^2\leq s\gamma^2_1\| U_{[p]}- U_{[p-1]}\|_F^2,
\]
and by Lemma~\ref{lem:gradient-diff-ob}, we have
\[
\sum_{i=s+1}^k\|\nabla_i f( U_{[p]})-U^{(i)}_{[p]}{\diag_2}\big((\nabla_i f( U_{[p]}))^\tp U^{(i)}_{[p]}\big)\|_F^2 \leq (k-s)\gamma_2^2 \| U_{[p]}- U_{[p-1]}\|_F^2.
\]

We denote $\nu:=\min\{\epsilon,2\kappa^2\}$ and $\gamma:=\max\{\gamma_1,\gamma_2\}$. We observe that
\[
f(U_{[p]})-f(U_{[p-1]})\geq \frac{{\nu}}{2}\|U_{[p]}-U_{[p-1]}\|_F^2\geq \frac{\mu{\nu}}{2k\gamma^2}(f(U^*)-f(U_{[p]})),
\]
where the first inequality follows from Proposition~\ref{prop:subfficient}, and the second follows from the preceding two inequalities together with Proposition~\ref{prop:monotone}.
Hence for a sufficiently large $p$, we have
\begin{equation}\label{eq:linear-obj}
f( U^*)-f( U_{[p]})\leq \frac{2k\gamma^2}{2k\gamma^2+\mu {\nu}}\big(f( U^*)-f( U_{[p-1]})\big),
\end{equation}
which establishes the local $Q$-linear convergence of the sequence $\lbrace f({U}_{\left[ p \right]}) \rbrace$. Consequently, using the fact that $\lbrace f({U}_{\left[ p \right]}) \rbrace$ is non-decreasing and \eqref{eq:linear-obj}, we may derive
\begin{align*}
\| U_{[s]}- U_{[s-1]}\|_F&\leq\sqrt{\frac{2}{{\nu}}}\sqrt{f( U_{[s]})-f( U_{[s-1]})}\\
&\leq \sqrt{\frac{2}{{\nu}}}\sqrt{f( U^*)-f( U_{[s-1]})} \\
&\leq \sqrt{\frac{2}{{\nu}}}\Bigg[\sqrt{\frac{2k\gamma^2}{2k\gamma^2+\mu{\nu}}}\Bigg]^{s-1}\sqrt{f( U^*)-f( U_{[0]})},
\end{align*}
which implies that
\[
\sum_{s=p}^\infty\| U_{[s]}- U_{[s-1]}\|_F<\infty,
\]
for any sufficiently large positive integer $p$.
Since $ U_{[s]}\rightarrow  U^*$ as $s \to \infty$, we have
\[
\| U_{[p]}- U^*\|_F\leq \sum_{s=p}^\infty\| U_{[s+1]}- U_{[s]}\|_F.
\]
{Hence, we obtain}
\[
\| U_{[p]}- U^*\|_F\leq\sqrt{\frac{2}{{\nu}}}\sqrt{f( U^*)-f( U_{[0]})}\frac{1}{1-\sqrt{\frac{2k\gamma^2}{2k\gamma^2+\mu{\nu}}}}\Big[\sqrt{\frac{2k\gamma^2}{2k\gamma^2+\mu{\nu}}}\Big]^{p},
\]
which is the claimed $R$-linear convergence of the sequence $\{ U_{[p]}\}$ {and this completes the proof.}
\end{proof}

We are now in the position to prove the R-linear convergence rate of Algorithm~\ref{algo}.
\begin{theorem}[Generic Linear Convergence]\label{thm:generic}
Let $s\geq 1$ and let $\{ U_{[p]}\}$ be a sequence generated by Algorithm~\ref{algo} for a generic tensor $\mathcal A\in\mathbb R^{n_1}\otimes\dots\otimes\mathbb R^{n_k}$. If $s = 1$, assume in addition that \eqref{lem:defectivity s=1:eq01} and \eqref{lem:defectivity s=1:eq02} hold. The sequence $\{ U_{[p]}\}$ converges $R$-linearly to a KKT point of \eqref{eq:sota-max}.
\end{theorem}
\begin{proof}
By Theorem~\ref{thm:global}, the sequence $\{ U_{[p]}\}$ converges globally to $ U^*$. Moreover, since $\mathcal{A}$ is generic, Propositions~\ref{prop:nondegnerate}, \ref{prop:critical-equivalence}, \ref{prop:KKT location s=3}, \ref{prop:KKT location s=2} and \ref{prop:KKT location s=1} imply that $ U^*$ together with
\[
\mathbf x^* \coloneqq \operatorname{Diag}_k\big(((U^{(*,1)})^\tp ,\dots,(U^{(*,k)})^\tp )\cdot\mathcal A\big)
\]
is a nondegenerate critical point of the function $g$ on $W_{\mathbf n,t,\ast}$ for some $0\le t \le r$. According to Lemma~\ref{lem:linear}, we conclude that $\{ U_{[p]}\}$ converges $R$-linearly to $ U^*$.
\end{proof}
Note that by Corollary~\ref{cor:KKT location s=1}, the requirement for \eqref{lem:defectivity s=1:eq01} and \eqref{lem:defectivity s=1:eq02} can be removed if $s=1, k\ge 3$ and $n_1 = \cdots = n_k$. Hence we obtain the following
\begin{theorem}[{Square Tensors}]\label{thm:generic-square}
Assume that $k\ge 3$ and $n_1 = \cdots =n_k \ge 2$. For a generic $\mathcal{A}$, the sequence $\{ U_{[p]}\}$ generated by Algorithm~\ref{algo} converges $R$-linearly to a KKT point of \eqref{eq:sota-max}.
\end{theorem}

{We remark that sublinear convergence rate is proved in Theorem~\ref{thm:sublinear} for the general case, while in Theorem~\ref{thm:generic-square} linear convergence rate is given for a generic case. This can be interpreted via the fact that the local branches of the solution mapping for the best low rank partially orthogonal tensor approximation problem have local Lipschitz continuity whenever the nondegeneracy of the converged KKT point holds \cite{DR-09}. However, the nondegeneracy only holds generically. Equivalently, for general polynomial systems, only local H\"olderian error bounds can hold \cite{LMP-15}. The explicit examples given in \cite{EHK-15,EK-15} present a witness for this phenomenon when $r=1$.}

\section{Conclusions}\label{sec:conclusion}

In this paper, we study the low rank partially orthogonal tensor approximation (LRPOTA) problem. To numerically solve the problem, we propose {iAPD-ALS} algorithm (cf. Algorithm~\ref{algo})  which is based on the block coordinate descent method. We conduct a thorough analysis of the convergent behaviour of our {algorithm}. To achieve this goal, we carefully investigate geometric properties of partially orthogonal tensors related to {iAPD-ALS} algorithm. With tools from differential and algebraic geometry, we successfully establish the sublinear global convergence in general and more importantly, the generic R-linear global convergence.

On the one hand, the geometric analysis carried out in this paper is on the feasible set of the approximation problem. Thus, it is not only applicable to the convergence analysis of iAPD-ALS algorithm, but can also  be applied to other algorithms. \if It also supplies a rich class of instances for the study of tensor decomposition and low rank tensor approximation.\fi On the other hand, the nondegeneracy of KKT points for a generic tensor obtained in this paper implies that the LRPOTA problem satisfies the strict saddle point condition generically. Thus, algorithms with the property of avoiding saddle points discussed in \cite{LPPSJR19,LSJR-16} would be investigated in the future and this would shed some light on the study of global optimizers for the LRPOTA problem in certain cases.




\subsection*{Acknowledgement}
This work is partially supported by National Science Foundation of China (Grant No. 11771328). The first author is also partially supported by National Science Foundation of China (Grant No.~11801548 and Grant No.~11688101), National Key R\&D Program of China (Grant No.~2018YFA0306702) and the Recruitment Program of Global Experts of China. The second author is also partially supported by the National Science Foundation of China (Grant No. 12171128) and the Natural Science Foundation of Zhejiang Province, China (Grant No. LD19A010002 and Grant No. LY22A010022).

\bibliographystyle{abbrv}
\bibliography{odtapr}

\end{document}